\documentclass[
  11pt,
  paper=a4,
  parskip=half,
  numbers=noenddot,
  headings=normal
]{scrartcl}
\usepackage{amsthm,amsmath,amsfonts,amssymb}
\usepackage[margin=2.1cm]{geometry}

\usepackage[T1]{fontenc}
\usepackage[utf8]{inputenc}
\usepackage[english]{babel}
\usepackage{newtxtext,newtxmath}
\usepackage{microtype}

\usepackage{mathtools}
\usepackage{bbm}

\usepackage[numbers,sort&compress]{natbib}
\usepackage[
  colorlinks=true,
  citecolor=blue,
  linkcolor=blue,
  urlcolor=blue
]{hyperref}

\usepackage{authblk}

\usepackage{pgf}
\usepackage{pgfplots}
\usepgfplotslibrary{polar,patchplots}
\pgfplotsset{compat=1.18}

\usepackage{tikz}
\usepackage{tikz-cd}
\usepackage{tikz-3dplot}
\definecolor{coarseteal}{RGB}{0,112,108}

\usetikzlibrary{
  arrows,
  arrows.meta,
  bending,
  automata,
  backgrounds,
  calc,
  chains,
  decorations.markings,
  decorations.pathmorphing,
  decorations.pathreplacing,
  fit,
  graphs,
  hobby,
  math,
  matrix,
  patterns,
  positioning,
  quotes,
  shapes,
  trees,
  babel
}

\usepackage{todonotes}

\newcommand{\bbZ}{\mathbb{Z}}

\newcommand{\Prob}{\mathbb{P}}

\newcommand{\E}{\mathbb{E}}

\DeclareMathOperator{\Cay}{Cay}

\newtheorem{theorem}{Theorem}[section]
\newtheorem{proposition}[theorem]{Proposition}
\newtheorem{lemma}[theorem]{Lemma}

\theoremstyle{definition}
\newtheorem{definition}[theorem]{Definition}

\theoremstyle{remark}
\newtheorem{remark}[theorem]{Remark}

\newcommand{\keywords}[1]{%
  \par\medskip
  \noindent\textbf{Keywords: }#1.
}

\newcommand{\subjclass}[2][]{%
  \par\smallskip
  \noindent\textbf{Mathematics Subject Classification (2020): }
  #2%
  \if\relax\detokenize{#1}\relax
  \else
    \space\textit{(secondary: #1)}
  \fi.
}

\title{%
  Extinction, Survival and Fluctuations for the Spatial
  Maki--Thompson Model on Infinite Graphs
}

\author[1]{Luciano Henrique Lacerda de Ara\'ujo}
\author[1]{Daniel Miranda Machado}
\author[1]{Cristian Favio Coletti}
\author[2]{Denis Araujo Luiz}

\affil[1]{%
  Universidade Federal do ABC (UFABC), Brazil\\
  \href{mailto:henrique.luciano@aluno.ufabc.edu.br}
       {\texttt{henrique.luciano@aluno.ufabc.edu.br}}\\
  \href{mailto:daniel.miranda@ufabc.edu.br}
       {\texttt{daniel.miranda@ufabc.edu.br}}\\
  \href{mailto:cristian.coletti@ufabc.edu.br}
       {\texttt{cristian.coletti@ufabc.edu.br}}
}

\affil[2]{%
  Universidade Estadual de Campinas (UNICAMP), Brazil\\
  \href{mailto:denis.luiz@unicamp.br}
       {\texttt{denis.luiz@unicamp.br}}
}

\date{}

\usepackage{scrlayer-scrpage}
\clearpairofpagestyles
\ihead{Spatial Maki--Thompson Model}
\ohead{Ara\'ujo, Machado, Coletti and Luiz}
\setkomafont{pageheadfoot}{\small}

\begin{document}

\maketitle

\begin{abstract}
 We study the spatial Maki--Thompson rumor model on infinite, connected
graphs of bounded degree. Spreaders transmit the rumor to ignorant
neighbors but become stiflers upon contacting non-ignorant neighbors.

We prove extinction on Cayley graphs of linear growth, for every
\(\lambda,\alpha>0\) and every initial configuration with finitely many
non-ignorant vertices, and establish an explicit extinction criterion on
arbitrary bounded-degree graphs for processes started from finitely many
spreaders. On Cayley graphs of superlinear growth, we prove survival from a
single spreader whenever the ratio of the stifling to the transmission rate
lies below an explicit threshold depending only on the maximum degree. On Cayley graphs of polynomial
growth of degree \(D\ge2\), we further show that the range has positive
lower density with positive probability. Under a stronger condition, macroscopic annuli contain a surface-order
number of simultaneously active spreaders for a total duration bounded
uniformly away from zero.  When \(\alpha>0\),
every finite region eventually contains no spreaders, so global survival
forces the rumor to move continually into new regions.

Under either subcriticality or a sufficiently small stifling rate, we
prove central limit theorems for the final stifler density and the total
spreader occupation time, together with a functional central limit theorem
for the empirical survival function. These results follow from a central
limit theorem for stationary stabilizing functionals of i.i.d.\ fields on
polynomial-growth Cayley graphs; the functionals may depend on the field
outside the observation set.

\end{abstract}

\keywords{%
    Maki--Thompson rumor model;
    Bernoulli site percolation;
    interacting particle systems;
    Cayley graphs;
    stabilization;
    central limit theorem
}

\subjclass[20F69; 60F05; 60F17; 82C22]{60K35}

\section{Introduction}

The Maki--Thompson model is a classical stochastic model of rumor propagation \cite{maki1973mathematical}, a phenomenon
with important consequences for social and political communication \cite{vosoughi2018spread}. Unlike in standard
epidemic models, spreaders do not become inactive at an intrinsic recovery rate. Instead, a spreader becomes a stifler
only after contacting an individual who already knows the rumor. The original model assumes homogeneous mixing and may
therefore be viewed as a process on a complete graph, where its macroscopic behavior can be described by deterministic
differential equations. Beyond its original motivation, the model provides a natural example of an irreversible
interacting particle system in which transmission and suppression arise from the same local interactions.

Under homogeneous mixing, the Maki--Thompson model and its Daley--Kendall predecessor are now well understood: the
final proportion of ignorants converges to an explicitly characterized constant and satisfies a central limit theorem,
and similar results are available for several extensions of the stifling mechanism
\cite{daley1965stochastic,sudbury1985,watson1988,lebensztayn2011limit}.

Real-world interactions, however, are typically structured and local. Replacing the complete graph by a spatial graph
substantially changes the dynamics: non-ignorant vertices mark the path taken by the rumor, but may also act as local
barriers to its further propagation. We study an irreversible spatial version of the Maki--Thompson model on infinite,
connected graphs of bounded degree, with particular emphasis on Cayley graphs. Each vertex \(x\) is in one of three
states, \(\eta(x)\in\{0,1,2\}\), corresponding, respectively, to an ignorant individual, a spreader, or a stifler. An
ignorant vertex becomes a spreader at rate \(\lambda>0\) times its number of neighboring spreaders, while a spreader
becomes a stifler at rate \(\alpha\ge0\) times its number of non-ignorant neighbors.

A basic difficulty is that the process is not attractive. Increasing the set of spreaders may create additional
opportunities for transmission, but may also increase the local stifling rates. Consequently, the usual monotone
couplings available for the contact process and other attractive interacting particle systems
\cite{liggett1985,liggett1999} do not apply directly. Our analysis instead exploits the irreversible nature of the
dynamics together with the geometry of the underlying graph.

Spatial versions of rumor models have been studied only more recently. Coletti, Rodr\'iguez, and Schinazi
\cite{coletti2011spatial} analyzed a related process on \(\mathbb Z^d\) in which non-ignorant individuals may forget
the rumor and return to the ignorant state. They obtained sufficient conditions for extinction and survival through
comparisons with oriented percolation and the contact process. On infinite Cayley trees, Junior, Rodr\'iguez, and
Speroto \cite{junior2020trees} established extinction and survival results for the classical irreversible model by
comparison with a branching process, while Agliari et al.\ \cite{agliari2017} identified a localization--propagation
transition on small-world networks.

Our first group of results describes how the geometry of the underlying graph affects extinction and survival. On
Cayley graphs of linear growth, the rumor dies out globally almost surely from every finite non-ignorant set, for all
\(\lambda,\alpha>0\); the proof relies on uniformly bounded cutsets. On arbitrary bounded-degree graphs,
branching-process domination yields an explicit extinction criterion in terms of the mean of a dominating offspring
distribution. By contrast, on every Cayley graph of superlinear growth, we construct an i.i.d.\ field of good vertices
and compare it with supercritical Bernoulli site percolation. This proves survival from a single spreader whenever
\(\alpha/\lambda<c_0/(\Delta(1+\log\Delta))\), for a universal constant \(c_0>0\). Thus the relevant geometric
distinction is between linear and superlinear growth: the former has no survival phase, whereas the latter does when
$\alpha/\lambda$ is sufficiently small.

The percolation comparison yields more than survival. On Cayley graphs of polynomial growth of degree \(D\ge2\), the
set of vertices ever reached by the rumor has positive lower density with positive probability. Under a stronger
high-density condition, we also prove that, for every \(\varepsilon\in(0,1)\), there exists \(\delta_\varepsilon>0\)
such that the annuli \(B(o,n)\setminus B(o,(1-\varepsilon)n)\) contain at least \(\delta_\varepsilon n^{D-1}\)
simultaneously active spreaders on a set of times whose Lebesgue measure is bounded below uniformly in \(n\).

When \(\alpha>0\), global survival does not imply persistence in any fixed region. We prove that, from an arbitrary
deterministic initial configuration on a bounded-degree graph, every finite set almost surely contains no spreaders
after some finite random time. Thus, whenever the rumor survives globally, it can do so only by continually moving into
new regions of the graph.

The second part of the paper concerns spatial fluctuations under i.i.d.\ Bernoulli initial states. We study
fluctuations at scale \(\lvert B_n\rvert^{1/2}\) for observables that may depend on the entire trajectory and are
therefore nonlocal functions of the graphical field.

We first establish a central limit theorem for stationary stabilizing functionals of i.i.d.\ fields indexed by
polynomial-growth Cayley graphs, along exhaustive two-sided F{\o}lner sequences. The argument combines a
martingale-difference decomposition \cite{mcleish1974dependent} with stabilization methods
\cite{penrose2001,de2026topology}. The theorem allows the functional to depend on the field outside the observation
set, provided that it admits local conditional approximations whose error is negligible on the variance scale.

To apply this framework, we analyze the discrepancy caused by resampling a single coordinate of the graphical field.
Under either subcriticality or a sufficiently small stifling rate, we obtain sufficiently strong spatial and temporal
localization estimates for this discrepancy. These estimates yield central limit theorems in growing balls for the
final stifler density and the total spreader occupation time. We also prove a functional central limit theorem for the
empirical survival function: after centering and normalization, it converges in Skorokhod space to a Gaussian process
with continuous sample paths. The limiting covariance is expressed in terms of stabilized resampling differences, and
the corresponding variances are shown to be strictly positive.

\textbf{Organization of the paper.} Section~\ref{sec:model-main-results}
defines the model and states the main results. Section~\ref{sec:geometric-probabilistic-tools}
develops the geometric and probabilistic tools: Subsection~\ref{subsec:abstract-clt}
contains the abstract central limit theorem, while Subsection~\ref{subsec:geometric-lemma}
characterizes linear growth through uniformly bounded cutsets. Section~\ref{sec:proofs-extinction-survival}
contains the proofs of the extinction, survival, active-propagation, and local-extinction
results. Finally, Section~\ref{sec:application-mt-observables} proves the central
and functional limit theorems for the Maki--Thompson observables.

\section{Model and Main Results}
\label{sec:model-main-results}

\subsection{The Spatial Maki--Thompson Model}

Throughout the paper, $G=(V,E)$ denotes an infinite, connected graph of bounded degree. We write $x\sim y$ whenever
$\{x,y\}\in E$, and denote the graph distance by $d_{G}(\cdot,\cdot)$. For $x\in V$ and $r\ge0$, let
\[
    B_{G}(x,r):=\{y\in V:d_{G}(x,y)\le r\}.
\]

A graph \(G\) is called \emph{transitive} if, for every \(x,y\in V\), there exists \(\varphi\in\operatorname{Aut}(G)\)
such that \(\varphi(x)=y\). We say that $G$ has \emph{polynomial growth} if there exist constants $C>0$ and $q\ge1$
such that
\[
    |B_{G}(x,r)|\le Cr^{q}\qquad\text{for every $x\in V$ and $r\ge1$}.
\]

Let $\Gamma$ be a finitely generated group and let $S\subset\Gamma\setminus\{ e\}$ be a finite symmetric generating
set. The \emph{Cayley graph} $G=\operatorname{Cay} (\Gamma, S)$ has vertex set $\Gamma$ and edge set
\[
    E = \bigl\{ \{g,gs\}:g\in\Gamma,\ s\in S \bigr\}.
\]
If \(G\) has polynomial growth, then, by Gromov's theorem, \(\Gamma\) is virtually nilpotent, meaning that it
  contains a nilpotent subgroup of finite index. The Bass--Guivarc'h formula \cite{Bass1972,Guivarch1973} then implies
  that there exist an integer \(D\ge 1\) and constants \(0<c\le C<\infty\) such that, writing \(o\) for the identity
  element of \(\Gamma\),
\[
    c\,r^{D}\le |B_G(o,r)|\le C\,r^{D}, \qquad r\ge 1.
\]
We call $D$ the \emph{degree of polynomial growth} of $G$. The case $D=1$ is referred to as linear growth.

Typical examples include the integer lattices $\mathbb{Z}^{d}$, finite extensions such as $\mathbb{Z}^{d}\times F$,
with $F$ finite, and the discrete Heisenberg group $H_{3}(\mathbb{Z})$.

For fixed $\lambda>0$ and $\alpha\ge0$, the process is defined through a Harris graphical representation. For each
ordered edge $(x,y)$ with $x\sim y$, let $N_{\mathrm{inf}}^{x,y}$ and $N_{\mathrm{stif}}^{x,y}$ be independent Poisson
processes with rates $\lambda$ and $\alpha$, respectively, and assume that all these processes are mutually
independent.

At each arrival time of $N_{\mathrm{inf}}^{x,y}$, if $\eta(x)=1$ and $\eta(y) =0$, then $y$ becomes a spreader.
Otherwise, the configuration is unchanged. At each arrival time of $N_{\mathrm{stif}}^{x,y}$, if $\eta(x)=1$ and
$\eta(y)\in\{1,2\}$, then $x$ becomes a stifler. Otherwise, the configuration is unchanged.

Equivalently, for $x\in V$, define
\[
    n_{1}(x,\eta) := \#\{y\sim x:\eta(y)=1\}, \qquad n_{2}(x,\eta) := \#\{y\sim
    x :\eta(y)=2\}.
\]
The local transition rates are
\[
    0\to1 \quad\text{at rate}\quad \lambda n_{1}(x,\eta),
\]
and
\[
    1\to2 \quad\text{at rate}\quad \alpha\bigl(n_{1}(x,\eta)+n_{2}(x,\eta)\bigr
    ).
\]
State $2$ is absorbing, i.e., once a vertex enters state \(2\), it remains a stifler forever.

Since the interaction rates are uniformly bounded on graphs of bounded degree, the process is well defined by the
standard Harris graphical construction; see, for instance, \cite{liggett1985,liggett1999}.

\subsection{Extinction and Survival Regimes}

Let \(S_t:=\{x\in V:\eta_t(x)=1\}\) be the set of spreaders at time \(t\), and define
\[
    \tau:=\inf\{t\ge0:S_t=\varnothing\}.
\]
We call \(\{\tau=\infty\}\) the event of \emph{global survival} and \(\{\tau<\infty\}\) the event of \emph{global
  extinction}. We say that \emph{local extinction} occurs if, for every finite \(K\subset V\), there exists a finite,
  possibly random, time \(T_K\) such that
\[
    S_t\cap K=\varnothing
    \qquad\text{for all }t\ge T_K.
\]

We also write \( \mathcal A_\infty := \{x\in V:\eta_t(x)=1\text{ for some }t\ge0\} \) for the range of the rumor.

We first consider processes whose initial non-ignorant set is finite. On Cayley graphs of linear growth, uniformly
bounded cutsets prevent the rumor from propagating indefinitely.

\begin{theorem}[Extinction on Cayley graphs of linear growth]
    \label{thm:linear-growth-extinction} Let $G=\operatorname{Cay}(\Gamma,S)$
    be the Cayley graph of an infinite finitely generated group of linear growth.
    Consider the spatial Maki--Thompson rumor model on $G$ with transmission
    rate $\lambda>0$ and stifling rate $\alpha>0$, started from an initial
    configuration with only finitely many non-ignorant vertices. Then the rumor
    dies out globally almost surely.
\end{theorem}

On general bounded-degree graphs, the transmission genealogy can be dominated by a branching process. This yields the
following subcritical regime. For $\Delta\ge2$ and $\rho>0$, set
\[
    m_{\Delta}(\rho) := \sum_{k=1}^{\Delta-1}\prod_{j=0}^{k-1}\frac{\Delta-1-j}{\Delta-1-j+(j+1)\rho}
    .
\]
As shown in the proof below, $m_{\Delta}(\rho)$ is the mean of the offspring distribution used in the
  branching-process comparison.

\begin{theorem}[Subcritical regime on bounded-degree graphs]
    \label{thm:MT-positive-subcritical} Let $G=(V,E)$ be an infinite,
    connected graph with
    \(
    \Delta:=\sup_{x\in V}\deg(x)<\infty.
    \)
    Consider the spatial Maki--Thompson rumor model on $G$, with
    transmission rate $\lambda>0$ and stifling rate $\alpha>0$, started from
    a finite set $S_{0}\subset V$ of spreaders, with all other vertices
    initially ignorant. If
    \(
    m_{\Delta}(\alpha/\lambda)<1,
    \)
    then global extinction occurs almost surely.
\end{theorem}
On Cayley graphs of superlinear growth, a survival phase appears. The key
input is the universal gap below one for the critical probability of
Bernoulli site percolation on this class of graphs
\cite{PanagiotisSevero}.

For \(q\in[0,1]\), let \(\mathbb P_q\) denote Bernoulli site percolation on \(G\) with parameter \(q\), and set
\[
    \theta_G(q)
    :=
    \mathbb P_q(o\text{ belongs to an infinite open cluster}).
\]
For \(\Delta\ge2\) and \(\rho\ge0\), set
\[
    q_\Delta(\rho)
    :=
    \prod_{j=1}^{\Delta}\frac{j}{j+\Delta\rho}.
\]

\begin{theorem}[Survival with positive-density range]
    \label{thm:MT-survival} Let \(G=(V,E)\) be a Cayley graph of superlinear growth and degree
    \(\Delta\), rooted at \(o\). For \(\lambda,\alpha>0\), consider the
    process started from a single spreader at \(o\), and put
    \(\rho:=\alpha/\lambda\). If
    \[
        q_\Delta(\rho)>p_c^{\mathrm{site}}(G),
    \]
    then \( \mathbb P(\text{global survival}) \ge \theta_G\bigl(q_\Delta(\rho)\bigr). \) If, in addition, \(G\) has
    polynomial growth of degree \(D\ge2\), then
    \begin{equation}
        \mathbb P\left(
        \liminf_{n\to\infty}
        \frac{
                |\mathcal A_\infty\cap B_G(o,n)|
            }{
                |B_G(o,n)|
            }
        \ge
        \theta_G\bigl(q_\Delta(\rho)\bigr)
        \right)
        \ge
        \theta_G\bigl(q_\Delta(\rho)\bigr).
        \label{eq:main-positive-density}
    \end{equation}
    Moreover, there exists a universal constant \(c_0>0\) such that
    \begin{equation}
        \frac{\alpha}{\lambda}
        <
        \frac{c_0}{\Delta(1+\log\Delta)}
        \label{eq:main-uniform-survival-condition}
    \end{equation}
    implies \(q_\Delta(\rho)>p_c^{\mathrm{site}}(G)\), and hence the
    conclusions above hold.
\end{theorem}
The dependence on the rates only through the ratio
\(\rho=\alpha/\lambda\) follows from invariance under a common rescaling
of time. The uniformity over the class of graphs instead follows from
the percolation estimate of Panagiotis and Severo.

The previous result concerns the range of the rumor, rather than the number of spreaders present at a given time.
Heuristically, vertices well inside the reached region have already become stiflers, while active spreaders are
concentrated near the advancing boundary. Since a boundary at distance \(n\) has surface-order scale \(n^{D-1}\), the
next theorem makes this picture quantitative; see Figure~\ref{fig:sim1}.

\begin{figure}
    \centering
    \includegraphics[width=0.45\linewidth]{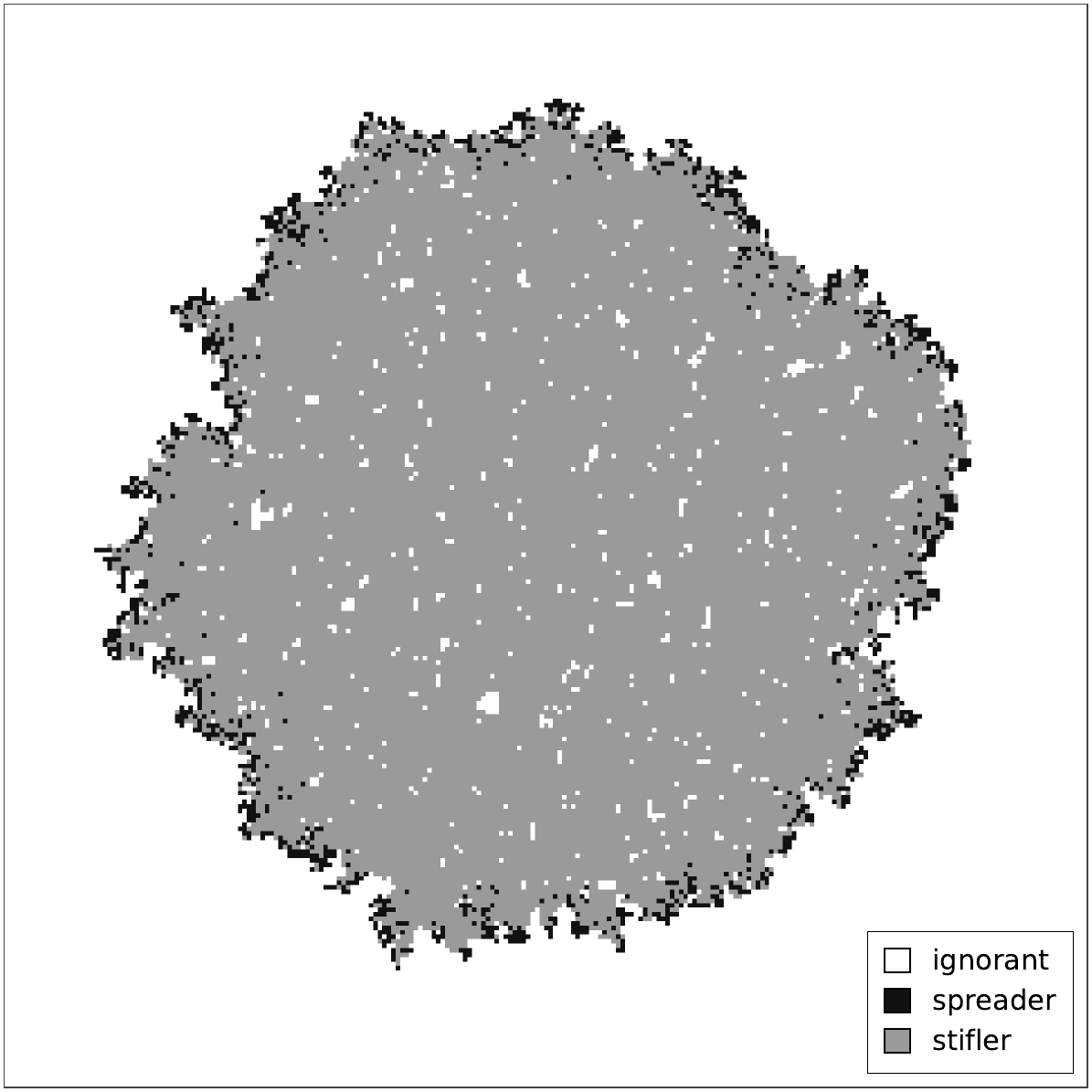}
    \caption{A sample realization of the spatial Maki--Thompson model on $\mathbb{Z}
            ^{2}$, started from a single spreader at the origin with $\lambda=1$ and
        $\alpha=0.25$. Spreaders are concentrated near the boundary of the reached
        region, while the interior is predominantly occupied by stiflers.}
    \label{fig:sim1}
\end{figure}

For a Cayley graph $G=\operatorname{Cay}(\Gamma,S)$ of polynomial growth of degree $D\ge2$, put $\Delta:=|S|$ and fix
$\Lambda=\Lambda(G)\ge1$ such that $G$ is $\Lambda$-simply connected; see
Proposition~\ref{prop:G-is-coarsely-simply-connected}. Set $N_{\Lambda}:=|B_{G}(o,\Lambda)|$ and
\begin{equation}
    \widetilde p_{G}:= \max\left\{ p_{c}^{\mathrm{site}}(G),\, 1-\frac{1}{N_{\Lambda}-1}
    \right\}. \label{eq:main-active-percolation-threshold}
\end{equation}
By transitivity, $N_{\Lambda}=\sup_{z\in V}|B_{G}(z,\Lambda)|$, and
$\widetilde p_{G}<1$.

Since $q_{\Delta}$ is continuous and strictly decreasing from $1$ to $0$, there exists a unique
$\rho_{\mathrm{act}}(G)>0$ such that $q_{\Delta}(\rho_{\mathrm{act}} (G))=\widetilde p_{G}$. For $\lambda>0$, define
$\alpha_{\mathrm{act}}(G,\lambda ) :=\lambda\rho_{\mathrm{act}}(G)$. Thus
\begin{equation}
    0\le\alpha<\alpha_{\mathrm{act}}(G,\lambda) \quad\Longleftrightarrow\quad
    q_{\Delta}(\alpha/\lambda)>\widetilde p_{G}. \label{eq:active-propagation-high-density}
\end{equation}
The value of $\alpha_{\mathrm{act}}(G,\lambda)$ depends on the fixed
admissible choice of $\Lambda$; this dependence is suppressed from the notation.

For $\varepsilon\in(0,1)$ and $n\ge1$, write
\[
    \mathcal{Q}_{n}^{\varepsilon}:= B_{G}(o,n)\setminus B_{G}\!\left(o,\lfloor
        (1-\varepsilon)n\rfloor\right).
\]

\begin{theorem}[Surface-order active propagation]
    \label{thm:surface-order-active-propagation}
    Let \(G=\operatorname{Cay}(\Gamma,S)\) have polynomial growth of degree
    \(D\ge2\), and suppose that the rumor starts from a single spreader at
    \(o\). Let \(\lambda>0\) and
    \(0\le\alpha<\alpha_{\mathrm{act}}(G,\lambda)\).
    Then there exists \(v_{+}<\infty\) such that, for every
    \(\varepsilon\in(0,1)\), there exist constants
    \(0<v_{-}(\varepsilon)<v_{+}\),
    \(\delta_{\mathrm{act}}(\varepsilon)>0\), and
    \(\ell_{\mathrm{act}}(\varepsilon)>0\)
    such that, with positive probability,
    \begin{equation}
        \operatorname{Leb}\left\{
        t\in[v_{-}(\varepsilon)n,v_{+}n]:
        \lvert S_t\cap\mathcal{Q}_n^\varepsilon\rvert
        \ge \delta_{\mathrm{act}}(\varepsilon)n^{D-1}
        \right\}
        \ge \ell_{\mathrm{act}}(\varepsilon)
        \label{eq:positive-duration-surface-activity}
    \end{equation}
    holds for all sufficiently large \(n\).
\end{theorem}

In particular,
\begin{equation}
    \mathbb{P}\left(
    \liminf_{n\to\infty}
    \frac{1}{n^{D-1}}
    \sup_{v_-(\varepsilon)n\le t\le v_+n}
    \lvert S_t\cap\mathcal{Q}_n^\varepsilon\rvert
    \ge \delta_{\mathrm{act}}(\varepsilon)
    \right)>0.
    \label{eq:surface-order-active-propagation}
\end{equation}

The explicit characterization \eqref{eq:active-propagation-high-density} shows in particular that
$\alpha_{\mathrm{act}}(G,\lambda)>0$. We expect the conclusion to remain valid under the weaker condition
$q_{\Delta}(\alpha/\lambda )>p_{c}^{\mathrm{site}}(G)$.

We next prove local extinction from arbitrary deterministic initial configurations. Thus, even when the rumor survives
globally, every finite region eventually contains no spreaders.

\begin{proposition}[Local extinction on bounded-degree graphs]
    \label{prop:local-extinction} Let $G=(V,E)$ be an infinite, connected
    graph of bounded degree, and consider the spatial Maki--Thompson rumor
    model on $G$ with transmission rate $\lambda>0$ and stifling rate
    $\alpha>0$, started from an arbitrary deterministic initial configuration
    \(
    \eta_{0}\in\{0,1,2\}^{V}.
    \)
    Then, for every finite set $K\subset V$, there exists an almost surely finite
    random time $\tau_{K}$ such that
    \[
        \eta_{t}(x)\ne1 \qquad \text{for every $x\in K$ and every
            $t\ge\tau_{K}$}.
    \]
\end{proposition}

\subsection{Fluctuation Results in Growing Balls}
\label{subsec:fluctuation-results}

Throughout this section, we fix $\lambda,\alpha>0$ and let $G=\operatorname{Cay}(\Gamma,S)$ be a Cayley graph of
polynomial growth. Set $\Delta:=|S|$, and let $o\in V$ denote the identity element. Assume that the initial states are
i.i.d., with
\[
    \mathbb{P}(\eta_{0}(x)=1)=\beta, \qquad \mathbb{P}(\eta_{0}(x)=0)=1-\beta
    , \qquad \mathbb{P}(\eta_{0}(x)=2)=0,
\]
for some $\beta\in(0,1)$.

Throughout this subsection, assume either that \( m_{\Delta}(\alpha/\lambda)<1, \) or that \(
0<\alpha<\alpha_{\mathrm{fl}}(G,\lambda,\beta), \) where the existence of $\alpha_{\mathrm{fl}}(G,\lambda,\beta)>0$ is
established in Proposition~\ref{prop:verification-mt-observables}.

All expectations below are taken with respect to both the initial configuration and the graphical construction. For
\(n\ge1\), let \(B_n:=B_G(o,n)\). The process evolves on the whole graph \(G\); the ball \(B_n\) is used only as an
observation window. Set
\[
    \mathcal A_n:=\mathcal{A}_\infty\cap B_n
\]
and define
\[
    \kappa_n:=\frac{|\mathcal A_n|}{|B_n|}.
\]
By Proposition~\ref{prop:local-extinction}, applied conditionally on the initial configuration, every vertex reached
  by the rumor eventually becomes a stifler almost surely. Since there are no initial stiflers, \(\kappa_n\) is both
  the density of the range and the final stifler density in \(B_n\).

By transitivity, equivariance of the graphical construction, and translation invariance of the initial law, the
distribution of $(\eta_t(x))_{t\ge0}$ does not depend on $x$.

Set \( \kappa := \mathbb{P}\bigl(\eta_{t}(o)=2\text{ for some }t\ge0\bigr). \) Then \( \mathbb{E}[\kappa_{n}]=\kappa.
\)

\begin{theorem}[Central limit theorem for the final stifler density]	\label{thm:clt-final-stifler-proportion-balls} There exists $\sigma_{\kappa}
        ^{2}>0$ such that
    \[
        |B_{n}|^{1/2}\bigl(\kappa_{n}-\kappa\bigr) \xrightarrow[n\to\infty]{d}
        \mathcal{N}(0,\sigma_{\kappa}^{2}).
    \]
\end{theorem}

We next consider the total spreader occupation time accumulated over \(B_n\):
\[
    \Xi_{n}:= \sum_{x\in B_n}\int_{0}^{\infty}\mathbf{1}_{\{\eta_t(x)=1\}}\,dt.
\]
Set
\[
    \xi := \mathbb{E}\!\left[ \int_{0}^{\infty}\mathbf{1}_{\{\eta_t(o)=1\}}\,
        dt \right].
\]
The time a vertex spends in the spreader state is stochastically dominated by $E_{\lambda}+E_{\alpha}$, where
  $E_{\lambda}$ and $E_{\alpha}$ are independent exponential random variables of rates $\lambda$ and $\alpha$,
  respectively. If a spreader has no non-ignorant neighbor, it creates one at total rate at least $\lambda$;
  thereafter, since non-ignorance is absorbing, it is stifled at rate at least $\alpha$. In particular, \(
  \xi\le\lambda^{-1}+\alpha^{-1}<\infty, \) and therefore \( \mathbb{E}[\Xi_{n}]=|B_{n}|\,\xi \) for every $n\ge1$.

\begin{theorem}[Central limit theorem for the total spreader occupation time] 	\label{thm:clt-spreader-occupation-balls} There exists $\sigma_{\Xi}^{2}>
        0$ such that
    \[
        |B_{n}|^{-1/2}\bigl(\Xi_{n}-|B_{n}|\xi\bigr) \xrightarrow[n\to\infty]
        {d}\mathcal{N}(0,\sigma_{\Xi}^{2}).
    \]
\end{theorem}

Finally, for each \(x\in V\), define the last time at which \(x\) is in the spreader state
\[
    L_{x}:=\sup\{t\ge0:\eta_{t}(x)=1\},
\]
with $L_{x}=0$ if $x$ never becomes a spreader. For fixed $t\ge0$, set
\[
    S_{n}(t) := \sum_{x\in B_n}\mathbf{1}_{\{L_x>t\}}, \qquad s_{n}(t):=\frac{S_{n}(t)}{|B_{n}|}
    .
\]
Thus $S_{n}(t)$ counts the vertices of $B_{n}$ that are spreaders at some time after $t$.

Define \( s(t):=\mathbb{P}(L_{o}>t), \qquad t\ge0. \) Then \( \mathbb{E}[s_{n}(t)]=s(t). \) Since a vertex that ever
becomes a spreader remains in that state for a positive time almost surely and eventually becomes a stifler almost
surely, we also have \( s(0)=\kappa. \)

\begin{theorem}[Functional central limit theorem for the empirical survival function]
    \label{thm:functional-clt-empirical-survival-function} For every $T>0$,
    define
    \[
        \mathbb{S}_{n}(t) := |B_{n}|^{1/2}\bigl(s_{n}(t)-s(t)\bigr), \qquad 0
        \le t\le T.
    \]
    Then
    \[
        \mathbb{S}_{n}\xrightarrow[n\to\infty]{d}\mathbb{S}\qquad\text{in }D(
        [0,T]),
    \]
    where $D([0,T])$ is endowed with the Skorokhod $J_{1}$-topology and $\mathbb{S}=(\mathbb{S}(t))_{0\le t\le T}$ is a
    centered Gaussian process with continuous sample paths. The covariance kernel of the limiting process is
    \[
        K_S(s,t)
        :=
        \lim_{n\to\infty}
        \frac{1}{|B_n|}
        \operatorname{Cov}\bigl(S_n(s),S_n(t)\bigr),
        \qquad s,t\ge0,
    \]
    where the existence of the limit follows from the proof below. Moreover, \[\sigma_S^2(t):=K_S(t,t)>0 \text{for all} t\ge0.\]
\end{theorem}

In particular, evaluation at any fixed $t\in[0,T]$ yields the corresponding one-dimensional central limit theorem for
$s_{n}(t)$.

\section{Geometric and Probabilistic Tools}
\label{sec:geometric-probabilistic-tools}

\subsection{Central Limit Theorem for Stabilizing Functionals}
\label{subsec:abstract-clt}

We establish a central limit theorem for stationary functionals of i.i.d.\ fields indexed by polynomial-growth Cayley
graphs, combining stabilization of resampling differences with approximation by local conditional functionals. This
extends Penrose's stabilization argument and accommodates functionals depending on coordinates outside the observation
set, provided that their local conditional approximation error is negligible on the variance scale. A similar
martingale construction for Bernoulli bond percolation on amenable Cayley graphs appears in \cite{de2026topology}.

For sets $A,B$, write \( A\oplus B:=(A\setminus B)\cup(B\setminus A) \) for their symmetric difference.

\begin{definition}[Exhaustive two-sided F{\o}lner sequence]
    Let $\Gamma$ be a finitely generated group. A sequence of finite sets 	$(A_{n})_{n\ge1}$ in $\Gamma$ is called an exhaustive two-sided F{\o}lner
    sequence if
    \[
        A_{n}\subseteq A_{n+1}, \qquad \bigcup_{n\ge1}A_{n}=\Gamma,
    \]
    and, for every finite set $K\subset\Gamma$,
    \[
        \frac{|KA_{n}\oplus A_{n}|+|A_{n}K\oplus A_{n}|}{|A_{n}|}\longrightarrow
        0 \qquad\text{as }n\to\infty.
    \]
\end{definition}

Let $G=(V,E)=\operatorname{Cay}(\Gamma,S)$ be a Cayley graph of an infinite group of polynomial growth. Let
$(\Theta,\mathcal{E},\mathbb{P}_{0})$ be a probability space, set $\Omega:=\Theta^{V}$ and
$\mathbb{P}:=\mathbb{P}_{0}^{\otimes V}$, and let $\mathcal{F}$ be the product $\sigma$-algebra. Write $X=(X_{x})_{x\in
V}$ for the corresponding i.i.d.\ field with common law $\mathbb{P} _{0}$.

Let $(A_{n})_{n\ge1}$ be an exhaustive two-sided F{\o}lner sequence and set $\mathfrak{A} :=\{gA_{n}:g\in\Gamma,\
n\ge1\}$. A measurable map $F:\Omega\times\mathfrak{A}\longrightarrow\mathbb{R}$ is called an $\mathfrak{A}$-functional
of $X$, and is denoted by $F(X,A)$.

Let $X'=(X'_{x})_{x\in V}$ be an independent copy of $X$. For each $y\in V$, let $X^{y}$ be obtained from $X$ by
replacing $X_{y}$ with $X'_{y}$. Define the diagonal action
\[
    \widehat\varphi_{g}(X,X') := \bigl( \widetilde\varphi_{g}X, \widetilde\varphi
    _{g}X' \bigr),
\]
where $(\widetilde\varphi_{g}X)_{x}:=X_{g^{-1}x}$. The dependence of the resampling differences on $X'$ will be
  suppressed from the notation. We say that $F$ is \emph{stationary} if $F(\widetilde\varphi_{g}X,gA)=F(X,A)$ for every
  $g\in\Gamma$ and $A\in\mathfrak{A}$.

\begin{definition}[Resampling difference]
    For $y\in V$ and $A\in\mathfrak{A}$, define
    \[
        \nabla_{y}(X,A) := F(X,A)-F(X^{y},A).
    \]
\end{definition}

For $R\ge0$ and finite $A\subset V$, define
\[
    \mathcal{G}_{R}(A) := \sigma\bigl(X_{z}:d_{G}(z,A)\le R\bigr).
\]
Fix a $\mathcal{G}_{R}(A)$-measurable version
\[
    F^{(R)}(X,A) := \mathbb{E}\bigl[F(X,A)\mid\mathcal{G}_{R}(A)\bigr].
\]
Since graph distance is integer-valued and the relevant family is countable, these versions may be chosen
  simultaneously so that
\[
    F^{(R)}(\widetilde\varphi_{g}X,gA) = F^{(R)}(X,A)
\]
almost surely for every $g\in\Gamma$ and $A\in\mathfrak{A}$. The notation $F^{(R)} (X^{y},A)$ denotes the same
  version evaluated at the resampled field $X^{y}$. Set
\[
    \nabla_{y}^{(R)}(X,A) := F^{(R)}(X,A)-F^{(R)}(X^{y},A).
\]

Since \( \widetilde\varphi_{g^{-1}}(X^{gy}) = \bigl(\widetilde\varphi_{g^{-1}}X\bigr )^{y}, \) stationarity gives
\[
    \nabla_{gy}(X,A) = \nabla_{y}\bigl( \widehat\varphi_{g^{-1}}(X,X'), g^{-1}
    A \bigr ),
\]
and the same identity holds for $\nabla^{(R)}$.

For a sequence \((C_k)_{k\ge1}\) in \(\mathfrak A\), we write \(C_k\to_{\mathrm{loc}}\Gamma\) if, for every \(r\ge0\),
\( B_G(o,r)\subseteq C_k \) for all sufficiently large \(k\). Equivalently, every finite subset of \(\Gamma\) is
eventually contained in \(C_k\).

\begin{definition}[Stabilization of resampling differences]
    \label{def:stabilization-resampling-differences}
    We say that \(F\) stabilizes on \(W\subseteq V\) if, for each
    \(x\in W\), there exists a random variable \(\nabla_x(\infty)\)
    such that, for every sequence \((C_k)_{k\ge1}\) in
    \(\mathfrak A\) satisfying \(C_k\to_{\mathrm{loc}}\Gamma\),
    \[
        \nabla_x(X,C_k)
        \xrightarrow[k\to\infty]{\mathbb P}
        \nabla_x(\infty).
    \]
\end{definition}

\begin{definition}[Local conditional approximation]
    \label{def:local-conditional-approximation} We say that $F$ is approximable
    by local conditional functionals along $(A_{n})$ if
    \[
        \lim_{R\to\infty}\limsup_{n\to\infty}|A_{n}|^{-1}\operatorname{Var}\bigl
        ( F (X,A_{n})-F^{(R)}(X,A_{n}) \bigr) =0.
    \]
\end{definition}

By \cite[Corollary~5.7]{de2026topology}, the group $\Gamma$ contains a normal, left-orderable subgroup $H$ of finite
index $q$. Fix a left-invariant order $\le_{H}$ on $H$, and let \( Y=\{y_{1},\ldots,y_{q}\} \) be a set of
representatives of the right cosets of $H$ in $\Gamma$. Every $g\in\Gamma$ can be written uniquely as $g=hy_{i}$, with
$h\in H$. Define a total order $\le$ on $\Gamma$ by declaring \( hy_{i}\le h'y_{j} \) if $i<j$, or if $i=j$ and
$h\le_{H}h'$. This order need not be left-invariant under the full group $\Gamma$, but it is invariant under left
multiplication by elements of $H$.

For $x\in V$, set
\[
    \mathfrak{F}_{x}:=\{z\in V:z\le x\}, \qquad \mathcal{F}_{x}:=\sigma(X_{z}
    :z\in \mathfrak{F}_{x}).
\]

We use the following martingale central limit theorem; see \cite{mcleish1974dependent}.

\begin{theorem}[McLeish, 1974]
    \label{thm:mcleish-clt} Let $\{Z_{n,j}:n\ge1,\ 1\le j\le b_{n}\}$ be a martingale-difference
    array with respect to filtrations $(\mathcal{H}_{n,j})_{0\le j\le b_n}$.
    Assume that
    \begin{enumerate}
        \item $\sup_{n\ge1}\mathbb{E}\left[ \max_{1\le j\le b_n}|Z_{n,j}|^{2}
                      \right ] <\infty;$

        \item $\max_{1\le j\le b_n}|Z_{n,j}| \xrightarrow[n\to\infty]{\mathbb{P}}
                  0;$

        \item $\sum_{j=1}^{b_n}Z_{n,j}^{2}\xrightarrow[n\to\infty]{\mathbb{P}}
                  1.$
    \end{enumerate}
    Then
    \[
        \sum_{j=1}^{b_n}Z_{n,j}\xrightarrow[n\to\infty]{d}\mathcal{N}(0,1).
    \]
\end{theorem}

\begin{theorem}[Central limit theorem for stabilizing stationary functionals] \label{thm:abstract-clt-functional} Let $G=\operatorname{Cay}(\Gamma,S)$
    be a Cayley graph of an infinite group of polynomial growth, and let $(A_{n})_{n\ge1}$ be an exhaustive two-sided F{\o}lner sequence. Set
    \[
        \mathfrak{A}:=\{gA_{n}:g\in\Gamma,\ n\ge1\}.
    \]
    Let $H\lhd\Gamma$ be the finite-index left-orderable subgroup fixed above, and let $Y$ be the corresponding set of
    right-coset representatives.

    Let $X=(X_{x})_{x\in V}$ be an i.i.d.\ field, and let $F$ be a stationary $\mathfrak{A}$-functional of $X$, with
    \[
        F(X,A)\in L^{2}(\mathbb{P}), \qquad A\in\mathfrak{A}.
    \]
    Assume that there exists $\gamma>2$ such that
    \[
        \sup_{R\ge0}\sup_{y\in Y}\sup_{A\in\mathfrak{A}}\mathbb{E}\left[ |\nabla
            _{y}^{(R)}(X,A)|^{\gamma}\right] <\infty.
    \]
    Assume also that $F$ stabilizes on $Y$ in the sense of Definition~\ref{def:stabilization-resampling-differences},
    and that, for every $R\ge0$, the truncated resampling differences stabilize on $Y$, with limits
    $\nabla_{y}^{(R)}(\infty)$ satisfying
    \[
        \nabla_{y}^{(R)}(\infty) \xrightarrow[R\to\infty]{L^2}\nabla_{y}(\infty
        ), \qquad y\in Y.
    \]
    Finally, assume that $F$ is approximable by local conditional functionals along $(A_{n})$.

    Then there exists $\sigma^{2}\ge0$ such that
    \[
        \frac{1}{|A_{n}|}\operatorname{Var}\bigl(F(X,A_{n})\bigr) \longrightarrow
        \sigma^{2},
    \]
    and
    \[
        |A_{n}|^{-1/2}\bigl( F(X,A_{n})-\mathbb{E}[F(X,A_{n})] \bigr) \xrightarrow
        [ n\to\infty]{d}\mathcal{N}(0,\sigma^{2}).
    \]
    Moreover,
    \[
        \sigma^{2}= \frac{1}{|Y|}\sum_{y\in Y}\mathbb{E}\left[ \mathbb{E}\bigl
            [ \nabla_{y}(\infty)\mid\mathcal{F}_{y}\bigr]^{2}\right].
    \]
\end{theorem}

\begin{proof}
    Fix $R\ge0$, and set
    \[
        B_{n,R}:=A_{n}^{+R}:= \{x\in V:d_{G}(x,A_{n})\le R\}, \qquad b_{n,R}:
        =|B_{n,R}|,
    \]
    and \( F_{n,R}:=F^{(R)}(X,A_{n}). \) The random variable $F_{n,R}$ is measurable with respect to the coordinates
    indexed by $B_{n,R}$. Since $(A_{n})$ is two-sided F{\o}lner,
    \[
        \frac{|B_{n,R}\setminus A_{n}|}{|A_{n}|}\longrightarrow0, \qquad \frac{b_{n,R}}{|A_{n}|}
        \longrightarrow1.
    \]

    Order the elements of $B_{n,R}$ according to the total order fixed above: \( x_{n,1}<x_{n,2}<\cdots<x_{n,b_{n,R}}.
    \) Set
    \[
        \mathcal{H}_{n,j}:= \sigma(X_{x_{n,i}}:1\le i\le j), \qquad 0\le j\le
        b_{n,R},
    \]
    where $\mathcal{H}_{n,0}$ is the trivial $\sigma$-algebra, and define
    \[
        D_{n,j}^{(R)}:= \mathbb{E}[F_{n,R}\mid\mathcal{H}_{n,j}] - \mathbb{E}
        [F_{n,R}\mid\mathcal{H}_{n,j-1}].
    \]
    Then
    \[
        F_{n,R}-\mathbb{E}[F_{n,R}] = \sum_{j=1}^{b_{n,R}}D_{n,j}^{(R)}
    \]
    and
    \[
        \operatorname{Var}(F_{n,R}) = \sum_{j=1}^{b_{n,R}}\mathbb{E}\left[(D_{n,j}
                ^{(R)})^{2}\right].
    \]

    Let $x=x_{n,j}$. Since the resampled coordinate $X_{x}'$ has the same law as $X_{x}$ and is independent of
    $\mathcal{H}_{n,j}$,
    \[
        \mathbb{E}\left[ F^{(R)}(X^{x},A_{n}) \,\middle|\, \mathcal{H}_{n,j}\right
        ] = \mathbb{E}\left[ F^{(R)}(X,A_{n}) \,\middle|\, \mathcal{H}_{n,j-1}
            \right] .
    \]
    Consequently,
    \begin{equation}
        D_{n,j}^{(R)}= \mathbb{E}\left[ \nabla_{x}^{(R)}(X,A_{n}) \,\middle|\,
            \mathcal{H}_{n,j}\right]. \label{eq:clt-martingale-resampling}
    \end{equation}

    Set
    \[
        C_{\gamma}:= \sup_{R\ge0}\sup_{y\in Y}\sup_{A\in\mathfrak{A}}\mathbb{E}
        \left [ |\nabla_{y}^{(R)}(X,A)|^{\gamma}\right].
    \]
    If $x=hy$, with $h\in H$ and $y\in Y$, stationarity gives
    \begin{equation}
        \nabla_{hy}^{(R)}(X,A) = \nabla_{y}^{(R)}\bigl( \widehat\varphi_{h^{-1}}
            (X, X'), h^{-1}A \bigr) \qquad\text{almost surely}. \label{eq:clt-translated-impact}
    \end{equation}
    Since $\widehat\varphi_{h^{-1}}$ preserves $\mathbb{P}$,
    \[
        \mathbb{E}\left[ |\nabla_{x}^{(R)}(X,A)|^{\gamma}\right] \le C_{\gamma}
        .
    \]
    Conditional Jensen's inequality and \eqref{eq:clt-martingale-resampling} therefore give
    \begin{equation}
        \mathbb{E}\left[ |D_{n,j}^{(R)}|^{\gamma}\right] \le C_{\gamma}. \label{eq:clt-difference-moment-bound}
    \end{equation}

    The two maximal conditions of Theorem~\ref{thm:mcleish-clt} follow from \eqref{eq:clt-difference-moment-bound}:
    \[
        \begin{aligned}
            \mathbb{E}\left[ \max_{1\le j\le b_{n,R}}\frac{|D_{n,j}^{(R)}|^2}{|A_n|}\right] & \le \frac{1}{|A_n|}\sum_{j=1}^{b_{n,R}}\mathbb{E}\left[ |D_{n,j}^{(R)}|^{2}\right] \\
                                                                                            & \le C_{\gamma}^{2/\gamma}\frac{b_{n,R}}{|A_n|},
        \end{aligned}
    \]
    which is uniformly bounded in $n$.

    Moreover, for every $\varepsilon>0$,
    \[
        \begin{aligned}
             & \mathbb{P}\left( \max_{1\le j\le b_{n,R}}|A_{n}|^{-1/2}|D_{n,j}^{(R)}| >\varepsilon \right)                                \\
             & \qquad\le \frac{1}{\varepsilon^\gamma|A_n|^{\gamma/2}}\sum_{j=1}^{b_{n,R}}\mathbb{E}\left[ |D_{n,j}^{(R)}|^{\gamma}\right] \\
             & \qquad\le \frac{C_\gamma b_{n,R}}{\varepsilon^\gamma|A_n|^{\gamma/2}}\longrightarrow0,
        \end{aligned}
    \]
    since $\gamma>2$.

    We next identify the limit of the sum of squared martingale differences. For $x\in B_{n,R}$, define
    \[
        \widehat D_{n,R}(x) := \mathbb{E}\left[ \nabla_{x}^{(R)}(X,A_{n}) \,\middle
            |\, \sigma(X_{z}:z\in B_{n,R},\ z\le x) \right].
    \]
    Then
    \begin{equation}
        D_{n,j}^{(R)}
        =
        \widehat D_{n,R}(x_{n,j}).
        \label{eq:clt-finite-martingale-field}
    \end{equation}

    For \(y\in Y\), set
    \[
        d_y^{(R)}
        :=
        \mathbb E\left[
            \nabla_y^{(R)}(\infty)
            \,\middle|\,
            \mathcal F_y
            \right].
    \]
    If \(x=hy\), with \(h\in H\) and \(y\in Y\), define
    \begin{equation}
        d_x^{(R)}
        :=
        d_y^{(R)}\circ\widetilde\varphi_{h^{-1}}.
        \label{eq:clt-limit-field}
    \end{equation}
    The invariance of the order under left multiplication by \(H\) gives
    \begin{equation}
        h^{-1}\mathfrak F_{hy}
        =
        \mathfrak F_y.
        \label{eq:clt-past-equivariance}
    \end{equation}

    The argument has two steps. First, we show that, at vertices deep inside \(A_n\), the martingale difference is
    asymptotically equal in \(L^2\) to the limiting field \(d_x^{(R)}\). We then average \(d_x^{(R)\,2}\) over \(A_n\)
    by decomposing \(A_n\) into right cosets of \(H\) and applying the mean ergodic theorem on each coset.

    We claim that
    \begin{equation}
        \frac{1}{|A_{n}|}\sum_{x\in B_{n,R}}\left| \widehat D_{n,R}(x)^{2}-d_{x}
            ^{(R)\,2}\right| \xrightarrow[n\to\infty]{L^1}0. \label{eq:clt-interior-approximation}
    \end{equation}

    To prove the claim, consider sequences $n_{k}\to\infty$ and $x_{k}\in A_{n_k}$ such that
    \[
        d_{G}(x_{k},V\setminus A_{n_k}) \longrightarrow\infty.
    \]
    Since $Y$ is finite, after passing to a subsequence we may write $x_{k}=h _{k}y$ with the same $y\in Y$ for every
    $k$. Set $C_{k}:=h_{k}^{-1}A_{n_k}.$ Since $d_{G}(x_{k},V\setminus A_{n_k})\to\infty$ and left translations
    preserve graph distance, $C_{k}$ tends locally to $\Gamma$. By stabilization,
    \[
        \nabla_{y}^{(R)}(X,C_{k}) \xrightarrow[k\to\infty]{\mathbb{P}}\nabla_{y}
        ^{(R)}(\infty).
    \]
    The uniform $\gamma$-moment bound implies uniform integrability of the squares, and hence
    \begin{equation}
        \nabla_{y}^{(R)}(X,C_{k}) \xrightarrow[k\to\infty]{L^2}\nabla_{y}^{(R)}
        (\infty ). \label{eq:clt-impact-l2}
    \end{equation}

    Let \( \mathcal{K}_{k}:= \sigma\bigl( X_{z}:z\le y,\ z\in C_{k}^{+R}\bigr). \) Every finite subset of
    $\mathfrak{F}_{y}$ is eventually contained in $C_{k} ^{+R}$. Thus, for every $Z\in L^{2}(\mathbb{P})$,
    \begin{equation}
        \mathbb{E}[Z\mid\mathcal{K}_{k}] \xrightarrow[k\to\infty]{L^2}\mathbb{E}
        [Z\mid \mathcal{F}_{y}]. \label{eq:clt-past-approximation}
    \end{equation}
    Since $\mathcal{K}_{k}\subseteq\mathcal{F}_{y}$, we have
    \[
        \mathbb{E}[Z\mid\mathcal{K}_{k}] = \mathbb{E}\left[ \mathbb{E}[Z\mid\mathcal{F}
                _{y}] \,\middle|\, \mathcal{K}_{k}\right].
    \]
    If $\zeta$ is a cylinder random variable depending on finitely many coordinates of $\mathfrak{F}_{y}$, then $\zeta$
    is $\mathcal{K}_{k}$-measurable for all sufficiently large $k$, and
    \[
        \left\| \mathbb{E}[Z\mid\mathcal{K}_{k}] - \mathbb{E}[Z\mid\mathcal{F}
            _{y}] \right\|_{2}\le 2\left\| \mathbb{E}[Z\mid\mathcal{F}_{y}]-\zeta
        \right\|_{2}.
    \]
    The conclusion follows from the density of cylinder random variables in $L ^{2}(\mathcal{F}_{y})$.

    Combining \eqref{eq:clt-impact-l2}, \eqref{eq:clt-past-approximation}, and the $L^{2}$-contractivity of conditional
    expectation yields
    \[
        \mathbb{E}\left[ \nabla_{y}^{(R)}(X,C_{k}) \,\middle|\, \mathcal{K}_{k}
            \right ] \xrightarrow[k\to\infty]{L^2}d_{y}^{(R)}.
    \]
    By stationarity and \eqref{eq:clt-past-equivariance}, this implies
    \[
        \widehat D_{n_k,R}(x_{k})-d_{x_k}^{(R)}\xrightarrow[k\to\infty]{L^2}0
        .
    \]
    The uniform second-moment bounds therefore give
    \begin{equation}
        \mathbb{E}\left[ \left| \widehat D_{n_k,R}(x_{k})^{2}-d_{x_k}^{(R)\,2}
            \right | \right] \longrightarrow0. \label{eq:clt-interior-squares}
    \end{equation}

    For $L\ge0$, let
    \[
        A_{n}^{-L}:= \{x\in A_{n}:d_{G}(x,V\setminus A_{n})>L\}.
    \]
    The argument above implies
    \[
        \lim_{L\to\infty}\sup_{\substack{n\ge1\\x\in A_n^{-L}}}\mathbb{E}\left
        [ \left | \widehat D_{n,R}(x)^{2}-d_{x}^{(R)\,2}\right| \right] =0.
    \]
    Indeed, otherwise there would exist $L_{k}\to\infty$, $n_{k}\ge1$, and $x_{k}\in A_{n_k}^{-L_k}$ for which the
    expectations remain bounded away from zero. Since each $A_{n}$ is finite, necessarily $n_{k}\to\infty$,
    contradicting \eqref{eq:clt-interior-squares}.

    For each fixed $L$, the two-sided F{\o}lner property gives
    \[
        \frac{|A_{n}\setminus A_{n}^{-L}|}{|A_{n}|}\longrightarrow0, \qquad \frac{|B_{n,R}\setminus
            A_{n}|}{|A_{n}|}\longrightarrow0.
    \]
    Together with the uniform second-moment bounds, these estimates prove \eqref{eq:clt-interior-approximation}.

    It remains to average the limiting field over the cosets of $H$. For $1\le m\le q$, set
    \[
        H_{n,m}:= \{h\in H:hy_{m}\in A_{n}\}.
    \]
    We first record the asymptotic density of the cosets. If $Hy_{i}$ and $Hy_{j}$ are two right cosets, right
    multiplication by $y_{i}^{-1}y_{j}$ maps $Hy_{i}$ onto $Hy_{j}$. Since $(A_{n})$ is right F{\o}lner, the
    cardinalities of $A_{n}\cap Hy_{i}$ and $A_{n}\cap Hy_{j}$ differ by $o(|A_{n}|)$. Summing over the $q$ cosets
    gives
    \begin{equation}
        \frac{|H_{n,m}|}{|A_{n}|}= \frac{|A_{n}\cap Hy_{m}|}{|A_{n}|}\longrightarrow
        \frac{1}{q}. \label{eq:clt-coset-density}
    \end{equation}

    Moreover, for every $k\in H$,
    \[
        |kH_{n,m}\oplus H_{n,m}| \le |kA_{n}\oplus A_{n}|.
    \]
    Since $H\lhd\Gamma$, we also have
    \[
        |H_{n,m}k\oplus H_{n,m}| \le \left|A_{n}(y_{m}^{-1}ky_{m})\oplus
        A_{n}\right|.
    \]
    Together with \eqref{eq:clt-coset-density}, which gives $|H_{n,m}|\sim |A_{n}|/q$, these inequalities imply the
    left and right F{\o}lner properties inside $H$. Hence $(H_{n,m})_{n\ge1}$ is a two-sided F{\o}lner sequence in $H$.

    Under the identification $\Theta^{\Gamma}\cong(\Theta^{Y})^{H},$ the restriction of the product action to $H$ is a
    Bernoulli shift and is therefore ergodic. By \eqref{eq:clt-limit-field},
    \begin{equation}
        \frac{1}{|A_{n}|}\sum_{x\in A_n}d_{x}^{(R)\,2}= \sum_{m=1}^{q}\frac{|H_{n,m}|}{|A_{n}|}
        \frac{1}{|H_{n,m}|}\sum_{h\in H_{n,m}}d_{y_m}^{(R)\,2}\circ\widetilde
        \varphi_{h^{-1}}. \label{eq:clt-coset-decomposition}
    \end{equation}

    By passing to an almost surely convergent subsequence and applying Fatou's lemma,
    \[
        \mathbb{E}\left[ |\nabla_{y}^{(R)}(\infty)|^{\gamma}\right] \le C_{\gamma}
        .
    \]
    In particular, $d_{y}^{(R)\,2}\in L^{1}(\mathbb{P})$. For $M>0$, set
    \[
        W_{y,M}^{(R)}:= d_{y}^{(R)\,2}\wedge M.
    \]
    Then $W_{y,M}^{(R)}\in L^{2}(\mathbb{P})$. Since $(H_{n,m}^{-1})_{n\ge1}$ is a left F{\o}lner sequence, the
    unitary-representation form of the mean ergodic theorem for amenable groups \cite[Theorem~8.13]{manfred259ergodic},
    applied to $(H_{n,m}^{-1})_{n\ge1}$, gives
    \[
        \frac{1}{|H_{n,m}|}\sum_{h\in H_{n,m}}W_{y_m,M}^{(R)}\circ\widetilde\varphi
        _{h^{-1}}\xrightarrow[n\to\infty]{L^2}\mathbb{E}\left[W_{y_m,M}^{(R)}
            \right ].
    \]
    Since ergodic averages are contractions in $L^{1}$, letting $M\to\infty$ yields
    \[
        \frac{1}{|H_{n,m}|}\sum_{h\in H_{n,m}}d_{y_m}^{(R)\,2}\circ\widetilde
        \varphi_{h^{-1}}\xrightarrow[n\to\infty]{L^1}\mathbb{E}\left[d_{y_m}^{(R)\,2}
            \right ].
    \]
    Combining this with \eqref{eq:clt-coset-density} and \eqref{eq:clt-coset-decomposition}, we obtain
    \begin{equation}
        \frac{1}{|A_{n}|}\sum_{x\in A_n}d_{x}^{(R)\,2}\xrightarrow[n\to\infty
        ]{L^1}\frac{1}{|Y|}\sum_{y\in Y}\mathbb{E}\left[d_{y}^{(R)\,2}\right]
        . \label{eq:clt-limit-field-average}
    \end{equation}
    The same limit holds with $B_{n,R}$ in place of $A_{n}$, since
    \[
        |B_{n,R}\setminus A_{n}|=o(|A_{n}|).
    \]

    Combining \eqref{eq:clt-finite-martingale-field}, \eqref{eq:clt-interior-approximation}, and
    \eqref{eq:clt-limit-field-average} gives
    \begin{equation}
        \frac{1}{|A_{n}|}\sum_{j=1}^{b_{n,R}}(D_{n,j}^{(R)})^{2}\xrightarrow[
            n\to\infty ]{L^1}\sigma_{R}^{2}, \label{eq:clt-quadratic-variation}
    \end{equation}
    where
    \begin{equation}
        \sigma_{R}^{2}:= \frac{1}{|Y|}\sum_{y\in Y}\mathbb{E}\left[ \mathbb{E}
            \left [ \nabla_{y}^{(R)}(\infty) \,\middle|\, \mathcal{F}_{y}\right]^{2}
            \right] . \label{eq:clt-sigma-r}
    \end{equation}
    In particular, the convergence in \eqref{eq:clt-quadratic-variation} holds
    in probability. Taking expectations and using the orthogonality of the
    martingale differences also gives
    \begin{equation}
        \frac{1}{|A_{n}|}\operatorname{Var}(F_{n,R}) \longrightarrow \sigma_{R}
        ^{2}. \label{eq:clt-truncated-variance}
    \end{equation}

    If $\sigma_{R}^{2}>0$, apply Theorem~\ref{thm:mcleish-clt} to
    \[
        Z_{n,j}^{(R)}:= \frac{D_{n,j}^{(R)}}{\sigma_{R}|A_{n}|^{1/2}}.
    \]
    The maximal estimates and \eqref{eq:clt-quadratic-variation} give
    \begin{equation}
        |A_{n}|^{-1/2}\bigl( F_{n,R}-\mathbb{E}[F_{n,R}] \bigr) \xrightarrow[
            n\to\infty ]{d}\mathcal{N}(0,\sigma_{R}^{2}). \label{eq:clt-truncated-clt}
    \end{equation}
    If $\sigma_{R}^{2}=0$, the same conclusion follows from \eqref{eq:clt-truncated-variance},
    since the normalized centered variables then converge to zero in $L^{2}$.

    To remove the truncation, observe that
    \[
        \nabla_{y}^{(R)}(\infty) \xrightarrow[R\to\infty]{L^2}\nabla_{y}(\infty
        ).
    \]
    By the $L^{2}$-contractivity of conditional expectation, this implies
    \[
        \mathbb{E}\left[ \nabla_{y}^{(R)}(\infty) \,\middle|\, \mathcal{F}_{y}
            \right ] \xrightarrow[R\to\infty]{L^2}\mathbb{E}\left[ \nabla_{y}(\infty
            ) \,\middle |\, \mathcal{F}_{y}\right].
    \]
    Therefore,
    \begin{equation}
        \sigma_{R}^{2}\longrightarrow \sigma^{2}:= \frac{1}{|Y|}\sum_{y\in Y}
        \mathbb{E}\left[ \mathbb{E}\left[ \nabla_{y}(\infty) \,\middle|\, \mathcal{F}
                _{y}\right ]^{2}\right]. \label{eq:clt-sigma-limit}
    \end{equation}

    Finally, local conditional approximation (Definition \ref{def:local-conditional-approximation}) gives
    \begin{equation}
        \lim_{R\to\infty}\limsup_{n\to\infty}\mathbb{E}\left[ \left| \frac{ F(X,A_{n})-F_{n,R}-
                \mathbb{E}[F(X,A_{n})-F_{n,R}] }{|A_{n}|^{1/2}}\right|^{2}\right] =0.
        \label{eq:clt-local-approximation}
    \end{equation}
    Together with \eqref{eq:clt-truncated-variance} and \eqref{eq:clt-sigma-limit},
    the inequality
    \[
        \left| \sqrt{\frac{\operatorname{Var}(F(X,A_{n}))}{|A_{n}|}}- \sqrt{\frac{\operatorname{Var}(F_{n,R})}{|A_{n}|}}
        \right| \le \sqrt{ \frac{\operatorname{Var}(F(X,A_{n})-F_{n,R})}{|A_{n}|}
        }
    \]
    and \eqref{eq:clt-local-approximation} show that
    \[
        \frac{1}{|A_{n}|}\operatorname{Var}(F(X,A_{n})) \longrightarrow \sigma
        ^{2}.
    \]
    The converging-together theorem \cite[Theorem~3.2]{billingsley1999convergence}, together with
    \eqref{eq:clt-truncated-clt}--\eqref{eq:clt-local-approximation}, gives
    \[
        |A_{n}|^{-1/2}\bigl(F(X,A_{n})-\mathbb{E}[F(X,A_{n})]\bigr) \xrightarrow
        [n\to\infty]{d}\mathcal{N}(0,\sigma^{2}).
    \]
\end{proof}

\begin{remark}
    If $F(X,A)$ is measurable with respect to $\sigma(X_{x}:x\in A)$, then \( F^{(R)}(X,A)=F(X,A) \) for every $R\ge0$.
    Hence \( \nabla_{y}^{(R)}(X,A)=\nabla_{y}(X,A),\) \(\nabla_{y}^{(R)}(\infty )=\nabla_{y}(\infty). \) Thus the
    additional approximation step is unnecessary for local functionals, and the theorem reduces to Penrose's local
    stabilization framework in the present Cayley-graph setting; see \cite{penrose2001}.
\end{remark}

\subsection{Geometric Lemma}
\label{subsec:geometric-lemma}

We characterize linear growth of polynomial-growth Cayley graphs through the existence of uniformly bounded disjoint
cutsets. This characterization is the geometric input in the proof of extinction.

\begin{definition}[Cutset for a vertex]
    Let $G=(V,E)$ be an infinite connected graph and let $o\in V$. A set of edges $A\subset E$ is called a cutset for $o$ if, after removing the edges
    of $A$, the vertex $o$ lies in a finite connected component. Equivalently, every infinite simple path starting from $o$ crosses at least one edge
    of $A$.
\end{definition}

\begin{definition}[One-dimensional graph]
    Let $G=(V,E)$ be an infinite connected graph. We say that $G$ is one-dimensional 	if there exist a vertex $o\in V$ and a sequence $(A_{n})_{n\ge1}$ of
    pairwise disjoint cutsets for $o$ such that
    \(
    \sup_{n\ge1}|A_{n}|<\infty.
    \)
\end{definition}

\begin{lemma}
    \label{lem:linear-growth-one-dimensional} Let $G=\Cay(\Gamma,S)$ be the Cayley graph of an infinite finitely generated group $\Gamma$, and assume that
    $G$ has polynomial growth. Then $G$ has linear growth if and only if $G$ is one-dimensional.
\end{lemma}

\begin{proof}
    We prove the two implications separately.

    \noindent
    \emph{Linear growth implies one-dimensional.} Assume first that $G$ has linear growth. Then $\Gamma$ is virtually cyclic; see
    \cite[Chapter~I, Section~5]{Woess}. Thus $\Gamma$ contains a finite-index 	subgroup $H\cong\mathbb{Z}$. Let
    \[
        C:=\operatorname{Core}_{\Gamma}(H) = \bigcap_{\gamma\in\Gamma}\gamma
        H\gamma^{-1}.
    \]
    Since $C$ is the kernel of the action of $\Gamma$ on the finite coset space $\Gamma/H$, it is a normal finite-index
    subgroup of $\Gamma$. Moreover, $C\subseteq H\cong\mathbb{Z}$, and hence $C$ is infinite cyclic. We may therefore
    write \( C=\langle g\rangle\cong\mathbb{Z}. \)

    Fix a complete set $T=\{t_{1},\dots,t_{m}\}$ of right coset representatives for $C$ in $\Gamma$, chosen so that
    $e\in T$. Then every $x\in\Gamma$ can be written uniquely as \( x=g^{h(x)}t(x), \) with $h(x)\in\mathbb{Z}$ and
    $t(x)\in T$.

    For each $t\in T$ and $s\in S$, write $ts=g^{a(t,s)}u(t,s)$, where $a(t,s)\in\bbZ$ and $u(t,s)\in T$. Since $T$ and
    $S$ are finite, we may choose
    \[
        L:=\max\{1,|a(t,s)|:\ t\in T,\ s\in S\}<\infty.
    \]
    It follows that adjacent vertices have heights differing by at most $L$: $|h(xs)-h(x)|\le L$ for every $x\in\Gamma$
    and $s\in S$.

    For $n\ge1$, set $r_{n}:=3Ln$. Let $A_{n}$ be the set of edges whose endpoints lie on opposite sides of one of the
    two thresholds $r_{n}$ and $-r_{n}$. More precisely, $\{x,y\}\in A_{n}$ if either $h(x)<r_{n}\le h(y)$ or $h(y
    )<r_{n}\le h(x)$, or if either $h(x)\le -r_{n}<h(y)$ or $h(y)\le -r_{n}<h(x)$.

    The sets $A_{n}$ are uniformly bounded. If an edge crosses the positive threshold $r_{n}$, then one endpoint has
    height in $\{r_{n},\dots,r_{n}+L-1\}$; if it crosses the negative threshold $-r_{n}$, then one endpoint has height
    in $\{-r_{n}-L+1,\dots,-r_{n}\}$. Each level of $h$ has cardinality $m=|T |$, and the degree of $G$ is at most
    $|S|$. Hence \( |A_{n}|\le 2Lm|S|\) for all \(n\ge1.\) Moreover, no edge can cross two different thresholds, since
    every edge has $h$-diameter at most $L$, whereas distinct thresholds among the numbers $\pm 3Ln$ are separated by
    at least $3L$. Therefore the cutsets $A_{n}$ are pairwise disjoint.

    Each $A_{n}$ separates the identity from infinity. Set $W_{n}:=\{x\in\Gamma:\ |h(x)|<r_{n}\}$. This is a finite set
    containing $e$. Any path from $e$ to $\Gamma\setminus W_{n}$ must have a first vertex whose height is at least
    $r_{n}$ or at most $-r_{n}$. The preceding edge then crosses the corresponding threshold, and hence belongs to
    $A_{n}$. Thus every path from $e$ to $\Gamma\setminus W_{n}$ meets $A_{n}$. Consequently, in $G\setminus A_{n}$,
    the connected component of $e$ is contained in the finite set $W_{n}$. Hence $A_{n}$ is a cutset for $e$.

    Thus $(A_{n})$ is a sequence of pairwise disjoint cutsets of uniformly bounded size. Therefore $G$ is
    one-dimensional.

    \noindent
    \emph{One-dimensional implies linear growth.} Assume now that $G$ is one-dimensional.
    Then there exist a vertex $o\in V$ and pairwise disjoint finite cutsets $(A_{n})_{n\ge1}$ for $o$ such that, if $U_{n}$ denotes the connected component
    of $o$ in $G\setminus A_{n}$, then $U_{n}$ is finite for every $n$, and $\sup_{n}|A_{n}|\le M < \infty$.

    We show that the polynomial growth degree cannot be larger than one. Suppose for contradiction that $G$ has
    polynomial growth degree $D>1$. By the edge-isoperimetric inequality for finitely generated groups of polynomial
    growth, there exists $c>0$ such that every finite nonempty set $F\subset\Gamma$ satisfies
    \[
        |\partial_{E}F| \ge c|F|^{(D-1)/D},
    \]
    where $\partial_{E}F$ denotes the edge boundary of $F$; see, for instance, \cite{CoulhonSaloffCoste}.

    Every edge from $U_{n}$ to $V\setminus U_{n}$ must belong to $A_{n}$, because $U_{n}$ is the connected component of
    $o$ in $G\setminus A_{n}$. Hence $\partial_{E}U_{n}\subseteq A_{n}$, and therefore
    \[
        c|U_{n}|^{(D-1)/D}\le |\partial_{E}U_{n}| \le |A_{n}| \le M.
    \]
    Since $D>1$, this implies that the cardinalities $|U_{n}|$ are uniformly bounded.

    On the other hand, the cutsets move away from $o$. To see this, fix $R<\infty$. Since $G$ has bounded degree, only
    finitely many edges have an endpoint in $B_{G}(o,R)$. Because the sets $A_{n}$ are pairwise disjoint, only finitely
    many of them can meet this finite edge set. Hence $\operatorname{dist} (o,A_{n})\to\infty$, where the distance from
    $o$ to $A_{n}$ is the minimum distance from $o$ to an endpoint of an edge in $A_{n}$.

    Now let $r_{n}:=\operatorname{dist}(o,A_{n})$. If $x\in B_{G}(o,r_{n}-1)$, then any geodesic from $o$ to $x$ avoids
    $A_{n}$. Thus $B_{G}(o,r_{n}-1)\subseteq U_{n}$. Since $r_{n}\to\infty$ and $G$ is infinite, the balls
    $B_{G}(o,r_{n}-1)$ have cardinality tending to infinity. Consequently $|U_{n}|\to\infty$, contradicting the uniform
    bound obtained above.

    This contradiction rules out $D>1$, so $D=1$ and $G$ has linear growth.
\end{proof}

\section{Proofs of the Extinction and Survival Results}
\label{sec:proofs-extinction-survival}

\subsection{Extinction on Graphs of Linear Growth}
\label{subsec:proof-1d-extinction}

Throughout the proof, we use deferred revelation of the Harris construction: the Poisson processes indexed by the
ordered edges \((x,y)\), \(y\sim x\), are left unrevealed until \(x\) first becomes a spreader. Since their marks are
ineffective before that time, this does not alter the process.

\begin{proof}[Proof of Theorem~\ref{thm:linear-growth-extinction}]
    By Lemma~\ref{lem:linear-growth-one-dimensional}, $G$ is one-dimensional.
    Hence there exist $o\in V$ and pairwise disjoint cutsets $(A_{n})_{n\ge1}$ 	for $o$ such that
    \(
    \sup_{n\ge1}|A_{n}|\le K<\infty.
    \)
    Since $G$ has bounded degree, set
    \(
    \Delta:=\sup_{x\in V}\deg(x)<\infty.
    \)

    Let $I_{0}$ be the initial non-ignorant set and define
    \[
        I_{0}^{+}:= I_{0}\cup \{y\in V:\text{ there exists }x\in I_{0}\text{
            with }y\sim x\}.
    \]
    Choose a finite connected set $F_{0}\subset V$ containing $\{o\}\cup I_{0} ^{+}$.

    Passing to a subsequence of the cutsets and relabelling, choose $A_{1}$ disjoint from every edge incident to
    $F_{0}$. For each $m\ge1$, let $V_{m}$ be the finite connected component of $o$ in $G\setminus A_{m}$. We may then
    choose the remaining cutsets successively so that
    \[
        F_{0}\subset V_{1}, \qquad \partial_{V}^{\mathrm{in}}V_{1}\cap F_{0}=
        \varnothing , \qquad V_{m}\subset V_{m+1}, \qquad \partial_{V}^{\mathrm{in}}
        V_{m+1}\subset V\setminus V_{m}
    \]
    for every $m\ge1$, where
    \[
        \partial_{V}^{\mathrm{in}}V_{m}:= \{x\in V_{m}:\text{ there exists }y
        \notin V_{m}\text{ with }x\sim y\}
    \]
    is the inner vertex boundary of $V_{m}$; see Figure~\ref{fig:cutset-z23}. These choices are possible because the
    cutsets are pairwise disjoint and \(G\) has bounded degree, so only finitely many of them meet the edges incident
    to any given finite set. If no edge of \(A_{m+1}\) is incident to \(V_m\), every vertex of \(V_m\) remains
    connected to \(o\) in \(G\setminus A_{m+1}\), and hence \(V_m\subset V_{m+1}\). Moreover, any
    \(x\in\partial_V^{\mathrm{in}}V_{m+1}\) sends an edge of \(A_{m+1}\) to \(V\setminus V_{m+1}\), so \(x\notin V_m\).

    Since $V_{m}$ is the component of $o$ in $G\setminus A_{m}$, every edge from $V_{m}$ to $V\setminus V_{m}$ belongs
    to $A_{m}$. Therefore
    \[
        |\partial_{V}^{\mathrm{in}}V_{m}| \le |\partial_{E}V_{m}| \le |A_{m}|
        \le K.
    \]

    \begin{figure}[ht]
        \centering
        \begin{tikzpicture}[scale=0.75, every node/.style={font=\small}]
            \filldraw[fill=gray!30, draw=gray!65, line width=0.8pt]
            (-5.5,-0.46) rectangle (5.5,0.66);

            \foreach \x in {-8,-7,...,6}
                { \pgfmathtruncatemacro{\y}{\x+2} \draw[black] (\x,0) to[bend right=18] (\y,0); }

            \foreach \x in {-8,-7,...,5} { \pgfmathtruncatemacro{\y}{\x+3} \draw[black] (\x,0) to[bend left=28] (\y,0); }

            \draw[red, very thick] (-7,0) to[bend right=18] (-5,0);
            \draw[red, very thick] (-6,0) to[bend right=18] (-4,0);
            \draw[red, very thick] (4,0) to[bend right=18] (6,0);
            \draw[red, very thick] (5,0) to[bend right=18] (7,0);

            \draw[red, very thick] (-8,0) to[bend left=28] (-5,0);
            \draw[red, very thick] (-7,0) to[bend left=28] (-4,0);
            \draw[red, very thick] (-6,0) to[bend left=28] (-3,0);
            \draw[red, very thick] (3,0) to[bend left=28] (6,0);
            \draw[red, very thick] (4,0) to[bend left=28] (7,0);
            \draw[red, very thick] (5,0) to[bend left=28] (8,0);

            \foreach \x in {-8,-7,...,8}
                { \node[circle, fill=black, inner sep=1.4pt] (v\x) at (\x,0) {}; \node[below=8pt] at (\x,0) {$\x$}; }

            \node[circle, draw=black, very thick, inner sep=3pt] at (v0) {};
            \node[above=14pt] at (0,0) {$o$};

            \foreach \x in {-5,-4,-3,3,4,5}
                { \node[circle, fill=red, very thick, inner sep=2.2pt] at (v\x) {}; }

            \node at (0,1.35) {$V_{m}$};
            \node[red] at (7.15,2.05) {$A_{m}$};
            \node[red!90!black]
            at
            (4.75,1.30)
            {$\partial_{V}^{\mathrm{in}}V_{m}$};

            \draw[->, thick] (6.75,1.88) -- (5.62,0.82);
            \draw[->, thick] (4.45,0.95) -- (5.02,0.12);
        \end{tikzpicture}
        \caption{The geometric setting in the proof of Theorem~\ref{thm:linear-growth-extinction}
            for the Cayley graph of $\bbZ$ with generators $\{\pm2,\pm3\}$. The shaded region is the finite component $V_{m}$ of the origin after removing the
            cutset $A_{m}$, shown in red. The red vertices form the inner vertex boundary $\partial_{V}^{\mathrm{in}}V_{m}$. Any infection from
            $V_{m}$ to $V\setminus V_{m}$ must start from this boundary and cross an edge of $A_{m}$.}
        \label{fig:cutset-z23}
    \end{figure}
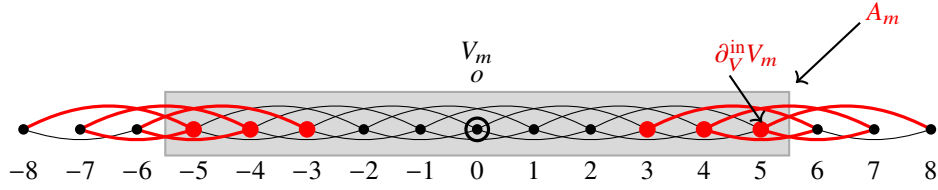

    Define
    \[
        \theta_{m}:= \inf\{ t\ge0: \text{ some }x\in V\setminus V_{m}\text{ satisfies
        }\eta_{t}(x)\ne0 \},
    \]
    and
    \[
        \beta_{m}:= \inf\left\{ t\ge0: \text{ some vertex of }\partial_{V}^{\mathrm{in}}
        V_{m}\text{ becomes a spreader}\right\}.
    \]
    Let $(\mathcal{T}_{t})_{t\ge0}$ be the natural filtration generated by the initial configuration and the graphical
    marks up to time $t$.

    We claim that there exists $p_{0}>0$, independent of $m$, such that
    \[
        \Prob(\theta_{m}=\infty\mid\mathcal{T}_{\beta_m}) \ge p_{0}\qquad \text{on
        }\{\beta_{m}<\infty\}.
    \]

    Fix $m$ and work on $\{\beta_{m}<\infty\}$. List chronologically the distinct vertices of
    $\partial_{V}^{\mathrm{in}}V_{m}$ that become spreaders before the rumor exits $V_{m}$. There are at most $K$ such
    vertices.

    If $x\in\partial_{V}^{\mathrm{in}}V_{m}$ becomes a spreader at time $s_{x} <\theta_{m}$, then it was infected by a
    parent $u\in V_{m}$, since no vertex outside $V_{m}$ is non-ignorant before $\theta_{m}$. From time $s_{x}$ onward,
    $u$ remains non-ignorant. Hence, while $x$ is a spreader, the clock $N_{\mathrm{stif}}^{x,u}$ is active at rate
    $\alpha$, whereas the infection clocks from $x$ to vertices outside $V_{m}$ have total rate at most $\Delta
    \lambda$. Therefore, conditionally on the past at time $s_{x}$, the probability that $x$ is stifled by its parent
    before spreading to the outside of $V_{m}$ is at least
    \[
        q:=\frac{\alpha}{\alpha+\Delta\lambda}.
    \]

    Perform these tests sequentially, beginning with the boundary spreader born at time $\beta_{m}$ and stopping if the
    rumor exits $V_{m}$ or no further boundary spreader is born. When a boundary vertex becomes a spreader, reveal its
    entire outgoing clock family. Since the tested vertices are distinct sources, the strong Markov property and
    independence of these families imply that, conditionally on all previously revealed information, each test succeeds
    with probability at least $q$. If the procedure stops after fewer than $K$ tests without an exit, regard the
    remaining tests as successes. Iterated conditioning then shows that the probability that all boundary spreaders are
    stopped before infecting outside is at least $q^{K}$. On this event the rumor cannot exit $V_{m}$, since the first
    infection from $V_{m}$ to $V\setminus V_{m}$ would have to be caused by a spreader in
    $\partial_{V}^{\mathrm{in}}V_{m}$. Thus the claim holds with
    \[
        p_{0}:= q^{K}= \left( \frac{\alpha}{\alpha+\Delta\lambda}\right)^{K}>
        0.
    \]

    We now iterate this estimate. Since $V_{m}\subset V_{m+1}$,
    \[
        \{\theta_{m+1}<\infty\} \subseteq \{\theta_{m}<\infty\}.
    \]
    Moreover, \( \partial_{V}^{\mathrm{in}}V_{m+1}\subset V\setminus V_{m}, \) so \( \{\beta_{m+1}<\infty\} \subseteq
    \{\theta_{m}<\infty\}. \) Any exit from $V_{m+1}$ must also be preceded by the appearance of a spreader on
    $\partial_{V}^{\mathrm{in}}V_{m+1}$. Therefore
    \[
        \Prob(\theta_{m+1}<\infty) \le (1-p_{0})\Prob(\beta_{m+1}<\infty) \le
        (1-p_{0})\Prob(\theta_{m}<\infty).
    \]
    By induction,
    \[
        \Prob(\theta_{r}<\infty) \le (1-p_{0})^{r-1}, \qquad r\ge1.
    \]

    The events $\{\theta_{r}<\infty\}$ are decreasing, and hence
    \[
        \Prob(\theta_{r}<\infty\text{ for all }r\ge1) = \lim_{r\to\infty}\Prob
        (\theta_{r}<\infty) = 0.
    \]
    Thus, almost surely, there exists a finite random index $M$ such that $\theta_{M}=\infty$. On this event no vertex
    outside the finite set $V_{M}$ ever becomes non-ignorant.

    By Proposition~\ref{prop:local-extinction} and the countability of \(V\), almost surely every vertex spends only a
    finite amount of time in the spreader state. Since no vertex outside the finite set \(V_M\) ever becomes
    non-ignorant, only finitely many spreaders appear. The maximum of their stifling times is therefore an almost
    surely finite time \(\tau\) such that
    \[
        \eta_t(x)\ne1
        \qquad\text{for every }x\in V\text{ and every }t\ge\tau.
    \]
\end{proof}

\subsection{Local Extinction}
\label{subsec:proof-local-extinction}

\begin{proof}[Proof of Proposition~\ref{prop:local-extinction}]
    Fix $x\in V$. Since the only possible transitions are $0\to1$ and $1\to2$, the vertex $x$ can become a spreader at most once.

    Once $x$ has a non-ignorant neighbor, it is stifled at rate at least $\alpha$. If, when it becomes a spreader, all
    its neighbors are ignorant, then it creates a non-ignorant neighbor in finite time almost surely, at total rate
    $\lambda\deg(x)>0$, and is then stifled in finite additional time.

    Thus \( L_{x}:=\sup\{t\ge0:\eta_{t}(x)=1\}, \) with $L_{x}=0$ if $x$ never becomes a spreader, is almost surely
    finite. For finite $K\subset V$, set $\tau_{K}:=\max_{x\in K}L_{x}$. Then $\tau_{K}<\infty$ almost surely, and no
    vertex of $K$ is a spreader after time $\tau_{K}$.
\end{proof}

\subsection{Subcritical Extinction via Branching-Process Domination}
\label{subsec:proof-subcritical-phase}

We prove Theorem~\ref{thm:MT-positive-subcritical} by comparing the transmission genealogy with a Galton--Watson
process. For random variables $U$ and $V$, we write $U\preceq V$ when $U$ is stochastically dominated by $V$.

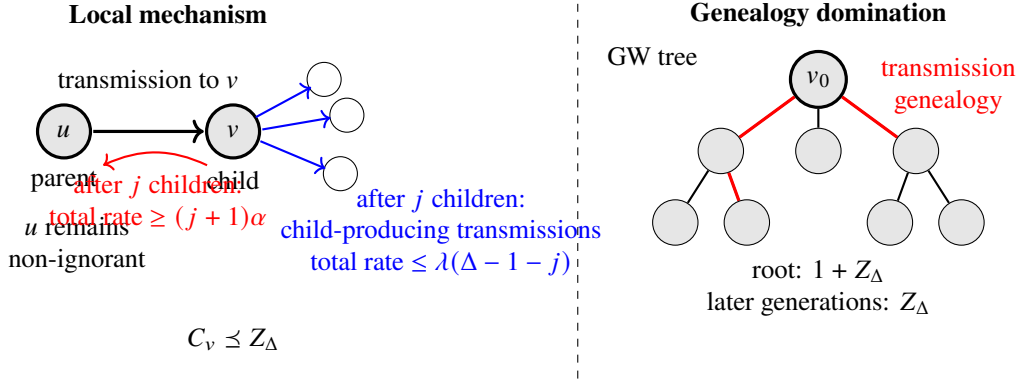
\begin{figure}[ht]
    \centering
    \begin{tikzpicture}[scale=0.86, every node/.style={font=\small}]
        \node at (-3.2,3.0) {\textbf{Local mechanism}};
        \node[
            circle,
            draw=black,
            fill=gray!20,
            very thick,
            minimum size=7mm
        ] (u) at (-4.8,1.2) {$u$};
        \node[
            circle,
            draw=black,
            fill=gray!20,
            very thick,
            minimum size=7mm
        ] (v) at (-2.2,1.2) {$v$};
        \node[below=6pt] at (-4.8,1.0) {parent};
        \node[below=6pt] at (-2.2,1.0) {child};

        \draw[->, very thick] (-4.35,1.2) -- (-2.65,1.2);
        \node[above] at (-3.5,1.7) {transmission to $v$};
        \node[align=center] at (-4.6,-0.55) {$u$ remains\\ non-ignorant};

        \draw[->, thick, red] (-2.6,0.68) to[bend right=25] (-4.2,0.68);
        \node[red, align=center]
        at
        (-3.35,0.1)
        {after $j$ children:\\ total rate $\ge(j+1)\alpha$};

        \node[circle, draw=black, minimum size=4.5mm] (x1) at (-0.8,2.0) {};
        \node[circle, draw=black, minimum size=4.5mm]
        (x2)
        at
        (-0.55,0.55)
        {};
        \node[circle, draw=black, minimum size=4.5mm]
        (x3)
        at
        (-0.45,1.45)
        {};

        \draw[->, thick, blue] (-1.85,1.42) -- (-1.0,1.88);
        \draw[->, thick, blue] (-1.78,1.05) -- (-0.78,0.63);
        \draw[->, thick, blue] (-1.75,1.23) -- (-0.7,1.41);

        \node[blue, align=center]
        at
        (1.0,-0.35)
        {after $j$ children:\\ child-producing transmissions\\ total rate $\le \lambda(\Delta-1-j)$};

        \node at (-2.2,-2.0) {$C_{v}\preceq Z_{\Delta}$};

        \draw[dashed] (3.1,-2.6) -- (3.1,3.25);

        \node at (6.8,3.0) {\textbf{Genealogy domination}};
        \node[
            circle,
            draw=black,
            fill=gray!20,
            very thick,
            minimum size=7mm
        ] (r) at (6.8,2.0) {$v_{0}$};

        \node[circle, draw=black, fill=gray!20, minimum size=6mm]
        (a1)
        at
        (5.3,0.9)
        {};
        \node[circle, draw=black, fill=gray!20, minimum size=6mm]
        (a2)
        at
        (6.8,0.9)
        {};
        \node[circle, draw=black, fill=gray!20, minimum size=6mm]
        (a3)
        at
        (8.3,0.9)
        {};

        \node[circle, draw=black, fill=gray!20, minimum size=6mm]
        (b1)
        at
        (4.6,-0.2)
        {};
        \node[circle, draw=black, fill=gray!20, minimum size=6mm]
        (b2)
        at
        (5.7,-0.2)
        {};
        \node[circle, draw=black, fill=gray!20, minimum size=6mm]
        (b3)
        at
        (7.9,-0.2)
        {};
        \node[circle, draw=black, fill=gray!20, minimum size=6mm]
        (b4)
        at
        (9.0,-0.2)
        {};

        \draw[thick] (r) -- (a1);
        \draw[thick] (r) -- (a2);
        \draw[thick] (r) -- (a3);
        \draw[thick] (a1) -- (b1);
        \draw[thick] (a1) -- (b2);
        \draw[thick] (a3) -- (b3);
        \draw[thick] (a3) -- (b4);

        \draw[red, very thick] (r) -- (a1);
        \draw[red, very thick] (r) -- (a3);
        \draw[red, very thick] (a1) -- (b2);

        \node[red, align=center] at (8.8,1.9) {transmission\\ genealogy};
        \node at (4.25,2.35) {GW tree};

        \node[align=center]
        at
        (6.8,-1.2)
        {root: $1+Z_{\Delta}$\\ later generations: $Z_{\Delta}$};
    \end{tikzpicture}
    \caption{Galton--Watson domination in the proof of Theorem~\ref{thm:MT-positive-subcritical}.
        Left: the offspring of a non-initial spreader are dominated by $Z_{\Delta}$.
        Right: the transmission genealogy is coupled with a dominating
        Galton--Watson tree whose root has offspring distribution \(1+Z_\Delta\).}
    \label{fig:subcritical-branching-domination}
\end{figure}

\begin{proof}[Proof of Theorem~\ref{thm:MT-positive-subcritical}]
    We first consider the process started from a single spreader $v_{0}$. Let $S_{t}=\{x\in V:\eta_{t}(x)=1\}$. Since almost surely no two Poisson
    marks occur simultaneously, every non-initial spreader has a unique parent in the transmission genealogy. Let $Z_{\Delta}$ be the random variable
    supported on $\{0,\ldots,\Delta-1\}$ whose tail is given by
    \begin{equation}
        \Prob(Z_{\Delta}\ge k) = \prod_{j=0}^{k-1}\frac{\Delta-1-j}{\Delta-1-j+(j+1)\alpha/\lambda}
        , \qquad 1\le k\le\Delta-1. \label{eq:subcritical-offspring-tail}
    \end{equation}
    Then
    \[
        \E[Z_{\Delta}] = \sum_{k=1}^{\Delta-1}\Prob(Z_{\Delta}\ge k) = m_{\Delta}
        (\alpha/\lambda).
    \]

    Consider first a non-initial spreader $v$ (see Figure~\ref{fig:subcritical-branching-domination}), let $u$ be its
    parent, and let $C_{v}$ denote the number of children produced by $v$. From the birth time of $v$ onward, $u$
    remains non-ignorant. Let $\sigma_j$ be the time at which $v$ produces its $j$th child, with $\sigma_0$ equal to
    the birth time of $v$, and let $(\mathcal{T}_{t} )_{t\ge0}$ be the natural filtration.

    Fix $0\le j\le\Delta-2$ and work on $\{C_{v}\ge j\}$. Set
    \[
        A_{j}:=\{y\sim v:\eta_{\sigma_j}(y)=0\}, \qquad B_{j}:=\{y\sim v:\eta
        _{\sigma_j}(y)\ne0\}.
    \]
    Then $|A_{j}|\le\Delta-1-j$ and $|B_{j}|\ge j+1$. Since non-ignorance is absorbing, for $v$ to produce a further
    child an infection mark with source $v$ must occur on an edge $(v,y)$ with $y\in A_{j}$, whereas a stifling mark on
    any edge $(v,y)$ with $y\in B_{j}$ turns $v$ into a stifler. Define
    \[
        T_{j}^{\mathrm{inf}}:= \min_{y\in A_j}\inf\bigl\{ s>0: N_{\mathrm{inf}}
            ^{v,y}(\sigma_{j}+s) > N_{\mathrm{inf}}^{v,y}(\sigma_{j}) \bigr\},
    \]
    with the minimum over the empty set interpreted as infinity, and define $T _{j}^{\mathrm{stif}}$ analogously using
    $N_{\mathrm{stif}}^{v,y}$ and $B_{j}$. Then
    \[
        \{C_{v}\ge j+1\} \subseteq \{T_{j}^{\mathrm{inf}}<T_{j}^{\mathrm{stif}}
        \}.
    \]
    By the strong Markov property at $\sigma_{j}$, conditionally on $\mathcal{T} _{\sigma_j}$ the two waiting times are
    independent exponentials with rates $\lambda|A_{j}|$ and $\alpha|B_{j}|$. Consequently, on $\{C_{v}\ge j\}$,
    \[
        \Prob(C_{v}\ge j+1\mid\mathcal{T}_{\sigma_j}) \le \frac{\lambda|A_{j}|}{\lambda|A_{j}|+\alpha|B_{j}|}
        \le \frac{\Delta-1-j}{\Delta-1-j+(j+1)\alpha/\lambda}.
    \]
    Multiplying by \(\mathbf 1_{\{C_v\ge j\}}\) and taking expectations yields
    \[
        \Prob(C_{v}\ge j+1) \le \frac{\Delta-1-j}{\Delta-1-j+(j+1)\alpha/\lambda}
        \Prob(C_{v}\ge j).
    \]
    Iterating this inequality for $j=0,\ldots,k-1$ and using \eqref{eq:subcritical-offspring-tail} gives
    \[
        \Prob(C_{v}\ge k)\le\Prob(Z_{\Delta}\ge k), \qquad 1\le k\le\Delta-1,
    \]
    and therefore $C_{v}\preceq Z_{\Delta}$.

    Now consider the initial spreader $v_{0}$. If $v_{0}$ never produces a child, then $C_{v_0}=0$. Otherwise, after
    its first successful transmission, it has a permanently non-ignorant neighbor, and the preceding domination applies
    from that time onward. Thus \( C_{v_0}\preceq Z_{\Delta}^{(0)}:=1+Z_{\Delta}. \)

    We now dominate the transmission genealogy by a Galton--Watson process with a modified root: the root has offspring
    distribution $Z_{\Delta}^{(0)}$, whereas every later individual has offspring distribution $Z_{\Delta}$. More
    explicitly, we explore the rumor genealogy generation by generation using deferred revelation of the Harris
    construction. When a spreader \(v\) is reached, the portions of the Poisson processes indexed by \((v,z)\), \(z\sim
    v\), after the birth time of \(v\) are revealed for the first time. Conditional on the previously explored
    genealogy, these portions have their original independent Poisson laws. Consequently, the conditional law of
    $C_{v}$ is stochastically dominated by $Z_{\Delta}^{(0)}$ when $v=v_{0}$, and by $Z_{\Delta}$ otherwise. We may
    therefore couple $C_{v}$ with a fresh offspring variable $Y_{v}$, having the corresponding law, so that $C_{v}\le
    Y_{v}$, and assign the children of $v$ injectively to the offspring of the corresponding tree vertex. Iterating
    this construction along the exploration yields a coupling in which the Galton--Watson tree dominates the rumor
    genealogy.

    If $m_{\Delta}(\alpha/\lambda)<1$, the transmission genealogy from a finite initial seed $S_0$ is dominated by a
    forest of $|S_0|$ Galton--Watson trees that are subcritical after their roots. This forest has finite total
    population almost surely, so only finitely many spreaders are created.

    Every non-initial spreader has a permanently non-ignorant parent and is therefore stifled at rate at least
    $\alpha$. Each initial spreader either is stifled before its first transmission or makes such a transmission in
    finite time almost surely, after which it also has a permanently non-ignorant neighbor. Hence every spreader leaves
    the spreader state in finite time almost surely, and the process becomes extinct in finite time.
\end{proof}

\subsection{Survival via Site Percolation}
\label{subsec:survival-small-stifling}

We prove survival for sufficiently small stifling, together with a positive-density conclusion, by comparison with
supercritical Bernoulli site percolation. Identify $o$ with the identity of $\Gamma$, and write \(
\rho:=\dfrac{\alpha}{\lambda}. \) Let $p_{c}^{\mathrm{site}}(G)$ denote the critical probability for Bernoulli site
percolation on $G$.

Since the only possible transitions are
\[
    \text{ignorant}\longrightarrow\text{spreader}\longrightarrow\text{stifler},
\]
each vertex enters the spreader state at most once. Set
\[
    \tau_{x}:=\inf\{t\ge0:x\text{ is a spreader at time }t\},
\]
with $\tau_{x}=\infty$ if $x$ never becomes a spreader.

We use an equivalent graphical construction in which the Poisson processes with source $x$ are indexed by the time
elapsed since $x$ became a spreader. For every ordered edge $(x,y)$, let \(\mathcal{N}_{\mathrm{inf}}^{x,y},\)
\(\mathcal{N}_{\mathrm{stif}}^{x,y}\) be independent Poisson processes of respective rates $\lambda$ and $\alpha$, with
all these processes mutually independent. The Poisson processes with source $x$ are activated at time $\tau_x$; while
$x$ remains a spreader, a mark at source age $s$ is read at real time $\tau_x+s$. An infection mark on $(x,y)$ turns
$y$ into a spreader if $y$ is ignorant, whereas a stifling mark on $(x,y)$ turns $x$ into a stifler if $y$ is
non-ignorant.

We construct the process recursively by revealing source processes only as needed. Let $\sigma_1=0$, $x_1=o$, and let
\( \sigma_1<\sigma_2<\cdots \) be the successive times at which new sources are activated, with $x_1,x_2,\ldots$ the
corresponding vertices. If fewer than $n$ sources are activated, set $\sigma_n=\infty$, and put \(
\sigma_\infty:=\lim_{n\to\infty}\sigma_n. \) Reveal the entire outgoing source family of a vertex only when that vertex
first becomes a spreader. Suppose that $x_1,\ldots,x_n$ have been activated. The identity of the next activated source
is determined by the families revealed at stages at most $n$; in particular, for every $z\in V$, the event
$\{x_{n+1}=z\}$ is measurable with respect to the information revealed up to that stage. Hence, conditional on the
exploration so far, every unrevealed source family retains its original law and is independent of all previously
revealed families.

Since the good/bad status of $x$ depends only on its outgoing source family, the statuses revealed along the
exploration are independent Bernoulli variables with the required law.

For $t<\sigma_\infty$, let $N_t$ denote the number of sources activated by time $t$. Here $N_0=1$. When $N_{t-}=n$, at
most $n\Delta$ infection clocks can create a new source, each at rate $\lambda$. Hence new sources are activated at
total rate at most $\lambda\Delta n$.

Therefore $(N_t)_{t<\sigma_\infty}$ is dominated, up to $\sigma_\infty$, by a Yule process with per-particle birth rate
$\lambda\Delta$. Since the latter is nonexplosive, $\sigma_\infty=\infty$ almost surely. The resulting nonexplosive
process has the transition rates of the Maki--Thompson dynamics and hence the same law as the usual graphical
construction.

For every ordered edge $(x,y)$, put
\[
    T_{\mathrm{inf}}^{x,y}
    :=
    \inf\{s>0:\mathcal{N}_{\mathrm{inf}}^{x,y}(s)\ge1\},
    \qquad
    T_{\mathrm{stif}}^{x,y}
    :=
    \inf\{s>0:\mathcal{N}_{\mathrm{stif}}^{x,y}(s)\ge1\},
\]
with $T_{\mathrm{stif}}^{x,y}=\infty$ when $\alpha=0$. We call $x$ \emph{good} if
\begin{equation}
    \max_{y\sim x}T_{\mathrm{inf}}^{x,y}
    <
    \min_{y\sim x}T_{\mathrm{stif}}^{x,y}.
    \label{eq:good-vertex-small-alpha}
\end{equation}
The event that $x$ is good depends only on the Poisson processes with source
$x$. Since $G$ is regular, the field
\(
\bigl(\mathbf{1}_{\{x\text{ is good}\}}\bigr)_{x\in V}
\)
is i.i.d.\ Bernoulli.

Suppose that the process starts with $o$ as its only spreader and every other vertex ignorant. If a good vertex $x$
becomes a spreader at time $\tau_x$, then every neighbor of $x$ has entered the spreader state by time
\[
    \tau_x+\max_{y\sim x}T_{\mathrm{inf}}^{x,y}.
\]
Indeed, only a stifling mark whose source is $x$ can turn $x$ into a stifler, and, since $x$ is good, no such mark
  occurs before this time. Thus $x$ remains a spreader until it has produced an infection mark toward every neighbor.
  When the corresponding mark occurs, each neighbor either becomes a spreader or has already passed through the
  spreader state.

Set
\[
    U_x:=\max_{y\sim x}T_{\mathrm{inf}}^{x,y}.
\]
Then $x$ is good precisely when
\[
    U_x<\min_{y\sim x}T_{\mathrm{stif}}^{x,y}.
\]
The latter minimum is $\operatorname{Exp}(\alpha\Delta)$, independent of $U_x$, with the convention that it equals
  $+\infty$ when $\alpha=0$. Hence
\[
    \mathbb P(x\text{ is good})
    =
    \mathbb E\bigl[e^{-\alpha\Delta U_x}\bigr].
\]
By R\'enyi's representation of exponential order statistics,
\[
    U_x\stackrel{d}{=}\sum_{j=1}^{\Delta}G_j,
    \qquad
    G_j\sim\operatorname{Exp}(j\lambda)
\]
independently. Therefore
\begin{equation}
    \mathbb P(x\text{ is good})
    =
    \prod_{j=1}^{\Delta}
    \frac{j\lambda}{j\lambda+\alpha\Delta}
    =
    \prod_{j=1}^{\Delta}\frac{j}{j+\Delta\rho}
    =
    q_\Delta(\rho).
    \label{eq:good-vertex-probability}
\end{equation}
In particular, $q_\Delta(0)=1$.
Write
\[
    \mathcal{O}:=\{x\in V:x\text{ is good}\},
\]
and let $\mathcal{O}_o$ be the connected component of $o$ in the subgraph induced by $\mathcal{O}$, with
  $\mathcal{O}_o=\varnothing$ if $o\notin\mathcal{O}$.

\begin{proof}[Proof of Theorem~\ref{thm:MT-survival}]
    By \cite[Theorem~1.1]{PanagiotisSevero}, there exists a universal
    $\varepsilon_0>0$ such that
    \(
    p_c^{\mathrm{site}}(G)\le1-\varepsilon_0
    \)
    for every Cayley graph of superlinear growth. Set
    \[
        c_0:=\log\frac{1}{1-\varepsilon_0}.
    \]
    Using \eqref{eq:good-vertex-probability} and $\log(1+u)\le u$,
    \[
        \log\frac{1}{q_\Delta(\rho)}
        =
        \sum_{j=1}^{\Delta}
        \log\left(1+\frac{\Delta\rho}{j}\right)
        \le
        \Delta\rho\sum_{j=1}^{\Delta}\frac{1}{j}
        \le
        \Delta\rho(1+\log\Delta).
    \]
    Thus \eqref{eq:main-uniform-survival-condition} implies
    \[
        q_\Delta(\rho)
        >
        e^{-c_0}
        =
        1-\varepsilon_0
        \ge
        p_c^{\mathrm{site}}(G).
    \]

    For the remaining conclusions, we assume only that \(q:=q_\Delta(\rho)>p_c^{\mathrm{site}}(G).\) The set
    $\mathcal{O}$ has the law of Bernoulli site percolation with parameter $q$. Therefore
    \[
        \mathbb{P}(|\mathcal{O}_o|=\infty)
        =
        \theta_G(q)
        >
        0.
    \]

    We claim that
    \begin{equation}
        \mathcal{O}_o\subseteq\mathcal{A}_\infty.
        \label{eq:open-cluster-contained}
    \end{equation}
    Indeed, let $x\in\mathcal{O}_o$ and let
    \(
    o=x_0,x_1,\ldots,x_m=x
    \)
    be an open path, so that every $x_i$ is good. Since $o$ is initially a 	spreader, $x_0$ enters the spreader state; and if $x_i$ enters the spreader state, then the preceding observation shows that $x_{i+1}$ does as well. Induction proves
    \eqref{eq:open-cluster-contained}. If 	$o\notin\mathcal{O}$, there is nothing to prove.

    On $\{|\mathcal{O}_o|=\infty\}$, infinitely many vertices therefore enter the spreader state. If the spreader set
    became empty at some finite time, no new spreader could appear afterwards, while nonexplosion implies that only
    finitely many vertices could have entered the spreader state before that time. This is a contradiction. Hence the
    rumor survives globally on $\{|\mathcal{O}_o|=\infty\}$.

    Finally, suppose that $\Gamma$ has polynomial growth. Since $\mathcal{O}$ has law $\mathbb{P}_q$, all
    $\mathbb{P}_q$-almost-sure percolation statements below hold $\mathbb{P}$-almost surely for $\mathcal{O}$. The
    graph $G$ is amenable and, since $q>p_c^{\mathrm{site}}(G)$, an infinite cluster exists almost surely. It is unique
    by \cite[Theorem~7.6]{LyonsPeres}; denote it by $\mathcal {C}_\infty$. Thus, almost surely on
    $\{|\mathcal{O}_o|=\infty\}$,
    \[
        \mathcal {C}_\infty
        =
        \mathcal{O}_o
        \subseteq
        \mathcal{A}_\infty.
    \]

    Let $\Gamma$ act on percolation configurations by \( (\sigma_g\omega)(x):=\omega(g^{-1}x), \) and set \(
    f:=\mathbf{1}_{\{o\in\mathcal {C}_\infty\}}. \) The Bernoulli product action is ergodic. Since left translations
    are graph automorphisms and the infinite cluster is almost surely unique, \( \mathcal {C}_\infty(\sigma_g\omega) =
    g\,\mathcal {C}_\infty(\omega), \) which gives \( \mathbf{1}_{\{g\in\mathcal {C}_\infty\}}(\omega) =
    f(\sigma_{g^{-1}}\omega). \)

    Write $B_n:=B_G(o,n)$ and take
    \[
        U_j:=S\cup\{e\},
        \qquad j\ge0.
    \]
    Thus, with $K=K'=S\cup\{e\}$, the inclusions
    \[
        K\subset U_j\subset K'
    \]
    hold for every $j$, as required in \cite[Theorem~13]{Tessera}. Moreover,
    \[
        N_n:=U_0U_1\cdots U_n=B_{n+1},
        \qquad n\ge0.
    \]
    Therefore that theorem yields
    \[
        \frac{|\mathcal C_\infty\cap B_{n+1}|}{|B_{n+1}|}
        =
        \frac{1}{|B_{n+1}|}
        \sum_{g\in B_{n+1}}
        f(\sigma_{g^{-1}}\omega)
        \xrightarrow[n\to\infty]{}
        \mathbb E_q[f]
        =
        \theta_G(q)
        \qquad\text{almost surely.}
    \]
    Hence the same limit holds along $(B_n)_{n\ge1}$. On $\{|\mathcal O_o|=\infty\}$,
    \[
        \mathcal C_\infty=\mathcal O_o\subseteq\mathcal A_\infty,
    \]
    and therefore
    \[
        \liminf_{n\to\infty}
        \frac{|\mathcal A_\infty\cap B_n|}{|B_n|}
        \ge \theta_G(q)
        \qquad\text{almost surely.}
    \]
    Since
    \[
        \mathbb P(|\mathcal O_o|=\infty)=\theta_G(q),
    \]
    this proves \eqref{eq:main-positive-density}.
\end{proof}

\subsection{Surface-Order Active Propagation}
\label{subsec:surface-order-active-propagation}

We first record the geometric and percolation estimates needed in the proof of
Theorem~\ref{thm:surface-order-active-propagation}.

A finite edge set $C\subset E$ is called an \emph{edge cycle} if every vertex has even degree in the finite graph
$(V,C)$.

Following \cite[Definition~1]{GorskiProcaccia}, a graph $G$ is called \emph{$\Lambda$-simply connected}, for
$\Lambda\ge0$, if every finite edge cycle $C$ can be written as $C=C_{1}\oplus\cdots\oplus C_{m}$, where
$C_{1},\ldots,C_{m}$ are edge cycles satisfying $\operatorname{diam}_{G}(V(C_{i}))\le\Lambda$ for every $i$. The graph
is called \emph{coarsely simply connected} if it is $\Lambda$-simply connected for some finite $\Lambda$.

\begin{proposition}[Coarse simple connectivity]
    \label{prop:G-is-coarsely-simply-connected} Let $G=\operatorname{Cay}(\Gamma,S)$ be a Cayley graph of polynomial growth. Then $G$ is $\Lambda$-simply
    connected for some $\Lambda=\Lambda(G)\ge1$.
\end{proposition}

\begin{proof}
    By Gromov's theorem, $\Gamma$ contains a finitely generated nilpotent subgroup of finite index. Such a subgroup is
    finitely presented \cite[Propositions~13.75 and~13.84]{DrutuKapovich}, and finite extensions of finitely presented
    groups are finitely presented. Hence $\Gamma$ is finitely presented. Its Cayley graphs are therefore coarsely
    simply connected \cite[Corollaries~9.5 and~9.36]{DrutuKapovich}; the latter citation also shows that the property
    is independent of the finite generating set, up to the value of $\Lambda$. Enlarging $\Lambda$ if necessary, we may
    assume that $\Lambda\ge1$.
\end{proof}

For a site-percolation configuration, an open path is a path all of whose vertices are open. Let $d_{\mathrm{ch}}(x,y)$
denote the minimum number of edges in an open path from $x$ to $y$, with $d_{\mathrm{ch}}(x,y)=\infty$ if no such path
exists.

\begin{proposition}[High-density chemical-distance estimate]
    \label{prop:site-chemical-distance} Let $G$ be an infinite connected graph 	of bounded degree that is $\Lambda$-simply connected for some $\Lambda\ge
        1$, and set
    \[
        N_{\Lambda}:=\sup_{z\in V}|B_{G}(z,\Lambda)|,
    \]
    so that $2\le N_{\Lambda}<\infty$. For Bernoulli site percolation with parameter
    \[
        p>1-(N_{\Lambda}-1)^{-1},
    \]
    there exist $K<\infty$ and $c,C>0$ such that, for all $x,y\in V$,
    \begin{equation}
        \mathbb{P}_{p}\bigl( x\longleftrightarrow y,\; d_{\mathrm{ch}}(x,y)>K
        d_{G}(x,y) \bigr) \le Ce^{-c d_G(x,y)}. \label{eq:site-chemical-distance}
    \end{equation}
\end{proposition}

\begin{proof}
    The bond analogue is the first part of \cite[Theorem~2]{GorskiProcaccia}. We give a short proof for site
    percolation; its branching estimate yields the explicit threshold above.

    Let \(G_{\Lambda}\) be the graph with vertex set \(V\) in which two distinct vertices are adjacent whenever their
    \(G\)-distance is at most \(\Lambda\). Its maximum degree is at most \(N_{\Lambda}-1\). Fix \(x,y\in V\), set
    \(m:=d_G(x,y)\), and let \(\beta\) be a \(G\)-geodesic from \(x\) to \(y\). Call a vertex closed if it is not open.
    Let \(F_\beta\) be the union of those connected components of the subgraph of \(G_\Lambda\) induced by the closed
    vertices that intersect \(\beta\).

    Explore sequentially the closed $G_{\Lambda}$-components that meet $\beta$, starting from the closed vertices of
    $\beta$. There are at most $m+1$ initial vertices. Conditionally on the exploration so far, each explored vertex
    has at most $N_{\Lambda}-1$ unexplored $G_{\Lambda}$-neighbors, each closed with probability $1-p$. Hence
    $|F_\beta|$ is stochastically dominated by
    \[
        \sum_{i=1}^{m+1}\mathcal P_i,
    \]
    where the $\mathcal P_i$ are independent copies of the total progeny of a Galton--Watson process with offspring
    distribution
    \[
        Z\sim\operatorname{Binomial}(N_\Lambda-1,1-p).
    \]
    Since $\mathbb E[Z]<1$ and $Z$ is bounded, $\mathcal P_1$ has a finite exponential moment. A Chernoff bound
    therefore gives constants $A<\infty$ and $c_{\mathrm{ch}}>0$ such that
    \begin{equation}
        \mathbb{P}_{p}\bigl(|F_{\beta}|>A(m+1)\bigr)
        \le e^{-c_{\mathrm{ch}}(m+1)}.
        \label{eq:closed-obstacle-tail}
    \end{equation}

    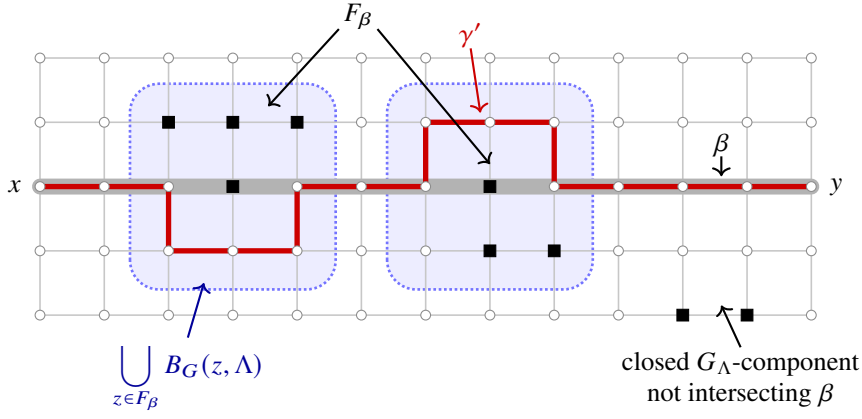
\begin{figure}[ht]
        \centering
        \begin{tikzpicture}[
                scale=0.85,
                every node/.style={font=\small},
                open/.style={circle, draw=black!45, fill=white, inner sep=1.3pt},
                closed/.style={rectangle, draw=black, fill=black, inner sep=2.2pt},
                halo/.style={
                        draw=blue!55, densely dotted, line width=1pt,
                        fill=blue!7, rounded corners=11pt
                    }
            ]

            \path[halo] (1.4,-1.6) rectangle (4.6,1.6);
            \path[halo] (5.4,-1.6) rectangle (8.6,1.6);

            \foreach \i in {0,...,11}
            \foreach \j in {-2,...,2}
                { \draw[black!22, line width=0.6pt] (\i,\j) -- (\i+1,\j); }
            \foreach \i in {0,...,12}
            \foreach \j in {-2,...,1}
                { \draw[black!22, line width=0.6pt] (\i,\j) -- (\i,\j+1); }

            \draw[black!30, line width=6pt, line cap=round]
            (0,0) -- (12,0);

            \draw[red!80!black, line width=2pt, line join=round, line cap=round]
            (0,0) -- (2,0) -- (2,-1) -- (4,-1) -- (4,0) --
            (6,0) -- (6,1) -- (8,1) -- (8,0) -- (12,0);

            \foreach \i in {0,...,12}
            \foreach \j in {-2,...,2}
                { \node[open] at (\i,\j) {}; }

            \foreach \p in {(2,1),(3,1),(4,1),(3,0)}
                { \node[closed] at \p {}; }
            \foreach \p in {(7,0),(7,-1),(8,-1)}
                { \node[closed] at \p {}; }

            \foreach \p in {(10,-2),(11,-2)}
                { \node[closed] at \p {}; }

            \node[left=3pt]  at (0,0)  {$x$};
            \node[right=3pt] at (12,0) {$y$};

            \node at (4.95,2.65) {$F_\beta$};
            \draw[->, thick] (4.65,2.45) -- (3.55,1.25);
            \draw[->, thick] (5.25,2.45) -- (7.0,0.35);

            \node at (10.6,0.65) {$\beta$};
            \draw[->, thick] (10.6,0.47) -- (10.6,0.16);

            \node[red!80!black] at (6.7,2.35) {$\gamma'$};
            \draw[->, thick, red!80!black]
            (6.7,2.15) -- (6.85,1.15);

            \node[blue!60!black, align=center] at (2.3,-3.0)
            {$\displaystyle \bigcup_{z\in F_\beta} B_G(z,\Lambda)$};
            \draw[->, thick, blue!60!black]
            (2.3,-2.35) -- (2.55,-1.5);

            \node[align=center] at (10.9,-3.0)
            {closed $G_\Lambda$-component\\
                not intersecting $\beta$};
            \draw[->, thick]
            (10.9,-2.5) -- (10.6,-1.85);

        \end{tikzpicture}

        \caption[Chemical-distance estimate]{Chemical-distance estimate, drawn on
            $\mathbb Z^2$ with $\Lambda=2$. The geodesic $\beta$ encounters closed
            $G_\Lambda$-components forming $F_\beta$. Coarse simple connectivity
            yields an open detour $\gamma'$ (red) that avoids $F_\beta$ and is
            contained in
            $\beta\cup\bigcup_{z\in F_\beta}B_G(z,\Lambda)$.
            Thus its extra length is controlled by $|F_\beta|$. Closed components
            not intersecting $\beta$ play no role.}
        \label{fig:chemical-distance-detour}
    \end{figure}

    On $\{x\longleftrightarrow y\}$, choose an open path $\gamma$ from $x$ to $y$. Since $F_{\beta}$ consists of closed
    vertices, $\gamma\cap F_{\beta}=\varnothing$. On the event $\{|F_{\beta}|\le A(m+1) \}$, apply
    \cite[Lemma~2.8]{GorskiProcaccia} with forbidden vertex set $F_{\beta}$. It yields a path $\gamma'$ from $x$ to $y$
    contained in
    \[
        \left( \beta\cup \bigcup_{z\in F_\beta}B_{G}(z,\Lambda) \right) \setminus
        F_{\beta}.
    \]
    The geometric detour used here is illustrated in Figure~\ref{fig:chemical-distance-detour}.
    Every vertex in this region is open. Indeed, a closed vertex of $\beta$ belongs to $F_{\beta}$, while a closed
    vertex within distance $\Lambda$ of $F_{\beta}$ belongs to the same closed $G_{\Lambda}$-component and hence also
    to $F_{\beta}$.

    After loop erasure, $|\gamma'|\le m+1+N_{\Lambda}|F_{\beta}|$. Consequently, for $m\ge1$,
    \[
        d_{\mathrm{ch}}(x,y) \le (1+AN_{\Lambda})(m+1) \le 2(1+AN_{\Lambda})m
    \]
    on $\{|F_{\beta}|\le A(m+1)\}$. Taking $K:=2(1+AN_{\Lambda})$ and using \eqref{eq:closed-obstacle-tail} gives
    \eqref{eq:site-chemical-distance} for $m\ge1$; the case $x=y$ is immediate.
\end{proof}

\begin{proof}[Proof of Theorem~\ref{thm:surface-order-active-propagation}]
    Write $\rho:=\alpha/\lambda$. Fix $C_{G}<\infty$ such that
    \[
        |B_{G}(o,r)|\le C_{G}(1+r)^{D}, \qquad r\ge0,
    \]
    and let $v_{G}\in(0,\infty)$ be defined by
    \[
        \frac{|B_{G}(o,n)|}{n^{D}}\longrightarrow v_{G}.
    \]

    Throughout the proof, we use the graphical construction introduced above.

    For $x\in V$, set \(\Psi_x:=\min_{y\sim x}T_{\mathrm{stif}}^{x,y}, \) with $\Psi_x=\infty$ when $\alpha=0$. For
    $L,b>0$, call $x$ \emph{$(L,b)$-good} if
    \begin{equation}
        U_{x}\le L, \qquad \Psi_x\ge U_{x}+b. \label{eq:robust-good-vertex}
    \end{equation}
    The parameter of this i.i.d.\ Bernoulli field is
    \begin{equation}
        p_{L,b}= e^{-\alpha\Delta b}\mathbb{E}\left[ e^{-\alpha\Delta U_x}\mathbf{1}
            _{\{U_x\le L\}}\right]. \label{eq:robust-good-probability}
    \end{equation}
    Since this event depends only on the outgoing source family of $x$, the
    $(L,b)$-good-vertex field is i.i.d.\ Bernoulli.
    For fixed $b>0$, we have $p_{L,b}\uparrow e^{-\alpha\Delta b}q_{\Delta}(\rho)$ as $L\to\infty$. Since $q_{\Delta}(\rho)>\widetilde p_{G}$, we may first choose $b>0$
    sufficiently small and then $L<\infty$ sufficiently large that
    \[
        p_{L,b}> \widetilde p_{G}\ge p_{c}^{\mathrm{site}}(G).
    \]

    Let $\mathcal{O}^{L,b}$ be the set of $(L,b)$-good vertices and let $\mathcal C_o^{L,b}$ be the component of $o$ in
    the subgraph induced by $\mathcal{O} ^{L,b}$, with $\mathcal C_o^{L,b}:=\varnothing$ when
    $o\notin\mathcal{O}^{L,b}$. Set $\mathcal{I}_{L,b}:=\{|\mathcal C_o^{L,b} |=\infty\}$. Then
    $\mathbb{P}(\mathcal{I}_{L,b})=\theta_{G}(p_{L,b})>0$.

    Let $d_{\mathrm{ch}}$ denote chemical distance in $\mathcal{O}^{L,b}$. We claim that
    \begin{equation}
        \tau_{x}\le Ld_{\mathrm{ch}}(o,x), \qquad x\in\mathcal C_o^{L,b}. \label{eq:infection-time-via-chemical-distance}
    \end{equation}
    Indeed, on $\mathcal{I}_{L,b}$ we have $o\in\mathcal{O}^{L,b}$ and
    $\tau_{o}=0$. Let $o=x_{0},x_{1},\ldots,x_{m}=x$ be a path of $(L,b)$-good
    vertices. If $x_i$ becomes a spreader at time $\tau_{x_i}$, then, since $x_i$ is
    $(L,b)$-good, its infection mark toward $x_{i+1}$ occurs within $L$
    units of time and before any stifling mark with source $x_i$. Thus $x_i$
    is still a spreader when that infection mark occurs. At that time, $x_{i+1}$ either
    becomes a spreader or has already entered the spreader state, since an ignorant vertex can leave that state only by
    becoming a spreader. Hence $\tau_{x_{i+1}}\le\tau_{x_i}+L$, and induction proves
    \eqref{eq:infection-time-via-chemical-distance}.

    Fix $\varepsilon\in(0,1)$ and set
    \[
        r_{n}:=\lfloor(1-\varepsilon)n\rfloor, \qquad \mathcal{Q}_{n}^{\varepsilon}:= B_{G}(o,n)\setminus B_{G}(o,r_{n}).
    \]
    By the choice of \(L\) and \(b\) and \eqref{eq:main-active-percolation-threshold}, \( p_{L,b}>\widetilde
    p_G\ge1-(N_\Lambda-1)^{-1}. \) Hence Proposition~\ref{prop:site-chemical-distance} applies.

    Let $\mathcal{B}_{n}^{\varepsilon}$ be the event that there exists $x\in\mathcal{Q}_{n}^{\varepsilon}$ such that
    $o\longleftrightarrow x$ in $\mathcal{O}^{L,b}$ and $d_{\mathrm{ch}}(o,x)>Kd_{G}(o,x)$, where $K$ is the constant
    in \eqref{eq:site-chemical-distance}. Every $x\in\mathcal{Q} _{n}^{\varepsilon}$ satisfies $d_{G}(o,x)\ge
    r_{n}+1>(1-\varepsilon)n$, while polynomial growth gives $|\mathcal{Q}_{n}^{\varepsilon}|\le|B_{G}(o,n)|\le
    C_{G}(1+n)^{D}$. Thus, for some constant \(C'<\infty\), a union bound gives
    \[
        \sum_{n\ge1}\mathbb{P}(\mathcal{B}_{n}^{\varepsilon})
        \le C'\sum_{n\ge1}n^{D}e^{-c(1-\varepsilon)n}<\infty.
    \]
    By Borel--Cantelli, almost surely, for all sufficiently large $n$, $d_{\mathrm{ch}}(o,x)\le Kd_{G}(o,x)\le Kn$ for
    every $x\in\mathcal C_o^{L,b}\cap\mathcal{Q} _{n}^{\varepsilon}$. Therefore
    \begin{equation}
        \tau_{x}\le LKn \qquad \text{for every }x\in\mathcal C_o^{L,b}\cap\mathcal{Q}_{n}^{\varepsilon}\label{eq:linear-infection-time-annulus}
    \end{equation}
    for all sufficiently large $n$.

    Whenever $\tau_{x}<\infty$ and $x$ is $(L,b)$-good, no stifling mark with source $x$ occurs before source age
    $\Psi_x\ge U_{x}+b$. Thus
    \begin{equation}
        [\tau_{x},\tau_{x}+U_{x}+b) \subseteq \{t\ge0:x\in S_{t}\}. \label{eq:robust-active-window}
    \end{equation}
    In particular, every vertex of $\mathcal C_o^{L,b}$ remains a spreader for at least $b$ time units.

    Since \(p_{L,b}>p_c^{\mathrm{site}}(G)\), there is almost surely a unique infinite cluster \(\mathcal C_\infty\),
    and \(\mathcal C_\infty=\mathcal C_o^{L,b}\) on \(\mathcal I_{L,b}\). By the ergodic argument used in
    Subsection~\ref{subsec:survival-small-stifling},
    \[
        \frac{|\mathcal C_\infty\cap B_G(o,n)|}{|B_G(o,n)|}
        \longrightarrow \theta_G(p_{L,b})
    \]
    almost surely. Using $\mathcal{Q}_{n}^{\varepsilon}=B_{G}(o,n)\setminus B_{G}(o,r_{n})$, the preceding convergence
    and the volume asymptotics $|B_{G}(o,n)|/n^{D}\to v_{G}$ from \cite{pansu1983} give
    \begin{equation}
        \begin{aligned}
            \frac{\bigl|\mathcal C_o^{L,b} \cap \mathcal{Q}_n^\varepsilon\bigr|}{n^D}
             & = \frac{|\mathcal {C}_\infty \cap B_G(o,n)|}{n^D} - \frac{|\mathcal {C}_\infty \cap B_G(o,r_n)|}{n^D} \\
             & \longrightarrow v_G \theta_G(p_{L,b}) \bigl[1-(1-\varepsilon)^D\bigr]
        \end{aligned}
        \label{eq:robust-annular-density}
    \end{equation}
    almost surely on $\mathcal{I}_{L,b}$.

    Set
    \[
        \gamma_{\varepsilon}:= \frac{1}{2}v_{G}\theta_{G}(p_{L,b}) \bigl[1-(
            1-\varepsilon)^{D}\bigr].
    \]
    Then, almost surely on $\mathcal{I}_{L,b}$,
    \begin{equation}
        |\mathcal C_o^{L,b}\cap\mathcal{Q}_{n}^{\varepsilon}| \ge \gamma_{\varepsilon}
        n^{D}\label{eq:robust-annular-lower-bound}
    \end{equation}
    for all sufficiently large $n$.

    Choose $v_{+}:=LK+b+1$, which is independent of $\varepsilon$. For every sufficiently large $n$ and every
    $x\in\mathcal C_o^{L,b}\cap\mathcal{Q} _{n}^{\varepsilon}$, \eqref{eq:linear-infection-time-annulus} and
    \eqref{eq:robust-active-window} imply that $[\tau_{x},\tau_{x}+b)\subseteq[0,v_{+}n]$ and that $x\in S_{t}$
    throughout this interval. Summing over the vertices in the annulus and using \eqref{eq:robust-annular-lower-bound},
    we obtain
    \begin{equation}
        \int_{0}^{v_+n}|S_{t}\cap\mathcal{Q}_{n}^{\varepsilon}|\,dt \ge b\gamma
        _{\varepsilon}n^{D}\label{eq:annular-space-time-occupation}
    \end{equation}
    almost surely on $\mathcal{I}_{L,b}$, for all sufficiently large $n$.

    Fix $r\in\mathbb{N}$. If the rumor reaches a vertex outside $B_{G}(o,r)$ by time $ar$, then its ancestral line
    contains a self-avoiding path $o=x_{0},x_{1},\ldots,x_{m}$, with $m\ge r+1$. The path is self-avoiding because
    infection times increase strictly along it and each vertex becomes a spreader at most once.

    For a fixed path, let $E_{i}$ be the source age of the first infection mark on the ordered edge $(x_{i-1},x_{i})$.
    These ordered edges are distinct, so $E_{1},\ldots,E_{m}$ are independent $\operatorname{Exp}(\lambda)$ random
    variables. The age at which transmission occurs along an edge is at least the age of its first infection mark.
    Hence traversal of the path by time $ar$ implies $E_{1}+\cdots+E_{m}\le ar\le am$, and
    \[
        \mathbb{P}(E_{1}+\cdots+E_{m}\le am) \le \frac{(\lambda am)^{m}}{m!}\le
        (\mathrm{e}\lambda a)^{m}.
    \]
    The number of self-avoiding paths of length $m$ from $o$ is at most $\Delta(\Delta-1)^{m-1}$. Choose $a>0$ such
    that $a\le1$ and $\mathrm{e}\lambda a(\Delta-1)<1$. A union bound then gives constants $c_{1},C_{1}>0$ such that
    \begin{equation}
        \mathbb{P}\left( \tau_{x}\le ar \text{ for some }x\notin B_{G}(o,r) \right
        ) \le C_{1}e^{-c_1r}. \label{eq:no-fast-active-propagation}
    \end{equation}

    Set $v_{-}(\varepsilon):=a(1-\varepsilon)/2$. Since $a\le1$ and $v_{+}\ge1$, $0<v_{-}(\varepsilon)<1/2<v_{+}$. For
    all sufficiently large $n$, $r_{n}\ge(1-\varepsilon)n/2$ and $ar_{n}\ge v_{-}(\varepsilon) n$. Since
    $\sum_{n}C_{1}e^{-c_1r_n}<\infty$, Borel--Cantelli and \eqref{eq:no-fast-active-propagation} imply that
    \begin{equation}
        S_{t}\cap\mathcal{Q}_{n}^{\varepsilon}= \varnothing \qquad \text{for
            every }t\le v_{-}(\varepsilon)n \label{eq:no-early-annular-activity}
    \end{equation}
    almost surely, for all sufficiently large $n$.

    Write
    \[
        f_{n}(t):=|S_{t}\cap\mathcal{Q}_{n}^{\varepsilon}|, \qquad c_{\mathrm{occ}}
        (\varepsilon):=b\gamma_{\varepsilon}.
    \]
    Combining \eqref{eq:annular-space-time-occupation} and \eqref{eq:no-early-annular-activity}, we obtain
    \begin{equation}
        \int_{v_-(\varepsilon)n}^{v_+n}f_{n}(t)\,dt \ge c_{\mathrm{occ}}(\varepsilon
        )n^{D}\label{eq:annular-occupation-final-window}
    \end{equation}
    almost surely on $\mathcal{I}_{L,b}$, for all sufficiently large $n$.

    Polynomial growth gives a constant $C_{2}<\infty$ such that $|\mathcal{Q}_{n}^{\varepsilon}|\le|B_{G}(o,n)|\le
    C_{2}n^{D}$ for all $n$. Set
    \[
        \delta_{\mathrm{act}}(\varepsilon) := \frac{c_{\mathrm{occ}}(\varepsilon)}{2\bigl(v_{+}-v_{-}(\varepsilon)\bigr)}
        , \qquad \ell_{\mathrm{act}}(\varepsilon) := \frac{c_{\mathrm{occ}}(\varepsilon)}{2C_{2}}
        ,
    \]
    and let
    \[
        J_n:=
        \left\{
        t\in[v_-(\varepsilon)n,v_+n]:
        f_n(t)\ge\delta_{\mathrm{act}}(\varepsilon)n^{D-1}
        \right\}.
    \]
    The contribution of $[v_{-}(\varepsilon)n,v_{+}n]\setminus J_n$ to \eqref{eq:annular-occupation-final-window} is at
    most
    \[
        \delta_{\mathrm{act}}(\varepsilon)n^{D-1}\bigl(v_{+}-v_{-}(\varepsilon
        )\bigr)n = \frac{1}{2}c_{\mathrm{occ}}(\varepsilon)n^{D}.
    \]
    It follows that
    \[
        \int_{J_n}f_{n}(t)\,dt \ge \frac{1}{2}c_{\mathrm{occ}}(\varepsilon)n^{D}
        .
    \]
    Since $f_{n}(t)\le C_{2}n^{D}$,
    \[
        \operatorname{Leb}(J_n) \ge \frac{c_{\mathrm{occ}}(\varepsilon)}{2C_{2}}
        = \ell_{\mathrm{act}}(\varepsilon).
    \]

    These conclusions hold almost surely on $\mathcal{I}_{L,b}$. Since
    $\mathbb{P}(\mathcal{I}_{L,b})=\theta_{G}(p_{L,b})>0$, this proves \eqref{eq:positive-duration-surface-activity}.

    Finally, $\operatorname{Leb}(J_n)>0$ implies $J_n\ne\varnothing$ for every sufficiently large $n$. Taking the
    supremum of $f_{n}$ over $J_n$ and then the lower limit proves \eqref{eq:surface-order-active-propagation}.
\end{proof}
\section{Fluctuations of Maki--Thompson Observables}
\label{sec:application-mt-observables}

Throughout this section, we keep the notation and standing assumptions of Subsection~\ref{subsec:fluctuation-results}:
the initial states are i.i.d.\ with no initial stiflers and spreader density $\beta\in(0,1)$, and assume either that
$m_{\Delta}(\alpha/\lambda)<1$, or that $0<\alpha<\alpha_{\mathrm{fl}}(G,\lambda,\beta)$.

Let $D$ be the degree of polynomial growth of $G$. By Pansu's theorem \cite{pansu1983},
\[
    \frac{|B_{n}|}{n^{D}}\longrightarrow v_{G}\in(0,\infty).
\]
Hence, for every finite $K\subset B_G(o,k)$,
\[
    |KB_n\oplus B_n|
    \le
    2|K|\bigl(|B_{n+k}|-|B_{n-k}|\bigr)
    =
    o(|B_n|).
\]
Since $S$ is symmetric, $B_n^{-1}=B_n$. Applying the preceding estimate to $K^{-1}$ gives
\[
    |B_nK\oplus B_n|
    =
    |K^{-1}B_n\oplus B_n|
    =
    o(|B_n|).
\]
Thus $(B_n)_{n\ge1}$ is an exhaustive two-sided F{\o}lner sequence. Write
\[
    \mathfrak{A}:=\{gB_{n}:g\in\Gamma,\ n\ge1\}.
\]
Fix constants $0<c_{G}\le C_{G}<\infty$ such that
\[
    c_{G}(1+r)^{D}\le |B_{G}(x,r)| \le C_{G}(1+r)^{D}, \qquad x\in V,\ r\ge0.
\]
We take $S$ symmetric. Let $X=(X_{x})_{x\in V}$ be the i.i.d.\ graphical field, where each coordinate $X_{x}$ encodes
  the initial state $\eta_{0}(x)$ alongside the realizations of the infection and stifling Poisson processes on the
  outgoing edges $(x,xs)$ for $s\in S$. This field generates the Maki--Thompson process $\eta$.

For $y\in V$, let $X^{y}$ be obtained from $X$ by replacing $X_{y}$ with an independent copy. Let $\eta^{y}$ be the
process constructed from $X^{y}$, coupled to $\eta$ through the common graphical field away from $y$. For any
functional $F$, write
\[
    \nabla_{y}F(A) := F(X,A)-F(X^{y},A).
\]

For finite $A\subset V$, define
\[
    F_{\mathcal {A}}(X,A) := \#\{x\in A:\eta_s(x)=1\text{ for some }s\ge0\},
\]
and
\[
    F_{\Xi}(X,A) := \sum_{x\in A}\int_{0}^{\infty}\mathbf{1}_{\{\eta_s(x)=1\}}
    \, ds.
\]
Recall that
\[
    L_x=\sup\{s\ge0:\eta_s(x)=1\},
\]
with $L_x=0$ if $x$ never becomes a spreader. For fixed $t\ge0$, set
\[
    F_{S(t)}(X,A) := \sum_{x\in A}\mathbf{1}_{\{L_x>t\}}.
\]
Then
\[
    F_{\mathcal {A}}(X,B_{n}) = |\mathcal {A}_{n}| = |B_{n}|\kappa_{n},
\]
and
\[
    F_{\Xi}(X,B_{n})=\Xi_{n}, \qquad F_{S(t)}(X,B_{n})=S_{n}(t).
\]

By the bound established in Subsection~\ref{subsec:fluctuation-results}, the remaining time spent in the spreader state
is stochastically dominated by $E_{\lambda}+E_{\alpha}$. In particular, every spreader eventually becomes a stifler,
and its occupation time has moments of every order.

The action of $\Gamma$ permutes the i.i.d.\ coordinates of $X$. Thus the field is stationary, and the product action is
mixing and therefore ergodic. The graphical construction is equivariant, so, for every $g\in\Gamma$, every finite
$A\subset V$, and every $\bullet\in\{\mathcal {A},\Xi,S(t)\}$,
\begin{equation}
    F_{\bullet}(\widetilde\varphi_{g}X,gA) = F_{\bullet}(X,A). \label{eq:mt-equivariance}
\end{equation}

\begin{proposition}[Verification for the Maki--Thompson observables]
    \label{prop:verification-mt-observables} There exists
    \[
        \alpha_{\mathrm{fl}}= \alpha_{\mathrm{fl}}(G,\lambda,\beta)>0
    \]
    such that, for every fixed $t\ge0$, the functionals
    \[
        F_{\mathcal {A}}, \qquad F_{\Xi}, \qquad F_{S(t)}
    \]
    satisfy all hypotheses of Theorem~\ref{thm:abstract-clt-functional} whenever either $m_{\Delta}(\alpha/\lambda)<1$,
    or
    \[
        0<\alpha<\alpha_{\mathrm{fl}}.
    \]
\end{proposition}

\begin{proof}
    \medskip

    \noindent
    \emph{Step 1: The discrepancy cluster and a finite-speed estimate.}\par

    Recall that \(X^y\) is obtained from \(X\) by resampling only \(X_y\), so the two constructions use the same
    graphical data away from \(y\). Consider the union of their graphical marks. A shared mark is called
    \emph{affected} if it produces different state updates in the two processes. A mark present in only one
    construction is called affected if it produces a state update in that construction.

    Work on the almost-sure event on which no two distinct marks in this union have the same time coordinate. A
    vertical segment keeps the spatial vertex fixed while time increases. Influence paths are oriented in time but not
    in space: an affected mark on the ordered edge $(x,z)$ may be traversed, at its occurrence time, between $x$ and
    $z$ in either direction. This only enlarges the family of admissible paths. A time-oriented influence path
    generated by the resampling at $y$ starts from $(y,0)$ and follows the two rules below.

    \emph{(i)} At any vertex $x$, the path may move vertically during any interval
    on which
    \[
        \eta_{s}(x)\ne\eta_{s}^{y}(x),
    \]
    and may traverse an affected mark incident to $x$ at its occurrence time. If such a mark occurs at the right
    endpoint of the interval, the path may reach that endpoint and traverse it.

    \emph{(ii)} If an affected mark from one of the resampled outgoing clock
    families with source $y$ occurs at time $s$, the path may follow the vertical
    segment from $(y,0)$ to $(y,s)$ and then traverse that mark, even if the two
    processes agree at $y$ before time $s$.
    This additional rule is needed because the families with source $y$ are
    resampled even when the two processes agree at $y$.

    Let \(\mathfrak D_y\subset V\times[0,\infty)\) be the union of all finite influence paths satisfying the preceding
    rules, together with the initial point \((y,0)\). Then every space--time point at which the two processes disagree
    on the state of a vertex, as well as both space--time endpoints of every affected mark, belongs to $\mathfrak{D}
    _{y}$. Indeed, an initial state discrepancy can occur only at $y$, while every later state discrepancy either
    continues an existing one or is created by an affected mark. Moreover, a mark outside the resampled outgoing clock
    families with source $y$ is present in both constructions and can be affected only if the processes already
    disagree at its source or target.

        For $s\ge0$, write
        \[
            \mathcal{D}_{s}^{y}:= \{x\in V:(x,s)\in\mathfrak{D}_{y}\}.
        \]

        Set
        \[
            R_{y}:= \sup\left\{ d_{G}(y,x): (x,s)\in\mathfrak{D}_{y}\text{ for some
            }s\ge0 \right\},
        \]
        and define separately
        \[
            \Lambda_{y}:= \sup\left\{ s\ge0:
            \begin{array}{l}
                \text{an affected mark occurs at time $s$, or} \\
                \mathbf{1}_{\{\eta_s(x)=1\}} \ne \mathbf{1}_{\{\eta_s^y(x)=1\}} \text{ for some }x\in V
            \end{array}
            \right\},
        \]
        with the convention that the supremum of the empty set is zero. Thus \(\mathfrak D_y\) contains every spatial
        state discrepancy, including persistent \(0/2\) discrepancies; \(R_y\) measures its spatial extent, whereas
        \(\Lambda_y\) measures the temporal extent of its dynamically active part.

        For $r\ge1$, let
        \[
            T_{\mathfrak{D},r}^{y}:= \inf\left\{ s\ge0: d_{G}(y,x)\ge r \text{ for
                some }x\in \mathcal{D}_{s}^{y}\right\}.
        \]

        If the discrepancy cluster reaches graph distance at least $r$ from $y$ before time $ar$, then there is a
        time-oriented path whose spatial projection contains a nearest-neighbor path
        \[
            y=x_{0},x_{1},\ldots,x_{m}, \qquad m\ge r,
        \]
        and whose infection and stifling marks occur before time $ar$.

        After identifying the clocks shared by the two constructions, superpose on each unoriented edge all infection
        and stifling clocks appearing in either construction and in either orientation. The resulting total rate is at
        most
        \[
            \nu_*:=4(\lambda+\alpha).
        \]
        For a fixed spatial path of length $m$, let $t_{i}$ be the time of the $i$th required mark, with $t_{0}=0$.
        Since each $t_{i-1}$ is a stopping time for the filtration generated by the superposed clocks, the strong
        Markov property shows that, conditionally on the past, the waiting time for the next required mark on the $i$th
        edge after $t_{i-1}$ stochastically dominates an exponential random variable of rate $\nu_{*}$, even if that
        edge has already been traversed. Hence the traversal time stochastically dominates
        \[
            E_{1}+\cdots+E_{m},
        \]
        where $E_{1},\ldots,E_{m}$ are independent exponential random variables of rate $\nu_{*}$.

        Fix $0<a<\nu_{*}^{-1}$. Applying Chernoff's inequality with parameter $a^{-1} -\nu_{*}>0$, and using $r\le m$,
        gives
        \[
            \begin{aligned}
                \mathbb{P}\left( \sum_{i=1}^{m}E_{i}\le ar \right) & \le e^{(a^{-1}-\nu_{*})ar}\left( \frac{\nu_{*}}{\nu_{*}+a^{-1}-\nu_{*}}\right)^{m} \\
                                                                   & \le \exp\left\{ m\bigl( 1-a\nu_{*}+\log(a\nu_{*}) \bigr) \right\}.
            \end{aligned}
        \]
        There are at most $\Delta^{m}$ nearest-neighbor paths of length $m$ starting from $y$. Choose $a>0$
        sufficiently small that
        \[
            \log\Delta + 1-a\nu_{*}+ \log(a\nu_{*}) <0.
        \]
        Summing over $m\ge r$, we obtain constants $c,C>0$ such that
        \begin{equation}
            \mathbb{P}(T_{\mathfrak{D},r}^{y}\le ar) \le Ce^{-cr}, \qquad r\ge1,
            \label{eq:fast-discrepancy-propagation}
        \end{equation}
        uniformly in $y$.

        \medskip

        \noindent
        \emph{Step 2: Localization in the subcritical regime.}\par

        Assume first that \( m_{\Delta}(\alpha/\lambda)<1, \) the subcritical parameter regime of
        Theorem~\ref{thm:MT-positive-subcritical}. Fix an initial spreader and reveal its descendants in chronological
        order of birth. Almost surely, every non-initial spreader in this family has a unique parent.

        The conditional comparison used in the proof of Theorem~\ref{thm:MT-positive-subcritical} remains valid when
        descendants of other initial spreaders are present, because it relies only on the fact that non-ignorance is
        absorbing. Indeed, any neighbor that becomes non-ignorant through another family can no longer become a child
        of the spreader and remains non-ignorant. Thus, conditionally on the revealed genealogy, the number of children
        of each explored spreader is dominated by the offspring law used in that proof. The transmission tree rooted at
        each initial spreader is therefore dominated by a Galton--Watson tree whose root offspring distribution is
        stochastically dominated by $1+Z_{\Delta}$ and whose later offspring distributions are stochastically dominated
        by $Z_{\Delta}$. Since $Z_{\Delta}\le\Delta-1$ and has mean less than one, the total population $N$ of the
        dominating tree has an exponential tail, and so does the maximal graph distance of a descendant from the root.
        By successive applications of the strong Markov property, the times spent in the spreader state by the explored
        vertices may be dominated by an i.i.d.\ sequence $(L_i)_{i\ge1}$ with a finite exponential moment.
        Consequently, for every $\varepsilon>0$,
        \[
            \mathbb{P}\left(\sum_{i=1}^{N}L_{i}>t\right) \le \mathbb{P}(N>\varepsilon
            t) + \mathbb{P}\left( \sum_{i=1}^{\lceil\varepsilon t\rceil}L_{i}>t \right
            ).
        \]
        Choose $\theta>0$ such that
        \[
            M(\theta):=\mathbb E[e^{\theta L_1}]<\infty,
        \]
        and then choose $\varepsilon>0$ so that \(\theta-\varepsilon\log M(\theta)>0\). The exponential tail of $N$
        bounds the first term, and Chernoff's inequality the second: for some $C_N,c_N>0$,
        \[
            \mathbb P(N>\varepsilon t)
            +
            \mathbb P\left(
            \sum_{i=1}^{\lceil\varepsilon t\rceil}L_i>t
            \right)
            \le
            C_Ne^{-c_N\varepsilon t}
            +
            M(\theta)e^{-(\theta-\varepsilon\log M(\theta))t}.
        \]
        Hence the family extinction time has a uniform exponential tail.

        Let
        \[
            S_{t}:=\{x\in V:\eta_{t}(x)=1\}, \qquad S_{t}^{y}:=\{x\in V:\eta_{t}^{y}
            (x) =1\}.
        \]
        Summing the preceding bounds over the possible initial spreaders and using polynomial growth, we obtain
        constants $c,C>0$ such that, for every $u,t\ge 0$,
        \begin{equation}
            \begin{aligned}
                 & \mathbb{P}\left( S_{s}\cap B_{G}(y,u)\ne\varnothing \text{ for some }s\ge t \right)                                     \\
                 & \quad+ \mathbb{P}\left( S_{s}^{y}\cap B_{G}(y,u)\ne\varnothing \text{ for some }s\ge t \right) \le C(1+u+t)^{D}e^{-ct},
            \end{aligned}
            \label{eq:subcritical-local-activity}
        \end{equation}
        uniformly in $y$. Indeed, split the initial spreaders according to
        whether they lie within distance $t$ of $B_{G}(y,u)$. For the nearby ones,
        the transmission tree rooted at that spreader must contain a spreader
        alive at time $t$. For a spreader at distance $u+k$ from $y$, with
        $k\ge t$, the family must have radius at least $k$. Thus a union bound
        over all vertices, and hence also over all initial spreaders, gives a sum
        bounded by a constant times
        \[
            |B_{G}(y,u+t)|e^{-ct}+ \sum_{k\ge t}|B_{G}(y,u+k+1)\setminus B_{G}(y,
            u+k)|e^{-ck}.
        \]
        Polynomial growth turns this into \( C(1+u+t)^{D}e^{-ct}, \) after decreasing $c>0$ if necessary.

        If $R_{y}>r$, then either the discrepancy cluster reaches distance $r$ by time $ar$, or an affected mark occurs
        after time $ar$ before the cluster reaches beyond distance $r$. In the latter case, an infection or stifling
        mark requires a spreader in at least one of the coupled processes inside $B_{G}(y,r+1)$. Therefore
        \[
            \{R_{y}>r\} \subseteq \{T_{\mathfrak{D},r}^{y}\le ar\} \cup \left\{ (
            S_{s}\cup S_{s}^{y})\cap B_{G}(y,r+1)\ne\varnothing \text{ for some }
            s\ge ar \right\}.
        \]
        Combining \eqref{eq:fast-discrepancy-propagation} with \eqref{eq:subcritical-local-activity}, and absorbing the
        polynomial factor into the exponential, gives constants $c, C>0$ such that
        \begin{equation}
            \mathbb{P}(R_{y}>r) \le Ce^{-cr}, \qquad r\ge0, \label{eq:subcritical-radius-tail}
        \end{equation}
        uniformly in $y$.

        Finally, suppose that $R_{y}\le t$ and $\Lambda_{y}>t$. At some time $s\ge t$, either an affected mark occurs
        or the processes disagree on the spreader status of some vertex. In either case, at least one of the coupled
        processes contains a spreader in $B_{G}(y,t+1)$ at some time $s\ge t$. Hence
        \[
            \{\Lambda_{y}>t\} \subseteq \{R_{y}>t\} \cup \left\{ (S_{s}\cup S_{s}
                ^{y})\cap B_{G}(y,t+1)\ne\varnothing \text{ for some }s\ge t \right\}
            .
        \]
        Using \eqref{eq:subcritical-radius-tail} and \eqref{eq:subcritical-local-activity}, we obtain
        \begin{equation}
            \mathbb{P}(\Lambda_{y}>t) \le Ce^{-ct}, \qquad t\ge0, \label{eq:subcritical-lifetime-tail}
        \end{equation}
        uniformly in $y$.

        \medskip

        \noindent
        \emph{Step 3: Localization in the small-stifling regime.}\par

        We now choose the scale used in the block construction and the corresponding threshold
        $\alpha_{\mathrm{fl}}>0$.

        Write $Z_{s}:=\{x\in V:\eta_{s}(x)=2\}$ and $Z_{s}^{y}:=\{x\in V:\eta_{s}^{y}(x)=2\}$. For each integer
        $\ell\ge2$, let $\mathcal{N}_{\ell}\subset V$ be a maximal $\ell$-separated set, and define a graph $H_{\ell}$
        on $\mathcal{N}_{\ell}$ by
        \[
            z\sim z' \quad\Longleftrightarrow\quad 0<d_{G}(z,z')\le3\ell.
        \]
        By maximality, every vertex of $V$ lies within distance $\ell$ of $\mathcal{N}_{\ell}$.

        The two-sided volume bounds imply that
        \[
            d_{0}:= \sup_{\ell\ge2}\sup_{z\in\mathcal{N}_\ell}|\mathcal{N}_{\ell}
            \cap B_{G}(z,3\ell)| <\infty
        \]
        and
        \[
            J_{0}:= \sup_{\ell\ge2}\sup_{z\in\mathcal{N}_\ell}|\mathcal{N}_{\ell}
            \cap B_{G}(z,8\ell)| <\infty.
        \]
        The balls
        \[
            B_{G}\left( w, \left\lfloor\frac{\ell-1}{2}\right\rfloor \right), \qquad
            w\in \mathcal{N}_{\ell},
        \]
        are pairwise disjoint, and the bounds follow by comparing their volumes with that of a ball whose radius is
        proportional to $\ell$. In particular, the degree of $H_{\ell}$ is at most $d_{0}$.

        Moreover, from every simple path in $H_{\ell}$ with $m+1$ vertices one can select at least $(m+1)/J_{0}$
        vertices whose balls of radius $4\ell$ are pairwise disjoint. This follows by a greedy selection, deleting
        after each choice all remaining centers within distance $8\ell$.

        Set
        \begin{equation}
            \varepsilon_{\mathrm{blk}}:=(2d_{0})^{-J_0}. \label{eq:bad-block-threshold}
        \end{equation}
        Then
        \begin{equation}
            d_{0}\varepsilon_{\mathrm{blk}}^{1/J_0}=\frac{1}{2}. \label{eq:bad-block-subcriticality}
        \end{equation}

        The constants are chosen in the following order. First, $d_0$ and $J_0$ are fixed in terms of $G$, and
        $\varepsilon_{\mathrm{blk}}=(2d_0)^{-J_0}$. We then choose $\ell$ and $T_\ell$ by \eqref{eq:local-seed-failure}
        and \eqref{eq:local-richardson-failure}, respectively; in particular, $T_\ell$ depends on $\lambda$ and $\ell$,
        but not on $\alpha$. Next, we choose $\alpha_{\mathrm{fl}}$ by \eqref{eq:fluctuation-stifling-threshold}, fix
        $0<\alpha<\alpha_{\mathrm{fl}}$, and define $u_\ell$ and $s_\ell$ by \eqref{eq:local-stifling-time}. Finally,
        for each $M>0$, we choose $K_M$, $W_r$, and $r_0$ below.

        For $z\in\mathcal{N}_{\ell}$, set
        \[
            I_{z}:=B_{G}(z,\ell), \qquad \Pi_{z}:=\{x\in V:d_{G}(z,x)=2\ell\}, \qquad
            Q_{z}:=B_{G}(z,4\ell),
        \]
        and let
        \[
            A_{z}:=\{x\in I_{z}:\eta_{0}(x)=1\}.
        \]
        Choose $\ell\ge2$ sufficiently large that
        \begin{equation}
            \mathbb{P}(A_{z}=\varnothing) = (1-\beta)^{|I_z|}\le \frac{\varepsilon_{\mathrm{blk}}}{4}
            . \label{eq:local-seed-failure}
        \end{equation}

        For every nonempty $A\subset I_{z}$, let $\zeta^{A,z}$ be the Richardson process in $Q_{z}$, started from $A$
        and using only infection marks on oriented edges with both endpoints in $Q_{z}$. Since $Q_{z}$ is finite and
        connected, after fixing $\ell$ we may choose a deterministic $T_{\ell}<\infty$ such that
        \begin{equation}
            \sup_{z\in\mathcal{N}_\ell}\sup_{\varnothing\ne A\subset I_z}\mathbb{P}
            \left( B_{G}(z,3\ell)\not\subset\zeta_{T_\ell}^{A,z}\right) \le \frac{\varepsilon_{\mathrm{blk}}}{4}
            . \label{eq:local-richardson-failure}
        \end{equation}
        Uniformity follows from translation invariance, and by monotonicity it is
        enough to consider single-vertex initial sets.

        There are at most $\Delta|Q_{z}|$ stifling clocks with source in $Q_{z}$. Hence
        \[
            \mathbb{P}\left( \text{some such clock rings during }[0,T_{\ell}] \right
            ) \le \alpha \Delta|Q_{z}|T_{\ell}.
        \]
        Since $|Q_{z}|=|B_{G}(o,4\ell)|$, choose
        \begin{equation}
            0<\alpha_{\mathrm{fl}}\le \frac{\varepsilon_{\mathrm{blk}}}{4\Delta|B_{G}(o,4\ell)|T_{\ell}}
            . \label{eq:fluctuation-stifling-threshold}
        \end{equation}
        From now on, fix $0<\alpha<\alpha_{\mathrm{fl}}$.

        Assume that $A_z\ne\varnothing$, that
        \[
            B_G(z,3\ell)\subset\zeta_{T_\ell}^{A_z,z},
        \]
        and that no stifling clock with source in $Q_z$ rings during $[0,T_\ell]$. Then no vertex of $Q_z$ becomes a
        stifler before time $T_\ell$. Thus every vertex of $Q_{z}$, once it becomes non-ignorant before time $T
        _{\ell}$, remains a spreader until $T_{\ell}$, and every infection mark used by the restricted Richardson
        process is therefore effective in the Maki--Thompson process as well. Consequently, every vertex of
        $B_{G}(z,3\ell)$ is non-ignorant at time $T_{\ell}$. This conclusion depends only on the graphical field inside
        $Q_{z}$.

        Set
        \[
            \pi_{\ell}:=|\Pi_{z}| = |\{x\in V:d_{G}(o,x)=2\ell\}|,
        \]
        which is independent of $z$. For each $x\in\Pi_{z}$, fix a neighbor $x^{\star}$ lying on a geodesic from $x$ to
        $z$. Then $x^{\star}\in B_{G}(z,3\ell)$.

        Choose
        \begin{equation}
            u_{\ell}:=
            \frac{1}{\alpha}
            \log\left(\frac{4\pi_{\ell}}{\varepsilon_{\mathrm{blk}}}\right), \qquad s_{\ell}:=T_{\ell}+u_{\ell}. \label{eq:local-stifling-time}
        \end{equation}
        A union bound gives
        \[
            \mathbb{P}\left(
            \begin{array}{c}
                    \text{for some $x\in\Pi_{z}$, the stifling clock on} \\
                    \text{$(x,x^{\star})$ does not ring during $(T_{\ell},s_{\ell}]$}
                \end{array}
            \right) \le\pi_{\ell}e^{-\alpha u_\ell}= \frac{\varepsilon_{\mathrm{blk}}}{4}
            .
        \]

        We call $z\in\mathcal{N}_{\ell}$ \emph{good} if $A_{z}\ne\varnothing$,
        \[
            B_{G}(z,3\ell)\subset\zeta_{T_\ell}^{A_z,z},
        \]
        no stifling clock with source in $Q_{z}$ rings during $[0,T_{\ell}]$, and, for every $x\in\Pi_{z}$, the clock
        associated with $(x,x^{\star})$ rings at least once during $(T_{\ell},s_{\ell}]$. By the preceding estimates,
        \begin{equation}
            \sup_{z\in\mathcal{N}_\ell}\mathbb{P}(z\text{ is bad}) \le \varepsilon
            _{\mathrm{blk}}. \label{eq:bad-block-probability}
        \end{equation}

        The event $\{z\text{ is good}\}$ is determined by the initial states and the graphical marks with sources in
        $Q_{z}$ occurring before time $s_{\ell}$. Moreover, on this event,
        \begin{equation}
            \Pi_{z}\subset Z_{s_\ell}. \label{eq:local-stifler-cutset}
        \end{equation}
        At time $T_{\ell}$, both $x$ and $x^{\star}$ are non-ignorant for every $x
    \in\Pi_{z}$. Since non-ignorance is absorbing, $x^{\star}$ remains non-ignorant.
        When the prescribed clock rings, either $x$ is already a stifler or it is
        still a spreader and is turned into one. The same conclusion holds for
        every process whose graphical field agrees with $X$ on $Q_z$ up to
        time $s_\ell$; see Figure~\ref{fig:good-block-anatomy}.

        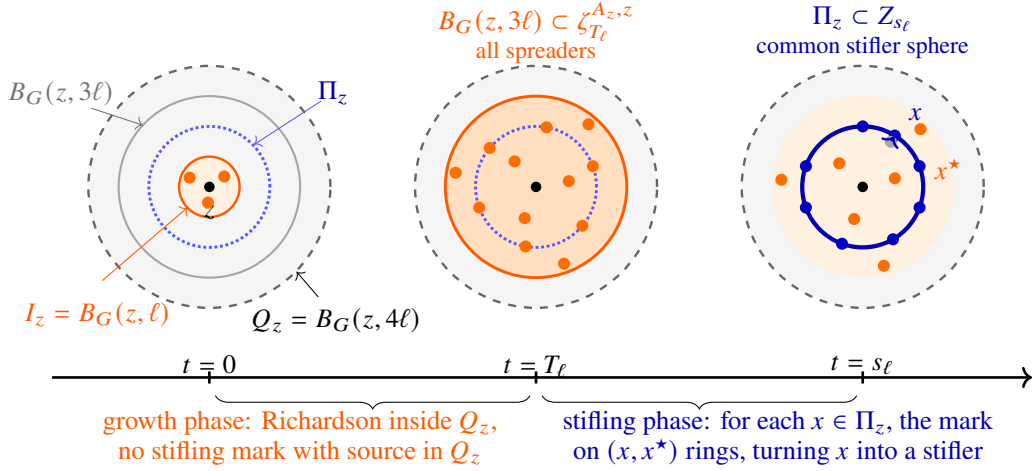
\begin{figure}[ht]
            \centering
            \begin{tikzpicture}[
                    scale=0.80,
                    every node/.style={font=\small},
                    spr/.style={circle, fill=orange!85!red, inner sep=1.6pt},
                    sti/.style={circle, fill=blue!75!black, inner sep=1.6pt},
                    nonign/.style={circle, fill=black!35, inner sep=1.6pt},
                    qbox/.style={draw=black!60, dashed, line width=0.9pt},
                    pisty/.style={draw=blue!65, densely dotted, line width=1.1pt}
                ]

                \begin{scope}[shift={(0,0)}]
                    \fill[black!4] (0,0) circle (2.0);
                    \fill[orange!12] (0,0) circle (0.5);

                    \draw[qbox] (0,0) circle (2.0);
                    \draw[black!35, line width=0.8pt] (0,0) circle (1.5);
                    \draw[pisty] (0,0) circle (1.0);
                    \draw[orange!70!red, line width=0.9pt] (0,0) circle (0.5);

                    \node[circle, fill=black, inner sep=1.4pt] at (0,0) {};
                    \node[below=3pt] at (0,0) {$z$};

                    \foreach \a/\rr in {35/0.30,155/0.36,265/0.26}
                        { \node[spr] at (\a:\rr) {}; }

                    \node[orange!70!red] at (-1.85,-2.1)
                    {$I_z=B_G(z,\ell)$};
                    \draw[->, orange!70!red]
                    (-1.770,-1.55) -- (-0.36,-0.34);

                    \node[blue!65!black] at (2.05,1.55) {$\Pi_z$};
                    \draw[->, blue!65]
                    (1.85,1.42) -- (0.74,0.70);

                    \node[black!55] at (-2.45,1.55)
                    {$B_G(z,3\ell)$};
                    \draw[->, black!55]
                    (-1.95,1.30) -- (-1.10,1.02);

                    \node at (2.1,-2.2)
                    {$Q_z=B_G(z,4\ell)$};
                    \draw[->]
                    (1.80,-1.82) -- (1.44,-1.40);

                    \node[below=6pt] at (0,-2.35) {$t=0$};
                \end{scope}

                \begin{scope}[shift={(5.4,0)}]
                    \fill[black!4] (0,0) circle (2.0);
                    \fill[orange!25] (0,0) circle (1.5);

                    \draw[qbox] (0,0) circle (2.0);
                    \draw[orange!70!red, line width=1.1pt] (0,0) circle (1.5);
                    \draw[pisty] (0,0) circle (1.0);

                    \node[circle, fill=black, inner sep=1.4pt] at (0,0) {};

                    \foreach \a in {20,80,140,200,260,320}
                        { \node[spr] at (\a:1.0) {}; }
                    \foreach \a in {50,170,290}
                        { \node[spr] at (\a:1.35) {}; }
                    \foreach \a in {10,130,250}
                        { \node[spr] at (\a:0.55) {}; }

                    \node[orange!70!red, align=center] at (0,2.55)
                    {$B_G(z,3\ell)\subset\zeta^{A_z,z}_{T_\ell}$\\[-1pt]
                    {\footnotesize all spreaders}};

                    \node[below=6pt] at (0,-2.35) {$t=T_\ell$};
                \end{scope}

                \begin{scope}[shift={(10.8,0)}]
                    \fill[black!4] (0,0) circle (2.0);
                    \fill[orange!12] (0,0) circle (1.5);

                    \draw[qbox] (0,0) circle (2.0);
                    \draw[blue!70!black, line width=1.6pt] (0,0) circle (1.0);

                    \node[circle, fill=black, inner sep=1.4pt] at (0,0) {};

                    \foreach \a in {20,90,160,200,250,300,340}
                        { \node[sti] at (\a:1.0) {}; }
                    \foreach \a in {45,175,285}
                        { \node[spr] at (\a:1.35) {}; }
                    \foreach \a in {15,135,255}
                        { \node[spr] at (\a:0.55) {}; }

                    \node[
                        sti,
                        label={[blue!70!black]above right:$x$}
                    ] (xx) at (58:1.0) {};

                    \node[nonign] (xs) at (58:0.88) {};

                    \node[orange!70!red] at (1.42,0.30) {$x^\star$};
                    \draw[->, blue!70!black, line width=0.9pt]
                    (xx) -- (xs);

                    \node[blue!70!black, align=center] at (0,2.55)
                    {$\Pi_z\subset Z_{s_\ell}$\\[-1pt]
                    {\footnotesize common stifler sphere}};

                    \node[below=6pt] at (0,-2.35) {$t=s_\ell$};
                \end{scope}

                \draw[->, line width=1pt]
                (-2.6,-3.15) -- (13.6,-3.15);

                \foreach \x in {0,5.4,10.8}
                    { \draw[line width=1pt] (\x,-3.05) -- (\x,-3.25); }

                \draw[
                    decorate,
                    decoration={brace, amplitude=5pt, mirror}
                ]
                (0.1,-3.45) -- (5.3,-3.45);

                \draw[
                    decorate,
                    decoration={brace, amplitude=5pt, mirror}
                ]
                (5.5,-3.45) -- (10.7,-3.45);

                \node[orange!70!red, align=center] at (1.5,-4.15)
                {growth phase: Richardson inside $Q_z$,\\[-1pt]
                    no stifling mark with source in $Q_z$};

                \node[blue!70!black, align=center] at (9.4,-4.15)
                {stifling phase: for each $x\in\Pi_z$, the mark\\[-1pt]
                    on $(x,x^\star)$ rings, turning $x$ into a stifler};

            \end{tikzpicture}

            \caption[Anatomy of a good block]{Good-block mechanism in Step~3.
                Starting from $A_z\subset I_z$, the restricted Richardson process
                covers $B_G(z,3\ell)$ by time $T_\ell$, while no stifling mark with
                source in $Q_z$ occurs. During $(T_\ell,s_\ell]$, the prescribed
                stifling marks turn every $x\in\Pi_z$ into a stifler, yielding
                $\Pi_z\subset Z_{s_\ell}$. The radii
                $\ell,2\ell,3\ell,4\ell$ are drawn to scale.}
            \label{fig:good-block-anatomy}
        \end{figure}

        The set $\Pi_{z}$ separates $B_{G}(z,2\ell-1)$ from $V\setminus B_{G}(z,2\ell )$. The field of good-block
        indicators is finite-range dependent: if the balls $Q_{z}$ are pairwise disjoint, the corresponding events are
        independent. Consequently, from every simple path of length $m$ in $H_{\ell}$ one can select at least $m/J_{0}$
        vertices whose bad-block events are mutually independent.

        Fix $M>0$, and choose
        \[
            K_{M}>\frac{8\ell(M+D+1)}{\log2}.
        \]
        For $r\ge2$, set \( W_{r}:=\lceil K_{M}\log r\rceil. \) Choose $r_{0}\ge2$ so that, for every $r\ge r_{0}$,
        \[
            W_{r}\le r, \qquad W_{r}\ge36\ell.
        \]
        For $r\ge r_{0}$, let $\mathsf{G}_{r,y}$ be the event that there is no simple path \( z_{0},z_{1},\ldots,z_{m}
        \) of bad vertices in $H_{\ell}$ such that
        \[
            Q_{z_i}\subset B_{G}(y,r+W_{r})\setminus B_{G}(y,r), \qquad 0\le i\le
            m,
        \]
        and
        \[
            m\ge \left\lfloor \frac{W_{r}-18\ell}{3\ell}\right\rfloor.
        \]

        For every $r\ge r_{0}$,
        \[
            \left\lfloor \frac{W_{r}-18\ell}{3\ell}\right\rfloor \ge \frac{W_{r}}{8\ell}
            .
        \]
        Since $W_{r}\le r$, the initial vertex of any such bad path belongs to $B _{G}(y,2r)$. Hence, for $r\ge1$, the
        number of possible initial vertices of such a path is at most
        \[
            |B_{G}(y,2r)| \le C_{G}(1+2r)^{D}\le 3^{D}C_{G}r^{D}.
        \]

        Moreover, a path of length $m$ contains at least $m/J_{0}$ vertices whose bad-block events are mutually
        independent. Therefore, by \eqref{eq:bad-block-probability} and \eqref{eq:bad-block-subcriticality},
        \[
            \begin{aligned}
                \mathbb{P}(\mathsf{G}_{r,y}^{c}) & \le 3^{D}C_{G}r^{D}\sum_{ m\ge \left\lfloor (W_{r}-18\ell)/(3\ell) \right\rfloor }d_{0}^{m}\varepsilon_{\mathrm{blk}}^{m/J_0} \\
                                                 & \le 2\cdot3^{D}C_{G}r^{D}2^{-W_{r}/(8\ell)}.
            \end{aligned}
        \]
        By the choice of $K_{M}$, there exists $C_{M}<\infty$ such that
        \begin{equation}
            \mathbb{P}(\mathsf{G}_{r,y}^{c}) \le C_{M}r^{-M}, \qquad r\ge r_{0},
            \label{eq:annular-block-polynomial-tail}
        \end{equation}
        uniformly in $y$. After increasing $C_{M}$, the same estimate holds for
        all $r\ge2$.

        On $\mathsf{G}_{r,y}$, define
        \[
            \Sigma_{r,y}:= \bigcup_{\substack{ z\in\mathcal{N}_\ell,\ z\text{ good}\\ Q_z\subset B_G(y,r+W_{r})\setminus B_G(y,r) }}
            \Pi_{z}.
        \]

        We will use the following nested balls:
        \begin{align*}
            B_G(y,r)
             & \subset B_G(y,r+6\ell)
            \subset B_G(y,r+7\ell)
            \subset B_G(y,r+W_r-7\ell)    \\
             & \subset B_G(y,r+W_r-6\ell)
            \subset B_G(y,r+W_r).
        \end{align*}
        We claim that \(\Sigma_{r,y}\) separates \(B_G(y,r)\) from
        \(V\setminus B_G(y,r+W_r)\); see
        Figure~\ref{fig:annulus-separation-unrolled}.

        \begin{figure}[ht]
            \centering

            \begin{tikzpicture}[
                    scale=0.85,
                    line cap=round,
                    line join=round,
                    every node/.style={font=\small},
                    >=Latex,
                    center/.style={
                            rectangle,
                            draw=coarseteal!85!black,
                            fill=coarseteal,
                            inner sep=1.45pt
                        },
                    pathv/.style={
                            circle,
                            fill=black,
                            inner sep=1.55pt
                        },
                    lvlouter/.style={
                            draw=black,
                            line width=1.10pt
                        },
                    lvlinner/.style={
                            draw=black!48,
                            line width=0.70pt
                        },
                    lvlmargin/.style={
                            draw=black!22,
                            line width=0.65pt
                        }
                ]

                %
                %

                \def\xr{0}
                \def\xsixL{1.8}
                \def\xsevenL{2.1}
                \def\xsevenR{9.9}
                \def\xsixR{10.2}
                \def\xwr{12}

                \def\ytop{1.45}
                \def\ybot{-1.95}


                \fill[black!10]
                (-1.40,\ybot) rectangle (\xr,\ytop);

                \fill[blue!5]
                (\xr,\ybot) rectangle (\xwr,\ytop);

                \draw[lvlouter]
                (\xr,\ybot) -- (\xr,\ytop);

                \draw[lvlouter]
                (\xwr,\ybot) -- (\xwr,\ytop);

                \draw[lvlmargin]
                (\xsixL,\ybot) -- (\xsixL,\ytop);

                \draw[lvlmargin]
                (\xsixR,\ybot) -- (\xsixR,\ytop);

                \draw[lvlinner]
                (\xsevenL,\ybot) -- (\xsevenL,\ytop);

                \draw[lvlinner]
                (\xsevenR,\ybot) -- (\xsevenR,\ytop);

                %

                \coordinate (zstart) at (2.25,-0.55);

                \coordinate (cA) at (3.05, 0.10);
                \coordinate (cB) at (3.75,-0.50);
                \coordinate (cC) at (3.55,-0.85);
                \coordinate (cD) at (4.55, 0.35);
                \coordinate (cE) at (5.45,-0.10);
                \coordinate (cF) at (6.45,-0.25);

                \coordinate (zi) at (7.15,-0.75);

                \coordinate (cG)    at (8.25,0.40);
                \coordinate (cH)    at (9.15,0.05);
                \coordinate (cLast) at (9.80,0.30);

                %

                \coordinate (cAd) at ($(cA)+(0,-0.55)$);
                \coordinate (cBd) at ($(cB)+(0,-0.55)$);
                \coordinate (cCd) at ($(cC)+(0,-0.55)$);
                \coordinate (cDd) at ($(cD)+(0,-0.55)$);
                \coordinate (cEd) at ($(cE)+(0,-0.55)$);
                \coordinate (cFd) at ($(cF)+(0,-0.55)$);

                \coordinate (cGd)    at ($(cG)+(0,-0.55)$);
                \coordinate (cHd)    at ($(cH)+(0,-0.55)$);
                \coordinate (cLastd) at ($(cLast)+(0,-0.55)$);

                %

                \draw[
                    black!55,
                    densely dotted,
                    line width=0.90pt
                ]
                (zstart) circle (0.85);

                \node[
                    black!55,
                    scale=0.92
                ] (Qlabel) at (0.78,-2.43)
                {$Q_{z_0}$};

                \draw[
                    ->,
                    black!50,
                    line width=0.45pt
                ]
                (Qlabel.north east)
                -- (1.67,-1.17);

                %

                \draw[
                    coarseteal!85!black,
                    densely dashed,
                    line width=0.70pt
                ]
                (zstart)
                -- (cAd)
                -- (cBd)
                -- (cCd)
                -- (cDd)
                -- (cEd)
                -- (cFd)
                -- (zi)
                -- (cGd)
                -- (cHd)
                -- (cLastd);


                \coordinate (vi) at (\xsevenL,-0.60);
                \coordinate (vj) at (\xsevenR, 0.40);

                \draw[
                    black,
                    line width=1.45pt
                ]
                (-1.20,0.20)
                -- (-0.65,-0.15)
                -- (0.55,0.45)
                -- (1.30,-0.25)
                -- (vi)
                -- (2.90,0.05)
                -- (3.90,-0.55)
                -- (3.40,-1.05)
                -- (4.60,0.45)
                -- (5.50,-0.15)
                -- (6.50,-0.30)
                -- (7.20,-0.90)
                -- (8.30,0.45)
                -- (9.20,0.00)
                -- (vj)
                -- (11.00,-0.05)
                -- (\xwr,0.35)
                -- (13.10,0.05);

                %

                \node[center] at (zstart) {};

                \node[center] at (cAd) {};
                \node[center] at (cBd) {};
                \node[center] at (cCd) {};
                \node[center] at (cDd) {};
                \node[center] at (cEd) {};
                \node[center] at (cFd) {};

                \node[center] at (zi) {};

                \node[center] at (cGd) {};
                \node[center] at (cHd) {};
                \node[center] at (cLastd) {};


                \node[pathv] at (vi) {};
                \node[pathv] at (vj) {};

                \node[
                    anchor=south east,
                    scale=0.96
                ] at ($(vi)+(-0.12,0.16)$)
                {$v_i$};

                \node[
                    anchor=south east,
                    scale=0.96
                ] at ($(vj)+(-0.12,0.10)$)
                {$v_j$};


                \draw[
                    coarseteal!85!black,
                    densely dotted,
                    line width=0.85pt
                ]
                (zi) circle (0.30);

                %
                \draw[
                    green!55!black,
                    dashed,
                    line width=1.00pt
                ]
                (zi) circle (0.60);

                \node[
                    circle,
                    fill=red,
                    inner sep=1.8pt
                ] at (7.623005,-0.380857) {};

                \node[
                    coarseteal!85!black,
                    anchor=north west,
                    scale=0.90
                ] at ($(zi)+(0.10,-0.08)$)
                {$z_i$};

                \node[
                    coarseteal!85!black,
                    anchor=north east,
                    scale=0.94
                ] at ($(zi)+(-0.68,-0.67)$)
                {$B_G(z_i,\ell)$};

                \node[
                    green!55!black,
                    anchor=north west,
                    scale=0.94
                ] at ($(zi)+(0.80,-0.69)$)
                {$\Pi_{z_i}$};


                \node[
                    black!60,
                    anchor=east,
                    scale=0.94
                ] at (0.0512,-1.43)
                {$B_G(y,r)$};

                \node[
                    blue!70!black,
                    anchor=west,
                    scale=0.94
                ] at (12.18,-1.43)
                {$V\setminus B_G(y,r+W_r)$};


                \draw[
                    decorate,
                    decoration={
                            brace,
                            amplitude=5pt
                        }
                ]
                (\xr,1.68)
                -- (\xwr,1.68);

                \node[
                    scale=1.02
                ] at (6,1.98)
                {$W_r$};


                \draw[
                    decorate,
                    decoration={
                            brace,
                            amplitude=4pt,
                            mirror
                        }
                ]
                (0.03,1.48)
                -- (1.76,1.48);

                \node[
                    black!65,
                    scale=0.88
                ] at (0.90,1.11)
                {$6\ell$};

                \draw[
                    decorate,
                    decoration={
                            brace,
                            amplitude=4pt,
                            mirror
                        }
                ]
                (10.24,1.48)
                -- (11.97,1.48);

                \node[
                    black!65,
                    scale=0.88
                ] at (11.10,1.11)
                {$6\ell$};


                \draw[
                    decorate,
                    decoration={
                            brace,
                            amplitude=5pt,
                            mirror
                        }
                ]
                (2.25,-2.10)
                -- (9.80,-2.10);

                \node[
                    scale=0.94
                ] at (6.02,-2.56)
                {$\geq W_r-18\ell$};


                \draw[
                    ->,
                    line width=0.90pt
                ]
                (-1.40,-3.32)
                -- (13.45,-3.32)
                node[
                    right=3pt,
                    scale=0.94
                ]
                {$d_G(y,\cdot)$};

                \draw[line width=0.90pt]
                (0,-3.22) -- (0,-3.42);

                \draw[line width=0.90pt]
                (1.8,-3.22) -- (1.8,-3.42);

                \draw[line width=0.90pt]
                (2.1,-3.22) -- (2.1,-3.42);

                \draw[line width=0.90pt]
                (9.9,-3.22) -- (9.9,-3.42);

                \draw[line width=0.90pt]
                (10.2,-3.22) -- (10.2,-3.42);

                \draw[line width=0.90pt]
                (12,-3.22) -- (12,-3.42);

                %

                \node[
                    anchor=north,
                    scale=0.84
                ] at (0,-3.52)
                {$r$};

                \node[
                    anchor=north east,
                    rotate=90,
                    scale=0.80
                ] at (1.5,-3.52)
                {$r+6\ell$};

                \node[
                    anchor=north east,
                    rotate=90,
                    scale=0.80
                ] at (2.0,-3.52)
                {$r+7\ell$};

                \node[
                    anchor=north east,
                    rotate=90,
                    scale=0.80
                ] at (9.7,-3.52)
                {$r+W_r-7\ell$};

                \node[
                    anchor=north east,
                    rotate=90,
                    scale=0.80
                ] at (10.2,-3.52)
                {$r+W_r-6\ell$};

                \node[
                    anchor=north east,
                    rotate=90,
                    scale=0.80
                ] at (12,-3.52)
                {$r+W_r$};

            \end{tikzpicture}

            \caption[Separation by the stifler cutset]{The annulus unrolled along
                $d_G(y,\cdot)$. A crossing path (black) avoiding $\Sigma_{r,y}$
                induces, through the $\ell$-net $\mathcal N_\ell$, a coarse path
                $z_0,\ldots,z_m$ in $H_\ell$ (teal). The intermediate portion of
                the coarse path is displaced slightly downward for visual clarity.
                The $6\ell$ margins keep the corresponding blocks $Q_z$ inside the
                annulus. If some $z_i$ were good, the crossing path would enter
                $B_G(z_i,\ell)$ and subsequently cross
                $\Pi_{z_i}\subset\Sigma_{r,y}$ (red point), a contradiction.
                Thus every $z_i$ is bad, while the first and last centres are
                separated by at least $W_r-18\ell$, forcing
                $m\geq\lfloor(W_r-18\ell)/(3\ell)\rfloor$ and contradicting
                $\mathsf G_{r,y}$. Drawn with $W_r=40\ell$; transverse positions
                are schematic.}
            \label{fig:annulus-separation-unrolled}
        \end{figure}

        Suppose otherwise that a nearest-neighbor path $v_{0},v_{1},\ldots,v_{N}$ crosses the annulus without meeting
        $\Sigma_{r,y}$. Let $j$ be the first index at which the path reaches distance at least $r+W_{r}-7\ell$ from
        $y$, and let $i<j$ be the last index before $j$ at which its distance from $y$ is at most $r+7\ell$. Since
        consecutive vertices of a nearest-neighbor path have distances to $y$ differing by at most one, the choices of
        $i$ and $j$ imply
        \[
            d_{G}(y,v_{i})=r+7\ell, \qquad d_{G}(y,v_{j})=r+W_{r}-7\ell,
        \]
        and
        \[
            r+7\ell < d_{G}(y,v_{k}) < r+W_{r}-7\ell
        \]
        for every $i<k<j$. Thus every vertex of the intervening segment lies between these two radial levels.

        Associate with each vertex of this segment a point of $\mathcal{N}_{\ell}$ at distance at most $\ell$. Every
        chosen center $z$ satisfies
        \[
            r+6\ell \le d_{G}(y,z) \le r+W_{r}-6\ell.
        \]
        Since $Q_{z}=B_{G}(z,4\ell)$,
        \[
            Q_{z}\subset B_{G}(y,r+W_{r})\setminus B_{G}(y,r).
        \]
        Consecutive chosen centers are equal or adjacent in $H_{\ell}$. After deleting repetitions and loops, they form
        a simple path $z_{0},z_{1},\ldots ,z_{m}$ with this property.

        If some $z_{i}$ were good, the original path would meet $B_{G}(z_{i},\ell)$. Since
        \[
            B_{G}(z_{i},2\ell)\subset B_{G}(y,r+W_{r})
        \]
        and the path eventually leaves $B_{G}(y,r+W_{r})$, it would subsequently leave $B_{G}(z_{i},2\ell)$. It would
        therefore meet $\Pi_{z_i}\subset\Sigma _{r,y}$, a contradiction. Hence every vertex of the coarse path is bad.

        The first and last centers are at distance at least $W_{r}-18\ell$, while every edge of $H_{\ell}$ has
        $G$-length at most $3\ell$. Hence
        \[
            m \ge \left\lfloor \frac{W_{r}-18\ell}{3\ell}\right\rfloor,
        \]
        contradicting $\mathsf{G}_{r,y}$.

        For every center $z$ used to define $\Sigma_{r,y}$, we have $y\notin Q_{z}$. Hence $X$ and $X^{y}$ agree on all
        coordinates determining whether $z$ is good. By \eqref{eq:local-stifler-cutset},
        \begin{equation}
            \Sigma_{r,y}\subset Z_{s_\ell}\cap Z_{s_\ell}^{y}. \label{eq:common-stifler-cutset}
        \end{equation}

        A cutset consisting of stiflers in both coupled processes blocks the discrepancy. Marks with source in the
        cutset and transmission marks targeting it are ineffective in both processes. A stifling mark directed toward
        the cutset can change only its source, which remains in the same component. Thus no affected mark can carry the
        discrepancy across the cutset; see Figure~\ref{fig:discrepancy-common-cutset}.

        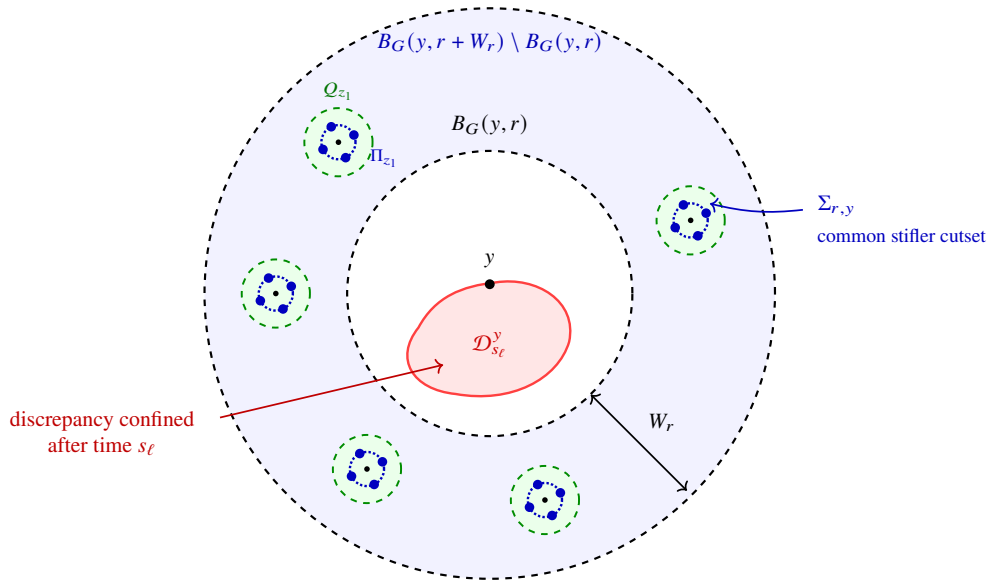
\begin{figure}[ht]
            \centering
            \scalebox{0.82}{
                \begin{tikzpicture}[
                        every node/.style={font=\small},
                        qblock/.style={
                                draw=green!55!black,
                                dashed,
                                fill=green!8,
                                line width=0.9pt
                            },
                        pisphere/.style={
                                draw=blue!75!black,
                                densely dotted,
                                line width=1.1pt
                            },
                        stifler/.style={
                                circle,
                                fill=blue!75!black,
                                inner sep=1.55pt
                            }
                    ]


                    \newcommand{\goodblock}[3]{%
                        \begin{scope}[shift={(#1,#2)}]
                            \filldraw[qblock] (0,0) circle (0.55);
                            \draw[pisphere] (0,0) circle (0.275);

                            \foreach \ang in {25,115,205,295}
                                { \node[stifler] at (\ang:0.275) {}; }

                            \node[
                                circle,
                                fill=black,
                                inner sep=0.9pt
                            ] at (0,0) {};
                        \end{scope}
                    }

                    \fill[blue!5] (0,0) circle (4.6);
                    \fill[white]  (0,0) circle (2.3);

                    \draw[dashed, line width=1pt] (0,0) circle (2.3);
                    \draw[dashed, line width=1pt] (0,0) circle (4.6);

                    \node at (0,2.72) {$B_G(y,r)$};

                    \node[blue!70!black, align=center]
                    at (0,4.05)
                    {$B_G(y,r+W_r)\setminus B_G(y,r)$};

                    \draw[
                        red!75,
                        fill=red!10,
                        line width=1.1pt
                    ]
                    (-1.15,-0.55)
                    .. controls (-1.55,-1.05) and (-1.25,-1.55) ..
                    (-0.55,-1.62)
                    .. controls (0.35,-1.78) and (1.10,-1.38) ..
                    (1.28,-0.78)
                    .. controls (1.42,-0.12) and (0.75,0.28) ..
                    (0.10,0.18)
                    .. controls (-0.45,0.12) and (-0.88,-0.05) ..
                    (-1.15,-0.55)
                    -- cycle;

                    \node[red!80!black]
                    at (0,-0.82)
                    {$\mathcal D_{s_\ell}^{y}$};

                    \node[
                        circle,
                        fill=black,
                        inner sep=1.6pt
                    ] (y) at (0,0.15) {};

                    \node[above=4pt] at (y) {$y$};

                    \goodblock{-2.44}{ 2.44}{1}
                    \goodblock{-3.45}{ 0.00}{2}
                    \goodblock{-1.98}{-2.83}{3}
                    \goodblock{ 0.89}{-3.33}{4}
                    \goodblock{ 3.24}{ 1.18}{5}

                    \node[green!45!black, font=\scriptsize]
                    at (-2.44,3.27)
                    {$Q_{z_1}$};

                    \node[blue!75!black, font=\scriptsize]
                    at (-1.70,2.18)
                    {$\Pi_{z_1}$};

                    \draw[
                        blue!75!black,
                        thick,
                        ->,
                        bend left=10
                    ]
                    (5.05,1.35) to (3.55,1.45);

                    \node[blue!75!black, align=left, anchor=west] at (5.15,1.20)
                    {$\Sigma_{r,y}$\\[1mm]
                    {\footnotesize common stifler cutset}};

                    \draw[<->, thick]
                    (1.63,-1.63) -- (3.18,-3.18);

                    \node[right]
                    at (2.43,-2.08)
                    {$W_r$};

                    \node[
                        red!75!black,
                        align=center
                    ]
                    at (-6.25,-2.25)
                    {discrepancy confined\\ after time $s_\ell$};

                    \draw[
                        red!75!black,
                        thick,
                        ->
                    ]
                    (-4.35,-2.00) -- (-0.75,-1.15);

                \end{tikzpicture}
            }

            \caption[Localization by a common stifler cutset]{Good blocks in the
                annulus produce a common stifler cutset $\Sigma_{r,y}$ for the two
                coupled processes. On $\mathsf G_{r,y}$ it separates the inner and
                outer regions, so a discrepancy that has not reached distance $r$ by
                time $s_\ell$ remains confined to the inner component thereafter.}
            \label{fig:discrepancy-common-cutset}
        \end{figure}

        Therefore, on
        \[
            \mathsf{G}_{r,y}\cap \{T_{\mathfrak{D},r}^{y}>s_{\ell}\},
        \]
        the discrepancy lies inside $B_{G}(y,r)$ when the cutset is completed and cannot cross it afterwards.
        Consequently,
        \begin{equation}
            \{R_{y}>r+W_{r}\} \subseteq \{T_{\mathfrak{D},r}^{y}\le s_{\ell}\} \cup
            \mathsf{G}_{r,y}^{c}. \label{eq:radius-block-inclusion}
        \end{equation}

        Once $(G,\lambda,\alpha,\beta)$ is fixed, $\ell$ and $s_{\ell}$ are fixed. Hence $s_{\ell}\le ar$ for all
        sufficiently large $r$. Combining \eqref{eq:fast-discrepancy-propagation},
        \eqref{eq:annular-block-polynomial-tail}, and \eqref{eq:radius-block-inclusion}, and taking $r=\lfloor
        R/2\rfloor$, for which $r+W_{r}\le R$ when $R$ is large, gives, for every $M>0$,
        \begin{equation}
            \mathbb{P}(R_{y}>R) \le C_{M}R^{-M}, \qquad R\ge2, \label{eq:supercritical-radius-tail}
        \end{equation}
        uniformly in $y$, after increasing $C_{M}$ to cover bounded $R$.

        The same cutset controls the lifetime of the discrepancy cluster. On \( \mathsf{G}_{r,y}\cap
        \{T_{\mathfrak{D},r}^{y}>s_{\ell}\}, \) every future transition whose effect may differ in the two processes is
        confined to the component of $V\setminus\Sigma_{r,y}$ containing $y$, which lies in $B_{G}(y,r+W_{r})$.

        From time $s_{\ell}$ onward, follow both processes inside this component, keeping $\Sigma_{r,y}$ as a fixed
        stifler boundary. Since the cutset prevents discrepancies from entering or leaving the component, it is enough
        to continue these restricted evolutions until neither contains a spreader. Each vertex of the component changes
        state at most twice in each evolution, so the number of remaining effective state changes is at most
        \[
            4|B_{G}(y,r+W_{r})| \le 4C_{G}(1+r+W_{r})^{D}\le 2^{D+2}C_{G}(1+r)^{D}
            ,
        \]
        where we used $W_{r}\le r$.

        Whenever at least one restricted process contains a spreader, the total rate of its next effective state change
        is at least \( \lambda_{*}:=\min\{\lambda,\alpha\}. \) Either the spreader has an ignorant neighbor, and then
        an infection mark is effective at rate at least $\lambda$; or all its neighbors are non-ignorant, and a
        stifling mark is effective at rate at least $\alpha$. Once neither restricted process contains a spreader, no
        event counted by $\Lambda_{y}$ can occur afterwards.

        Conditionally on the configuration after each effective transition, the waiting time to the next effective
        state change is exponential with rate at least $\lambda_{*}$. By successive applications of the strong Markov
        property, these waiting times may be coupled with independent exponential random variables of rate
        $\lambda_{*}$ that dominate them. Hence, for every $u\ge0$,
        \begin{equation}
            \begin{aligned}
                \mathbb{P}(\Lambda_{y}>s_{\ell}+u) \le{} & \mathbb{P}(T_{\mathfrak{D},r}^{y}\le s_{\ell}) + \mathbb{P}(\mathsf{G}_{r,y}^{c}) \\
                                                         & + \mathbb{P}\left( \sum_{i=1}^{\lceil 2^{D+2}C_G(1+r)^D\rceil}E_{i}>u \right),
            \end{aligned}
            \label{eq:supercritical-lifetime-bound}
        \end{equation}
        where the $E_{i}$'s are independent exponential random variables of rate
    $\lambda_{*}$.

        Take
        \[
            r:=\left\lfloor t^{1/(2D)}\right\rfloor, \qquad u:=t-s_{\ell}.
        \]
        For all sufficiently large $t$,
        \[
            u\ge\frac{t}{2}, \qquad 2^{D+2}C_{G}(1+r)^{D}=O(t^{1/2}).
        \]
        Chernoff's inequality gives an exponential bound in $t$ for the last term in
        \eqref{eq:supercritical-lifetime-bound}. Applying \eqref{eq:annular-block-polynomial-tail} with exponent $2DM$,
        together with \eqref{eq:fast-discrepancy-propagation}, yields, for every $M>0$,
        \begin{equation}
            \mathbb{P}(\Lambda_{y}>t) \le C_{M}t^{-M}, \qquad t\ge2, \label{eq:supercritical-lifetime-tail}
        \end{equation}
        uniformly in $y$.

        In either parameter regime considered above, for every $M>0$,
        \begin{equation}
            \sup_{y\in V}\left[ \mathbb{P}(R_{y}>r) + \mathbb{P}(\Lambda_{y}>r) \right
            ] \le C_{M}r^{-M}, \qquad r\ge2. \label{eq:discrepancy-polynomial-tails}
        \end{equation}
        Consequently, for every $q>0$,
        \begin{equation}
            \sup_{y\in V}\mathbb{E}[(1+R_{y})^{q}] + \sup_{y\in V}\mathbb{E}[(1+\Lambda
                _{y})^{q}] <\infty. \label{eq:discrepancy-all-moments}
        \end{equation}

        \medskip

        \noindent
        \emph{Step 4: Moment bounds for the observables.}\par

        Define the spatial projection of the discrepancy cluster by $\mathcal{R}_{y} :=\{x\in
        V:(x,s)\in\mathfrak{D}_{y}\text{ for some }s\ge0\}$. Every vertex whose eventual state or last spreader time
        changes under resampling belongs to $\mathcal{R}_{y}$. Hence, for every finite $A\subset V$,
        \begin{equation}
            \begin{aligned}
                |\nabla_{y}F_{\mathcal {A}}(A)| + |\nabla_{y}F_{S(t)}(A)| & \le C|\mathcal{R}_{y}| \\
                                                                          & \le C|B_{G}(y,R_{y})|  \\
                                                                          & \le C(1+R_{y})^{D}.
            \end{aligned}
            \label{eq:observable-local-bound}
        \end{equation}
        For the occupation-time observable, the spreader indicators can differ
        only at vertices of $\mathcal{R}_{y}$ and at times not exceeding
    $\Lambda_{y}$. Therefore
        \begin{equation}
            |\nabla_{y}F_{\Xi}(A)| \le |\mathcal{R}_{y}|\Lambda_{y}\le C(1+R_{y})
            ^{D}\Lambda_{y}. \label{eq:occupation-local-bound}
        \end{equation}
        By H\"older's inequality and \eqref{eq:discrepancy-all-moments}, for every
    $q>0$,
        \begin{equation}
            \sup_{y\in V}\sup_{A\in\mathfrak{A}}\mathbb{E}\left[ |\nabla_{y}F_{\mathcal {A}}
                (A)|^{q}+ |\nabla_{y}F_{\Xi}(A)|^{q}+ |\nabla_{y}F_{S(t)}(A)|^{q}\right
            ] <\infty. \label{eq:resampling-uniform-moments}
        \end{equation}

        \medskip

        \noindent
        \emph{Step 5: Stabilization.}\par

        For $x\in V$, let
        \[
            \psi_{\mathcal {A},x}(X)
            :=
            \mathbf{1}_{\{\eta_s(x)=1\text{ for some }s\ge0\}},
        \]
        \[
            \psi_{\Xi,x}(X) := \int_{0}^{\infty}\mathbf{1}_{\{\eta_s(x)=1\}}\,ds,
            \qquad \psi_{S(t),x}(X) := \mathbf{1}_{\{L_x>t\}}.
        \]
        Then, for every finite $A\subset V$ and every $\bullet\in\{\mathcal {A},\Xi ,S(t)\}$,
        \[
            \nabla_{y}F_{\bullet}(A) = \sum_{x\in A}\bigl( \psi_{\bullet,x}(X) -
            \psi_{\bullet,x}(X^{y}) \bigr).
        \]
        Every nonzero summand corresponds to a vertex of $\mathcal{R}_{y}$. Indeed, a difference in the range indicator
        must be created at an affected mark, whereas a difference in occupation time or last spreader time requires a
        disagreement in spreader status at some time.

        By \eqref{eq:discrepancy-polynomial-tails}, $R_{y}<\infty$ almost surely. Since $G$ is locally finite,
        $\mathcal{R}_{y}\subset B_{G}(y,R_{y})$ is finite almost surely.

        Let $(A_{n})_{n\ge1}\subset\mathfrak{A}$ tend locally to $\Gamma$. For fixed $y\in V$, almost surely there
        exists $n_{0}<\infty$ such that
        \[
            \mathcal{R}_{y}\subset A_{n}\qquad\text{for every }n\ge n_{0}.
        \]
        Consequently, for every $\bullet\in\{\mathcal {A},\Xi,S(t)\}$,
        \begin{equation}
            \nabla_{y}F_{\bullet}(A_{n}) \xrightarrow[n\to\infty]{\mathrm{a.s.}}\nabla
            _{y}F_{\bullet}(\infty), \label{eq:resampling-stabilization}
        \end{equation}
        where
        \[
            \nabla_{y}F_{\bullet}(\infty) := \sum_{x\in\mathcal{R}_{y}}\bigl( \psi
                _{\bullet,x}(X) - \psi_{\bullet,x}(X^{y}) \bigr).
        \]
        The uniform moment estimate \eqref{eq:resampling-uniform-moments} implies convergence in $L^{p}$ for every
        finite $p$.

        \medskip

        \noindent
        \emph{Step 6: Local conditional approximation and truncated resampling
            differences.}\par

        Use the notation $A^{+R}$ and $\mathcal{G}_{R}(A)$ from Subsection~\ref{subsec:abstract-clt}, and write
        \[
            F_{\bullet}^{(R)}(X,A) := \mathbb{E}\left[ F_{\bullet}(X,A)\mid\mathcal{G}
                _{R}(A) \right]
        \]
        for the corresponding local conditional approximation. Define $F_{\bullet}^{(R)}(X^{y},A)$ analogously using
        the local coordinates of $X^{y}$, and set
        \[
            H_{n,R}:= F_{\bullet}(X,B_{n}) - F_{\bullet}^{(R)}(X,B_{n}).
        \]

        Since \( \mathbb{E}[ H_{n,R}\mid\mathcal{G}_{R}(B_{n}) ] =0, \) we have
        \[
            \operatorname{Var}(H_{n,R}) = \mathbb{E}\left[ \operatorname{Var}\left
                ( H_{n,R}\mid\mathcal{G}_{R}(B_{n}) \right) \right].
        \]
        Conditionally on $\mathcal{G}_{R}(B_{n})$, the variable $F_{\bullet}^{(R)} (X,B_{n})$ is fixed, and only the
        exterior coordinates
        \[
            \{X_{y}:d_{G}(y,B_{n})>R\}
        \]
        remain random. Enumerate the exterior vertices as
        \[
            \{y:d_G(y,B_n)>R\}=\{y_1,y_2,\ldots\},
        \]
        and set
        \[
            \mathcal{K}_m:=\mathcal{G}_R(B_n)\vee
            \sigma(X_{y_1},\ldots,X_{y_m}).
        \]
        Conditional Efron--Stein, followed by Jensen's inequality, gives
        \[
            \mathbb{E}\!\left[
                \operatorname{Var}\!\left(
                \mathbb{E}[H_{n,R}\mid\mathcal{K}_m]
                \,\middle|\,\mathcal{G}_R(B_n)
                \right)
                \right]
            \le
            \frac{1}{2}\sum_{i=1}^{m}
            \mathbb{E}\left[
                |\nabla_{y_i}F_{\bullet}(B_n)|^{2}
                \right].
        \]
        Since
        \[
            \mathbb{E}[H_{n,R}\mid\mathcal{K}_m]
            \xrightarrow[m\to\infty]{L^2}H_{n,R},
        \]
        letting \(m\to\infty\) yields
        \begin{equation}
            \operatorname{Var}(H_{n,R})
            \le
            \frac{1}{2}\sum_{y:\,d_G(y,B_n)>R}
            \mathbb{E}\left[
                |\nabla_yF_{\bullet}(B_n)|^{2}
                \right].
            \label{eq:conditional-efron-stein}
        \end{equation}

        The summand vanishes on $\{R_{y}<d_{G}(y,B_{n})\},$ because then no vertex of $B_{n}$ belongs to
        $\mathcal{R}_{y}$. Using H\"older's inequality, \eqref{eq:resampling-uniform-moments}, and the moments of
        $R_{y}$, for every $N>0$ there exists $C_{N}<\infty$ such that
        \begin{equation}
            \mathbb{E}\left[ |\nabla_{y}F_{\bullet}(B_{n})|^{2}\right] \le C_{N}(
            1+d_{G}(y,B_{n}))^{-N}. \label{eq:far-resampling-decay}
        \end{equation}

        Since $\{y:d_{G}(y,B_{n})=k\} \subseteq B_{G}(o,n+k)\setminus B_{G}(o,n+k -1),$ the two-sided polynomial volume
        bounds give
        \[
            |B_{n}|^{-1}\#\{y:d_{G}(y,B_{n})=k\} \le C(1+k)^{D}.
        \]
        Thus, for $N>D+1$,
        \[
            \begin{aligned}
                 & |B_{n}|^{-1}\sum_{y:\,d_G(y,B_n)>R}(1+d_{G}(y,B_{n}))^{-N}   \\
                 & \qquad\le C\sum_{k>R}(1+k)^{D-N}\xrightarrow[R\to\infty]{}0.
            \end{aligned}
        \]
        Combining \eqref{eq:conditional-efron-stein} and \eqref{eq:far-resampling-decay}, we obtain
        \begin{equation}
            \lim_{R\to\infty}\limsup_{n\to\infty}|B_{n}|^{-1}\operatorname{Var}\left
            ( F_{\bullet}(X,B_{n}) - F_{\bullet}^{(R)}(X,B_{n}) \right) =0. \label{eq:exterior-approximation}
        \end{equation}

        Let $X_{y}'$ be the independent copy used to form $X^{y}$, and set
        \[
            \mathcal{H}_{R,y}(A) := \mathcal{G}_{R}(A)\vee\sigma(X_{y}').
        \]
        On the product space carrying $(X,X_{y}')$, suppose first that $y\in A^{+R}$. Then the local field used to
        compute $F_{\bullet}^{(R)}(X^{y},A)$ is obtained from the local field of $X$ by replacing $X_{y}$ with
        $X_{y}'$. Hence conditional expectation on $\mathcal{H}_{R,y}(A)$ gives
        \begin{equation}
            \nabla_{y}F_{\bullet}^{(R)}(A) = \mathbb{E}\left[ \nabla_{y}F_{\bullet}
                (A) \,\middle|\, \mathcal{H}_{R,y}(A) \right]. \label{eq:truncated-difference-identity}
        \end{equation}
        If $y\notin A^{+R}$, then
    $F_{\bullet}^{(R)}(X,A)=F_{\bullet}^{(R)}(X^{y},A)$, so the left-hand side
        is zero and the same estimates are immediate. Jensen's inequality and
        \eqref{eq:resampling-uniform-moments} give the required uniform moment bounds
        for the truncated differences.

        Let $(A_{n})_{n\ge1}\subset\mathfrak{A}$ tend locally to $\Gamma$. Then $y\in A_{n}^{+R}$ for all sufficiently
        large $n$. By \eqref{eq:truncated-difference-identity} and the $L^{2}$-contraction of conditional expectation,

        \begin{equation}
            \begin{aligned}
                 & \left\| \nabla_{y}F_{\bullet}^{(R)}(A_{n}) - \nabla_{y}F_{\bullet}(\infty) \right\|_{2}                                                                         \\
                 & \qquad\le \left\| \nabla_{y}F_{\bullet}(A_{n}) - \nabla_{y}F_{\bullet}(\infty) \right\|_{2}                                                                     \\
                 & \qquad\quad+ \left\| \mathbb{E}\left[ \nabla_{y}F_{\bullet}(\infty) \,\middle|\, \mathcal{H}_{R,y}(A_{n}) \right] - \nabla_{y}F_{\bullet}(\infty) \right\|_{2}.
            \end{aligned}
            \label{eq:truncated-difference-comparison}
        \end{equation}
        The first term tends to zero by \eqref{eq:resampling-stabilization} and the
        uniform moment bound.

        To control the second term, approximate $\nabla_{y}F_{\bullet}(\infty)$ in $L^{2}$ by a cylinder random
        variable $\chi$ depending on finitely many coordinates of $X$ and on $X_{y}'$. Since $A_{n}$ tends locally to
        $\Gamma$, the corresponding finite coordinate set is contained in $A_{n}^{+R}$ for all sufficiently large $n$.
        Thus $\chi$ is $\mathcal{H}_{R,y}(A_{n})$-measurable, and
        \[
            \begin{aligned}
                 & \left\| \mathbb{E}\left[ \nabla_{y}F_{\bullet}(\infty) \,\middle|\, \mathcal{H}_{R,y}(A_{n}) \right] - \nabla_{y}F_{\bullet}(\infty) \right\|_{2} \\
                 & \qquad\le 2 \left\| \nabla_{y}F_{\bullet}(\infty)-\chi \right\|_{2}.
            \end{aligned}
        \]
        The right-hand side can be made arbitrarily small. Hence, for every fixed $R\ge 0$,
        \begin{equation}
            \nabla_{y}F_{\bullet}^{(R)}(A_{n}) \xrightarrow[n\to\infty]{L^2}\nabla
            _{y}F_{\bullet}(\infty). \label{eq:truncated-difference-stabilization}
        \end{equation}
        Therefore the truncated resampling differences stabilize, and we may
        take
        \[
            \nabla_{y}F_{\bullet}^{(R)}(\infty) := \nabla_{y}F_{\bullet}(\infty)
        \]
        for every $R\ge0$. In particular,
    \begin{equation}
        \nabla_{y}F_{\bullet}^{(R)}(\infty) \xrightarrow[R\to\infty]{L^2}\nabla
        _{y}F_{\bullet}(\infty). \label{eq:truncated-limit-convergence}
    \end{equation}

    Equation~\eqref{eq:mt-equivariance} gives stationarity. The uniform moment condition follows from
    \eqref{eq:resampling-uniform-moments}, stabilization from \eqref{eq:resampling-stabilization}, and local
    conditional approximation from \eqref{eq:exterior-approximation}. Finally,
    \eqref{eq:truncated-difference-stabilization} and \eqref{eq:truncated-limit-convergence} verify the assumptions on
    the truncated differences. Thus all hypotheses of Theorem~\ref{thm:abstract-clt-functional} are satisfied.
\end{proof}

\begin{proof}[Proofs of Theorems~\ref{thm:clt-final-stifler-proportion-balls},
        \ref{thm:clt-spreader-occupation-balls}, and
        \ref{thm:functional-clt-empirical-survival-function}]
    The hypotheses verified in
    Proposition~\ref{prop:verification-mt-observables} are preserved under
    finite linear combinations. Hence
    Theorem~\ref{thm:abstract-clt-functional}, together with polarization,
    shows that the limit defining \(K_S(s,t)\) exists and admits the
    representation
    \begin{equation}
        K_S(s,t)
        =
        \frac{1}{|Y|}
        \sum_{y\in Y}
        \mathbb E\bigl[d_y(s)d_y(t)\bigr],
        \qquad
        d_y(t)
        :=
        \mathbb E\left[
            \nabla_yF_{S(t)}(\infty)
            \,\middle|\,
            \mathcal F_y
            \right].
        \label{eq:survival-covariance}
    \end{equation}

    By Proposition~\ref{prop:verification-mt-observables} and Theorem~\ref{thm:abstract-clt-functional}, for every
    fixed $t\ge0$ and every $\bullet\in\{\mathcal {A},\Xi,S(t)\}$, there exists $\sigma_{\bullet}^{2}\ge0$ such that
    \[
        |B_{n}|^{-1/2}\left( F_{\bullet}(X,B_{n}) - \mathbb{E}[F_{\bullet}(X,
            B_{n})] \right) \xrightarrow[n\to\infty]{d}\mathcal{N}(0,\sigma_{\bullet}
        ^{2}).
    \]
    Setting $\sigma_{\kappa}^{2}:=\sigma_{\mathcal {A}}^{2}$ and recalling that
    \[
        \mathbb{E}[\kappa_{n}]=\kappa, \qquad \mathbb{E}[\Xi_{n}]=|B_{n}|\xi,
        \qquad \mathbb{E}[S_{n}(t)]=|B_{n}|s(t),
    \]
    we obtain
    \begin{align*}
        |B_{n}|^{1/2}\bigl(\kappa_{n}-\kappa\bigr)   & \xrightarrow[n\to\infty]{d}\mathcal{N}(0,\sigma_{\kappa}^{2}), \\
        |B_{n}|^{-1/2}\bigl(\Xi_{n}-|B_{n}|\xi\bigr) & \xrightarrow[n\to\infty]{d}\mathcal{N}(0,\sigma_{\Xi}^{2}),    \\
        |B_{n}|^{1/2}\bigl(s_{n}(t)-s(t)\bigr)       & \xrightarrow[n\to\infty]{d}\mathcal{N}(0,\sigma_{S}^{2}(t)).
    \end{align*}

    \noindent
    \textit{Functional central limit theorem for the empirical survival
        function.}\par Fix $t_{1},\ldots,t_{k}\in[0,T]$ and $a_{1},\ldots,a_{k}\in\mathbb{R}$, and set
    \[
        F_{\mathbf{a},\mathbf{t}}(X,A) := \sum_{i=1}^{k}a_{i}F_{S(t_i)}(X,A).
    \]
    The hypotheses of Theorem~\ref{thm:abstract-clt-functional} remain valid under finite linear combinations.
    Moreover, $\nabla_{y}F_{\mathbf{a},\mathbf{t}} (\infty ) = \sum_{i}a_{i}\nabla_{y}F_{S(t_i)}(\infty)$. Hence
    \[
        \sum_{i=1}^{k}a_{i}\mathbb{S}_{n}(t_{i}) \xrightarrow[n\to\infty]{d}\mathcal{N}
        \left( 0, \sum_{i,j=1}^{k}a_{i}a_{j}K_{S}(t_{i},t_{j}) \right).
    \]
    The Cramér--Wold theorem gives convergence of the finite-dimensional distributions.

    For tightness, let
    \[
        N_{n}(t) := \sum_{x\in B_n}\mathbf{1}_{\{0<L_x\le t\}}= \sum_{x\in B_n}
        \mathbf{1}_{\{\eta_t(x)=2\}}.
    \]
    Since there are no initial stiflers and each vertex becomes a spreader at most once, $S_{n}(t)=S_{n}(0)-N_{n}(t)$.

    Let $(\mathcal{T}_{t})_{t\ge0}$ be the filtration generated by the initial configuration and the graphical marks up
    to time $t$. The predictable intensity of $N_{n}$ is
    \[
        \lambda_{n}(t) := \alpha \sum_{x\in B_n}\mathbf{1}_{\{\eta_{t-}(x)=1\}}
        \sum_{z\sim x}\mathbf{1}_{\{\eta_{t-}(z)\in\{1,2\}\}}.
    \]
    Thus $M_{n}(t):=N_{n}(t)-\int_{0}^{t}\lambda_{n}(u)\,du$ is a square-integrable $(\mathcal{T}_{t})$-martingale. Set
    \[
        \mathbb{M}_{n}(t):=|B_{n}|^{-1/2}M_{n}(t), \qquad \mathbb{A}_{n}(t) :
        = |B_{n}|^{-1/2}\int_{0}^{t}\bigl( \lambda_{n}(u)-\mathbb{E}[\lambda_{n}
            (u)] \bigr)\, du.
    \]
    Then
    \[
        |B_{n}|^{-1/2}\bigl( N_{n}(t)-\mathbb{E}[N_{n}(t)] \bigr) = \mathbb{M}
        _{n}(t) +\mathbb{A}_{n}(t).
    \]

    Distinct graphical marks occur at distinct times almost surely, and each stifling mark changes the state of at most
    one vertex. Hence $N_{n}$ has jumps of size one and
    \[
        \sup_{t\le T}
        \bigl|\mathbb{M}_{n}(t)-\mathbb{M}_{n}(t-)\bigr|
        \le |B_{n}|^{-1/2},
        \qquad
        \langle\mathbb{M}_{n}\rangle_{t}
        \le \alpha\Delta t.
    \]
    Moreover, for $0\le s\le t\le T$,
    \[
        0 \le \langle\mathbb{M}_{n}\rangle_{t}- \langle\mathbb{M}_{n}\rangle_{s}
        = \frac{1}{|B_{n}|}\int_{s}^{t}\lambda_{n}(u)\,du \le \alpha\Delta(t-
        s).
    \]
    Hence $(\langle\mathbb{M}_{n}\rangle)_{n\ge1}$ is uniformly Lipschitz and therefore $C$-tight. The martingale
    tightness criterion \cite[VI.4.13]{JacodShiryaev} gives tightness of $(\mathbb{M}_{n})$ in $D([0,T])$. Since the
    maximal jump size converges to zero, the sequence is $C$-tight by \cite[VI.3.26]{JacodShiryaev}.

    Let $\lambda_{n}^{y}(u)$ denote the corresponding intensity for the resampled process $\eta^{y}$. Any state
    disagreement away from $y$, including one between states $0$ and $2$, must be created at an affected mark and
    therefore lies in $\mathcal{R}_{y}$. Hence $\lambda_{n}(u)=\lambda_{n}^{y}(u)$ whenever $d_{G}(y,B_{n})>R_{y}+1$.
    Otherwise, since $G$ has polynomial volume growth of degree $D$, the number of vertices that can contribute to the
    intensity difference is bounded by $C(1+R_{y})^{D}$. Thus
    \[
        |\lambda_{n}(u)-\lambda_{n}^{y}(u)| \le C(1+R_{y})^{D}\mathbf{1}_{\{d_G(y,B_n)\le
            R_y+1\}},
    \]
    where $C<\infty$ is independent of $n$, $u$, and $y$. The Efron--Stein argument used in
    \eqref{eq:conditional-efron-stein}--\eqref{eq:far-resampling-decay}, together with
    \eqref{eq:discrepancy-all-moments} and polynomial volume growth, gives
    \[
        \sup_{n\ge1}\sup_{0\le u\le T}\frac{\operatorname{Var}(\lambda_{n}(u))}{|B_{n}|}
        <\infty.
    \]
    Therefore, for $0\le s<t\le T$,
    \[
        \mathbb{E}\left[ |\mathbb{A}_{n}(t)-\mathbb{A}_{n}(s)|^{2}\right] \le
        C(t-s)^{2}.
    \]
    Since $\mathbb{A}_{n}(0)=0$ and $\mathbb{A}_{n}$ has continuous paths, the Kolmogorov tightness criterion shows
    that $(\mathbb{A}_{n})$ is tight in $C([0,T])$.

    It follows that
    \[
        \left( |B_{n}|^{-1/2}\bigl( N_{n}-\mathbb{E}[N_{n}] \bigr) \right)_{n\ge1}
    \]
    is $C$-tight. Since $\mathbb{S}_{n}(0)$ is tight by the one-dimensional convergence already established,
    \[
        \mathbb{S}_{n}(t) = \mathbb{S}_{n}(0) - |B_{n}|^{-1/2}\bigl( N_{n}(t)
        -\mathbb{E}[N_{n}(t)] \bigr)
    \]
    is $C$-tight in $D([0,T])$. Every subsequential limit is supported on $C( [0,T])$ and has the Gaussian
    finite-dimensional distributions identified above. These distributions determine a unique limit, so the whole
    sequence $(\mathbb{S}_{n})_{n\ge1}$ converges in $D([0,T])$.
\end{proof}

\begin{proof}[Proof of the nondegeneracy assertions]

    All the preceding verifications are uniform in \(y\in V\) and \(A\in\mathfrak A\), and therefore remain valid for
    every admissible choice of \(H\) and \(Y\) in Theorem~\ref{thm:abstract-clt-functional}. The proof of that theorem
    applies to any normal, finite-index, left-orderable subgroup of $\Gamma$, and the limiting variances are intrinsic
    because they are limits of the normalized variances $|B_{n} |^{-1}\operatorname{Var}(F_{\bullet}(X,B_{n}))$. Since
    $\Gamma$ is finitely generated and virtually nilpotent, it is residually finite. For each $g\in
    B_G(o,2)\setminus\{e\}$, choose a finite-index normal subgroup $N_g\lhd\Gamma$ such that $g\notin N_g$. Replacing
    $H$ by
    \[
        H\cap\bigcap_{g\in B_G(o,2)\setminus\{e\}}N_g,
    \]
    which remains normal, finite-index, and left-orderable, we may assume that \( H\cap B_G(o,2)=\{e\}. \)

    Let $y$ be the representative of the last right coset. If \(z=yg\in B_G(y,2)\cap Hy\) with \(z\ne y\), then \(g\in
    B_G(o,2)\) and \(ygy^{-1}\in H\), normality gives \(g\in H\cap B_G(o,2)=\{e\}\), a contradiction. Hence
    $B_{G}(y,2)\subset\mathfrak{F} _{y}$.

    For each $z\sim y$, choose $w_{z}\sim z$, $w_{z}\ne y$, and let $\mathcal{E}_{y}$ be the event that all $z$ and
    $w_{z}$ are initially spreaders and the first mark with source $z$ is the stifling mark from $z$ to $w_{z}$. The
    event $\mathcal{E}_{y}$ is $\mathcal{F}_{y}$-measurable and has positive probability. Set
    \[
        I_y:=\mathbf{1}_{\{\eta_0(y)=1\}}, \qquad
        T_y:=\min_{z\sim y}\inf\{s>0:N_{\mathrm{stif}}^{y,z}(s)\ge1\},
        \qquad
        \mu:=\alpha\Delta.
    \]
    Moreover, $\mathcal E_y$ is independent of $X_y$, and hence of $(I_y,T_y)$. On $\mathcal{E}_{y}$, each neighbor $z$
    of $y$ remains a spreader until its first mark with source $z$. At that mark, $z$ becomes a stifler because $w_z$
    is non-ignorant, and $z$ has produced no previous source mark. Incoming marks cannot alter $z$: infection marks are
    ineffective at a non-ignorant target, while a stifling mark may alter only its source. Thus every neighbor of \(y\)
    has the same trajectory in the two coupled processes, and no vertex other than \(y\) is affected by the resampling
    of \(X_y\).

    Then $I_{y}\sim\operatorname{Bernoulli}(\beta)$, $T_y\sim\operatorname{Exp}(\mu)$, and they are independent. Hence,
    on $\mathcal{E}_{y}$,
    \[
        d_y^{\mathcal A}=I_{y}-\beta,\qquad d_{y}^{\Xi}=I_{y}T_{y}-\frac{\beta}{\mu}
        ,\qquad d_{y}^{S(t)}= d_{y}(t) = I_{y}\mathbf{1}_{\{T_y>t\}}-\beta e^{-\mu
                t},
    \]
    where
    \[
        d_{y}^{\bullet}:= \mathbb{E}\!\left[ \nabla_{y}F_{\bullet}(\infty)\mid
            \mathcal{F}_{y}\right].
    \]
    Writing $q=|Y|$ and $p_{t}=\beta e^{-\mu t}$, the variance representation in
    Theorem~\ref{thm:abstract-clt-functional} gives
    \[
        \sigma_{\kappa}^{2}\ge \frac{\mathbb{P}(\mathcal{E}_{y})}{q}\,\beta(
        1-\beta), \qquad \sigma_{\Xi}^{2}\ge \frac{\mathbb{P}(\mathcal{E}_{y})}{q}
        \, \frac{\beta(2-\beta)}{\mu^{2}},
    \]
    and
    \[
        \sigma_{S}^{2}(t) \ge \frac{\mathbb{P}(\mathcal{E}_{y})}{q}\,p_{t}(1-
        p_{t}).
    \]
    Since $\beta\in(0,1)$, all three lower bounds are strictly positive. This proves nondegeneracy.
\end{proof}

    This study was financed, in part, by FAPESP, Brazil, under grants 2024/21482-8,
    2025/02013-0, and 2023/13453-5, the latter corresponding to the Thematic
    Project ``Modelagem de sistemas estocásticos.''


\bibliographystyle{amsplain}
\bibliography{referencias}
\end{document}